\documentclass[11pt]{amsart}
\usepackage[colorlinks=true, urlcolor=blue]{hyperref}

\usepackage{amsaddr}
\usepackage{amssymb}
\usepackage{amsmath}
\usepackage{caption}
\usepackage{graphicx}
\usepackage{geometry} 
\usepackage{subcaption}
\usepackage{verbatim} 
\usepackage{multirow}
\usepackage{caption}
\usepackage[normalem]{ulem} 

\usepackage[table]{xcolor} 

\numberwithin{equation}{section}

\newcommand{\mN}{\mathbb{N}}
\newcommand{\mR}{\mathbb{R}}

\newcommand{\vs}{\vct{s}}
\newcommand{\vx}{\vct{x}}
\newcommand{\vy}{\vct{y}}

\newcommand{\vct}[1]{\mathbf{#1}}

\newcommand{\supp}{\mathrm{supp}\,}

\newcounter{cntD}

\newtheorem{assmD}[cntD]{Assumption}

\newtheorem{thm}{Theorem}[section]
\newtheorem{lem}[thm]{Lemma}

\newtheorem{rem}[thm]{Remark}

\newtheorem{prop}[thm]{Proposition}
\newtheorem{assm}[thm]{Assumption}

\definecolor{ilariablue}{rgb}{0 0.4 0.85}

\definecolor{sjoerdgreen}{rgb}{0 0.7 0.2}

\title[LAP for variable-coefficient wave equation ]{On the limiting amplitude principle for the wave equation with variable coefficients} \author{Anton~Arnold$^{1}$, Sjoerd~Geevers$^{2}$, Ilaria~Perugia$^{2}$, Dmitry~Ponomarev$^{1,3,4}$}
\address{{\small
$^1$ Institute of Analysis and Scientific Computing, Vienna
University of Technology\\
Wiedner Hauptstrasse 8-10, 1040 Vienna, Austria\\
$^2$ Faculty of Mathematics, University of Vienna\\
Oskar-Morgenstern-Platz 1, 1090 Vienna, Austria\\
$^3$ Factas research team, Centre Inria d'Universit{\'e} C{\^o}te d'Azur\\
Route des Lucioles 2004, BP 93, 06902 Sophia Antipolis, France\\
$^4$ St. Petersburg Department of V. A. Steklov Mathematical Institute, RAS,\\
Fontanka 27, 191023 St. Petersburg, Russia
}}

\begin{document}


\begin{abstract}
In this paper, we prove new results on the validity of the
  limiting amplitude principle (LAP) for the wave equation with nonconstant
  coefficients, not necessarily in divergence form. Under suitable assumptions on the coefficients and on
  the source term, we establish the LAP for space dimensions 2 and 3.
  This result is extended to one space dimension with an
  appropriate modification.  We also quantify the LAP and thus provide estimates for the convergence of the time-domain solution to the frequency-domain solution. 
Our proofs are based on time-decay results of solutions of some
auxiliary problems. 
\end{abstract}

\maketitle

\medskip

{\footnotesize
}
\medskip

\section{Introduction}
\label{sec:intro}

An essential ingredient in connecting time- and frequency-domain wave problems is the limiting amplitude principle (LAP). Originally proposed as one of the tools to select the unique solution of the Helmholtz equation problem in an infinite domain, it has been studied in numerous works over the last 70 years.

The LAP can be crudely stated as follows: The solution to the time-dependent wave equation with time-harmonic source term converges, for large times, to the solution of the Helmholtz equation with the spatial source term and frequency corresponding to the original time-harmonic source.   
Our main motivation for revisiting the LAP comes from numerical analysis.
Helmholtz problems
can be challenging to solve in practice for large wavenumbers. Numerical methods have 
been proposed to address a classical Helmholtz problem efficiently through its reformulations in the time domain. They include the controllability
method introduced in~\cite{bristeau1998, glowinski2006}, together with its spectral version~\cite{heikkola2007a} and
its extensions~\cite{grote2019,grote2020}, the WaveHoltz
method~\cite{appelo2020}, the time-domain preconditioner 
of~\cite{stolk2020}, and the front-tracking adaptive method of~\cite{paper-AFEM}.
{Numerical results presented in the latter paper, for instance, have shown that solving the Helmholtz equation in the time domain can be advantageous for high frequencies, when computations are essentially reduced to the neighborhood of a lower-dimensional manifold (wave-front area).}

The analysis of these methods requires a quantification of the modeling error (reformulation of the frequency-domain problem into a time-domain problem), which will add to the error due to the numerical approximation of the problem in the time domain.
This motivates the study of the LAP under some new angles, with
particular focus on the quantification of large-time convergence.

{As opposed to a direct study of the resolvent operator, our
  analysis is based on decay estimates for the solutions of some
  auxiliary PDE problems. Since decay results are still the subject of
intense investigation, an advantage of this approach is that new findings in that area directly translate into improvements in the quantification of the large-time convergence in the LAP.}\medskip

{\bf Main results.} We consider the
following setup. Given an angular frequency $\omega>0$, material
parameters $\alpha$, $\beta$, which smoothly vary within some bounded
domain, and a compactly supported source term $F$, we consider the
following frequency-domain and time-domain problems, respectively:
\begin{equation}
\begin{cases}
-\omega^{2}U\left(\mathbf{x}\right)-\beta^{-1}\left(\mathbf{x}\right)\nabla\cdot\left(\alpha\left(\mathbf{x}\right)\nabla U\left(\mathbf{x}\right)\right)=F\left(\mathbf{x}\right), & \mathbf{x}\in\mathbb{R}^{d},\\
\underset{\left|\mathbf{x}\right|\rightarrow\infty}{\lim}\left|\mathbf{x}\right|^\frac{d-1}{2}\left[\partial_{\left|\mathbf{x}\right|}U\left(\mathbf{x}\right)-i\omega\sqrt{\beta_{0}/\alpha_{0}}U\left(\mathbf{x}\right)\right]=0,
\end{cases}\label{eq:U_Helm}
\end{equation}
and 
\begin{equation}
\begin{cases}
\partial_{t}^{2}u\left(\mathbf{x},t\right)-\beta^{-1}\left(\mathbf{x}\right)\nabla\cdot\left(\alpha\left(\mathbf{x}\right)\nabla u\left(\mathbf{x},t\right)\right)=e^{-i\omega t}F\left(\mathbf{x}\right), & \mathbf{x}\in\mathbb{R}^{d},\:t>0,\\
u\left(\mathbf{x},0\right)=0,\hspace{1em}\partial_{t}u\left(\mathbf{x},0\right)=0, & \mathbf{x}\in\mathbb{R}^{d}.
\end{cases}\label{eq:wave_LAP}
\end{equation}
Our assumptions on $\alpha$, $\beta$, and $F$ are stated as follows.

\begin{assm}\label{assm:alph_bet_base}
{\bf (smoothness, compactly supported derivatives \& positivity of coefficients)}
Assume {$d\geq 2$, and let} $\alpha$, $\beta\in C^\infty\left(\mR^d\right)$ 
be real-valued functions such that $\alpha(\vx)\geq\alpha_{\min}$, $\beta(\vx)\geq\beta_{\min}$ for $\vx\in\mR^d$,  and $\alpha(\vx)\equiv\alpha_0$, $\beta(\vx)\equiv\beta_0$
for $\vx\in\mR^d\backslash\Omega_{in}$, with some bounded domain $\Omega_{in}\subset\mR^d$ and constants $\alpha_{\min}$, $\beta_{\min}$, $\alpha_0$, $\beta_0>0$.
\end{assm}
\begin{assmD}\label{assm:alph_bet_base_1D}
{\bf (regularity, compactly supported derivatives \& positivity of coefficients; $1D$ case)}
Assume $d=1$, and let $\alpha$, $\beta\in W^{1,\infty}\left(\mR\right)$
be real-valued functions such that $\alpha(x)\geq\alpha_{\min}$, $\beta(x)\geq\beta_{\min}$ for $x\in\mR$,  and $\alpha(x)\equiv\alpha_0$, $\beta(x)\equiv\beta_0$
for $x\in\mR\backslash\Omega_{in}$, with some  
open bounded interval 
$\Omega_{in}\subset\mR$ and constants $\alpha_{\min}$, $\beta_{\min}$, $\alpha_0$, $\beta_0>0$.
\end{assmD}

\begin{assm}\label{assm:alph_bet_trap}
{\bf (nontrapping coefficients)}
Let $\alpha$, $\beta$ be non-trapping, i.e. such that all rays
associated with the metric $\alpha/\beta$ escape to infinity
\cite[Sect. 1]{Boucl-Burq}. In other words (see
  e.g. \cite[Def. 7.6 \& Cor. 7.10]{Grah-Pemb-Spen}), 
  {defining $H\left(\mathbf{q},\mathbf{p}\right):=\alpha\left(\mathbf{q}\right)\left|\mathbf{p}\right|^2-\beta\left(\mathbf{q}\right)$,
given
$\mathbf{q}_0$, $\mathbf{p_0}\in\mathbb{R}^d$ such that $H\left(\mathbf{q_0},\mathbf{p_0}\right)=0$,
the solution vector of the canonical system of differential equations with the Hamiltonian $H\left(\mathbf{q},\mathbf{p}\right)$, 
\begin{equation}\label{eq:rays}
\begin{cases}
\frac{d}{d t}\mathbf{q}\left(t\right)=2\alpha\left(\mathbf{q}\right)\mathbf{p}\left(t\right),&\hspace{1em}t>0,\\
\frac{d}{d t}\mathbf{p}\left(t\right)=-\frac{\beta\left(\mathbf{q}\right)}{\alpha\left(\mathbf{q}\right)}\nabla_\mathbf{q}\alpha\left(\mathbf{q}\right)+\nabla_\mathbf{q}\beta\left(\mathbf{q}\right),&\hspace{1em}t>0,\\
\mathbf{q}\left(0\right)=\mathbf{q}_0,\hspace{1em}\mathbf{p}\left(0\right)=\mathbf{p}_0,
\end{cases}
\end{equation}
must satisfy $\left|\mathbf{q}\left(t\right)\right|\rightarrow\infty$ as $t\rightarrow\infty$.}
\end{assm}

\begin{assm}\label{assm:F_base}
{\bf (compactly {supported} source)}
Let the complex-valued function $F\in L^2\left(\mathbb{R}^d\right)$ be such that $\supp F \subset \Omega_{in}$, with $\Omega_{in}$ as in Assumption \ref{assm:alph_bet_base} or \ref{assm:alph_bet_base_1D}.
\end{assm}

Under the above assumptions, 
we prove the following versions of the LAP.
\begin{thm}
\label{thm:LAP}
Let {$d=2, 3$}.
Suppose that Assumptions \ref{assm:alph_bet_base}--\ref{assm:F_base} are satisfied. Let
$U\left(\mathbf{x}\right)$ and $u\left(\mathbf{x},t\right)$ be 
solutions to \eqref{eq:U_Helm} and \eqref{eq:wave_LAP},
respectively. Then, there exists a constant $C>0$ 
depending on
$F$, $\alpha$, $\beta$, $\omega$, and $\Omega$
such that\\
for $d=2$:
\begin{equation*}
\left\Vert u\left(\cdot,t\right)-e^{-i\omega t}U\right\Vert_{H^1\left(\Omega\right)}
+\left\Vert \partial_t u\left(\cdot,t\right)+i\omega e^{-i\omega t}U\right\Vert_{L^2\left(\Omega\right)}\leq C\,\frac{1+\log\left(1+t^2\right)}{\left(1+t^2\right)^{1/2}},\quad t\geq 0;
\end{equation*}
for $d=3$:
\begin{equation*}
\left\Vert u\left(\cdot,t\right)-e^{-i\omega t}U\right\Vert_{H^1\left(\Omega\right)}
+\left\Vert \partial_t u\left(\cdot,t\right)+i\omega e^{-i\omega t}U\right\Vert_{L^2\left(\Omega\right)}\leq \frac{C}{\left(1+t^2\right)^{1/2}},\quad t\geq 0,
\end{equation*}
where $\Omega\subset\mR^d$ is an arbitrary bounded domain.
\end{thm}
 
\begin{thm}\label{thm:LAP_1D}
Let $d=1$. Suppose that Assumptions {\ref{assm:alph_bet_base_1D}} and
\ref{assm:F_base} are satisfied. Let $U\left(\mathbf{x}\right)$
and $u\left(\mathbf{x},t\right)$ be
solutions to \eqref{eq:U_Helm} and \eqref{eq:wave_LAP}, respectively. Then,
there exist constants $\Lambda>0$ {(depending on
 $\alpha$, $\beta$)} and $C>0$ (depending on
$F$, $\alpha$, $\beta$, $\omega$, and $\Omega$) such that
\begin{equation*}
\left\Vert u\left(\cdot,t\right)-e^{-i\omega t}U-{U_\infty}\right\Vert_{H^1\left(\Omega\right)}+\left\Vert \partial_t u\left(\cdot,t\right)+i\omega e^{-i\omega t}U\right\Vert_{L^2\left(\Omega\right)}\leq C e^{-\Lambda t},\hspace{1em}t\geq0,
\end{equation*}
where
	{
	\begin{equation}\label{eq:U_infty_expl}
	U_\infty:=\frac{1}{2i\omega\sqrt{\alpha_0\beta_0}}\int_{\Omega_{in}}F\left(x\right) \beta\left(x\right)dx,
	\end{equation}
and $\Omega\subset\mR^d$ is an arbitrary bounded domain.
}
\end{thm}

\begin{rem}
Note that, in contrast with Theorem \ref{thm:LAP}, Theorem \ref{thm:LAP_1D} does not require Assumption  \ref{assm:alph_bet_trap}. This is because in the one-dimensional setting, the rays can only be associated with left- and right-propagating waves, and the trapping cannot occur for the regular coefficients.
\end{rem}

These results are summarized in Table~\ref{table:summary}.


\begin{table}[ht]
  \begin{tabular}{|c|c|c|c|}
    \hline
    & $d=1$ & $d=2$ & $d=3$ \\
    \hline
    assumptions & 1.1', 1.3 & 1.1, 1.2, 1.3& 1.1, 1.2, 1.3\\
    \hline
    $u^{\text{\sc diff}}(\vx,t)$ & $u^{\text{\sc
                                   w}}(\vx,t)-u^{\text{\sc
                                   h}}(\vx,t)-U_\infty$& $u^{\text{\sc
                                                         w}}(\vx,t)-u^{\text{\sc
                                                         h}}(\vx,t)$&
                                                                      $u^{\text{\sc w}}(\vx,t)-u^{\text{\sc h}}(\vx,t)$\\
    \hline
    bound of & \multirow{2}{*}{$C e^{-\Lambda t}$}&
                       \multirow{2}{*}{$C\,\frac{1+\log\left(1+t^2\right)}{\left(1+t^2\right)^{1/2}}$}&
                       \multirow{2}{*}{$C\,\frac{1}{\left(1+t^2\right)^{1/2}}$}\\
$\left\Vert u^{\text{\sc
    diff}}\left(\cdot,t\right)\right\Vert_\star$ & & & \\   
\hline    
statement & {Thm.~\ref{thm:LAP_1D}}& {Thm.~\ref{thm:LAP}}&
                                                           {Thm.~\ref{thm:LAP}}\\
\hline
proof & Sect.~\ref{sec:LAP_proof}    &Sect.~\ref{sec:LAP_proof}&Sect.~\ref{sec:LAP_proof}\\
    \hline
time-decay & {Prop.~\ref{prop:1D_decay}}&
                                          {Prop.~\ref{prop:dec_free}, Prop.~\ref{prop:dec_source}}&
                                             {Prop.~\ref{prop:dec_free}, Prop.~\ref{prop:dec_source}}\\
results used & (see~\cite[Thm~1.4]{paper-decay1D})& {Lem.~\ref{lem:gen_dec},
                 Lem.~\ref{lem:spec_dec}}&{Lem.~\ref{lem:gen_dec},
                                           Lem.~\ref{lem:spec_dec}}\\
    &&(proofs: Sect.~\ref{sec:proofs}) & (proofs: Sect.~\ref{sec:proofs})\\
   \hline 
  \end{tabular}  
  \caption{\small{Summary of our results. Here,
    $u^{\text{\sc w}}(\vx,t)$ is the solution to the
    wave problem~\eqref{eq:wave_LAP}, $u^{\text{\sc
        h}}(\vx,t):=e^{-i\omega t}U(\vx)$, with $U$ solution to the
    Helmholtz problem~\eqref{eq:U_Helm}, $U_\infty$ is the
    constant in~\eqref{eq:U_infty_expl}, and $\left\Vert
      u^{\text{\sc diff}}\left(\cdot,t\right)\right\Vert_\star:=\left\Vert u^{\text{\sc diff}}\left(\cdot,t\right)\right\Vert_{H^1\left(\Omega\right)}
+\left\Vert \partial_t u^{\text{\sc diff}}\left(\cdot,t\right)\right\Vert_{L^2\left(\Omega\right)}$.}}
\end{table}\label{table:summary}


\medskip
{\bf Previous results on the LAP.} Let us provide a brief overview of previous works on the LAP.
The simplest version of the LAP dealing with the constant-coefficient,
three-dimensional wave equation has been known at least since 1948
\cite{Tikh-Sam1, Tikh-Sam2}. There, it is proven that this physical principle selects the unique solution of the stationary problem satisfying the Sommerfeld radiation condition.

Starting from the seminal work by Ladyzhenskaya \cite{Ladyzh}, variable-coefficient equations 
of the form $\partial_{t}^{2}u\left(\mathbf{x},t\right)-c^2\left(\mathbf{x}\right)\Delta
u\left(\mathbf{x},t\right)+q\left(\mathbf{x}\right)u\left(\mathbf{x},t\right)=f\left(\mathbf{x}\right)e^{-i\omega t}$ have been considered. Namely, while \cite{Ladyzh,Odeh1} treat the case $c\equiv 0$,   
paper \cite{odeh1961} deals with the case $q\equiv0$. When $q$, $\nabla c$ and $f$ are sufficiently localised, the validity of the LAP is proven in a pointwise sense but a rate of the convergence is not specified.

For the case $c\equiv 0$ and dimension $d=3$, Ramm \cite{Ramm1} establishes an algebraic pointwise convergence and shows that the convergence rate is directly related to the localisation of $q$ and $f$.

Eidus' paper \cite{Eidus} provides an extensive overview of the
results available at the time and treats the problem in great
generality. In particular, it deals with the wave equation arising
from a positive second-order differential operator in divergence form \linebreak 
$-\sum_{k,j=1}^{d}\partial_{x_{k}}\left(a_{kj}\left(\mathbf{x}\right)\partial_{x_{j}}\right)+q\left(\mathbf{x}\right)$. It
is assumed that $q$ is real-valued and locally H{\"o}lder
continuous, each $a_{kj}\in C^2\left(\mathbb{R}^d\right)$ is real-valued, $a_{kj}=a_{jk}$, $1\leq k,j \leq d$, and for any vector  $\mathbf{v}\in
\mathbb{R}^d$, $\sum_{k,j}^d a_{kj}v_k v_j\geq a_0
{\left|\mathbf{v}\right|^2}$ with some $a_0>0$; moreover, it is assumed that $\left|\nabla
  a_{kj}\right|$ and
$q$ decay fast enough at infinity. The
problem is posed in an unbounded domain of $\mathbb{R}^d$ with a
finite boundary where the zero Dirichlet boundary condition is
imposed. However, it is mentioned in \cite[Ch. 2, p. 21]{Eidus} that
the obtained results must also hold if this unbounded (exterior)
domain is taken to be the whole $\mathbb{R}^d$. The time convergence
is proven in $H^1$-norm of the solution and in the $L^2$-norm of
its time derivative, with both norms taken over bounded sets.  

As a generalization, Vainberg \cite{Vainb}, besides geometrical features, also considers higher-order constant coefficient hypoelliptic operators in $\mathbb{R}^d$, whereas Iwasaki \cite{Iwasaki} treats dissipative wave equations with variable dissipation and potential terms.

Ramm \cite{Ramm2} considers a general linear operator and formulates
necessary and sufficient conditions for the validity of the LAP in
terms of certain properties of the resolvent operator.
A more general form of the LAP is formulated, involving time convergence in mean, namely, the convergence of the quantity $\frac{1}{t}\int_0^t e^{i\omega \tau}u\left(\mathbf{x},\tau\right)d\tau$, for $t\rightarrow\infty$, to the stationary solution{. This} is shown to be equivalent to the validity of the limiting absorption principle.        

More recently, the LAP was established for the wave equation of the form {\eqref{eq:wave_LAP}}
by Tamura \cite{Tamura}, but only in dimension $d=3$ and without quantification of the convergence.

In this brief literature review, we have almost entirely
omitted the geometrical issues, which are
the most commonly discussed
aspects
in the literature, see the classical works of Morawetz, e.g. \cite{morawetz1962}, and her collaborators. More on that can also be
found in the introductory part of \cite{Eidus}. Finally, we mention some very recent works related to the validity of the LAP for wave propagation in metamaterials \cite{cassier2017,cassier2022,carvalho2022}.
\medskip

In the present work, we study the LAP for a problem where
both material parameters $\alpha$ and $\beta$ are allowed
to be nonconstant and prove our
results in {spatial dimensions $d=1, 2$ and $3$. The main result
  given in Theorem~\ref{thm:LAP} 
  establishes 
  the validity of the LAP 
  and
  estimates the convergence rates. Additionally, 
  Theorem~\ref{thm:LAP_1D} covers the one-dimensional case
  where a classical formulation of the LAP (i.e. when $U_\infty=0$) is
  known not to be valid~\cite[Sect. 3, Thm. 6]{Eidus1}.}
On a technical side, the novelty of our approach to the proof of the LAP is that it avoids the direct study of the resolvent operator and relies instead on several decay/convergence results. 
{The main features of the present work are:
\begin{itemize}
\item{The LAP is proven for the wave equation with
    nonconstant coefficients, which is not necessarily in divergence form {(i.e.~the equation may have a nonconstant coefficient in front of the divergence operator $\nabla\cdot$)}. Besides the ``classical" case $d=3$, we also consider  $d=2$.}
\item{The validity of the LAP is extended to the case $d=1$ with an appropriate modification.}
\item{The convergence in the LAP is quantified and is shown to be algebraic in time for $d=2,3$ and exponential for $d=1$.}
\end{itemize}} 
{We believe that the exponential and algebraic convergence behavior
  for the cases $d=1$ and $d=2$ are generally sharp, but that the rate
  of the decay for the case $d=3$ might be improved.}
\medskip

{\bf Outline.} The paper is organised as follows. In Section
\ref{sec:conv}, we state time-decay estimates for the time-domain problem
with suitable initial data and source term.
In Section \ref{sec:LAP_proof}, we prove the LAP in the form given
in {Theorems~\ref{thm:LAP} and~\ref{thm:LAP_1D}}. The auxiliary time-decay estimates
of Section \ref{sec:conv}
are proven in Section \ref{sec:proofs}.
Finally, in Section \ref{sec:conc}, we summarise the obtained results and give prospects for further work in related directions.

\section{Time-decay results}
\label{sec:conv}
{In this section, we state some decay-in-time results for}
solutions to the wave equation with sufficiently localised initial
data{, which are used in our proof of the LAP in Section~\ref{sec:LAP_proof} below. The
  proofs of these results are deferred to Section
  \ref{sec:proofs}. More precisely}, we are concerned with the solution of the Cauchy problem
\begin{equation}
\begin{cases}
\partial_{t}^{2}u\left(\mathbf{x},t\right)-\beta^{-1}\left(\mathbf{x}\right)\nabla\cdot\left(\alpha\left(\mathbf{x}\right)\nabla u\left(\mathbf{x},t\right)\right)=f\left(\mathbf{x},t\right),\hspace{1em} \mathbf{x}\in\mathbb{R}^{d},\:t>0,\\
u\left(\mathbf{x},0\right)=u_0\left(\mathbf{x}\right),\hspace{1em}\partial_{t}u\left(\mathbf{x},0\right)=u_1\left(\mathbf{x}\right),\hspace{1em}\mathbf{x}\in\mathbb{R}^{d},
\end{cases}
\label{eq:wave}
\end{equation}
and its constant-coefficient analog with zero source term:
\begin{equation}
\begin{cases}
\partial_{t}^{2}v\left(\mathbf{x},t\right)-c_0^2\Delta v\left(\mathbf{x},t\right)=0,\hspace{1em} \mathbf{x}\in\mathbb{R}^{d},\hspace{1em} t>0,\\
v\left(\mathbf{x},0\right)=v_0\left(\mathbf{x}\right), \hspace{1em}\partial_{t}v\left(\mathbf{x},0\right)=v_1\left(\mathbf{x}\right),\hspace{1em} \mathbf{x}\in\mathbb{R}^{d},
\end{cases}
\label{eq:wave_c0}
\end{equation}
where $c_0:=\sqrt{\alpha_0/\beta_0}$.
We start by considering the problem~\eqref{eq:wave} in the case of localised initial data and zero source term.

\begin{prop}
\label{prop:dec_free}
Let $d\geq2$, $f\equiv 0$. Suppose that  $u_{0}$, $u_{1}\in
H^{2}\left(\mathbb{R}^d\right)$ and $\alpha$, $\beta$ satisfy
Assumptions \ref{assm:alph_bet_base} and \ref{assm:alph_bet_trap}. Additionally, the initial data are assumed to satisfy the following localisation condition: 
\begin{equation}\label{eq:u0u1_loc_cond}
\int_{\mathbb{R}^d} \left(1+\left|\vx\right|^2\right)^{d+1+\epsilon}\left(\left|u_0\left(\vx\right)\right|^2+\left|u_1\left(\vx\right)\right|^2+\left|\Delta u_0\left(\vx\right)\right|^2+\left|\Delta u_1\left(\vx\right)\right|^2\right)d\vx<\infty
\end{equation}
with some $\epsilon>0$. {Denote  $\mathbb{R}_{+}:=\left[0,\infty\right)$.}
Then, for any bounded $\Omega\subset\mathbb{R}^d$, the {unique} solution {$u\in C^{2}\left(\mathbb{R}_{+},L^2\left(\mathbb{R}^d\right)\right)\cap C^{1}\left(\mathbb{R}_{+},H^1\left(\mathbb{R}^d\right)\right)\cap C\left(\mathbb{R}_{+},H^2\left(\mathbb{R}^d\right)\right)$} of
\eqref{eq:wave} obeys the {following} decay estimate with some constant $C>0$, depending on {$u_0$, $u_1$,} $\alpha$, $\beta$, $\epsilon$, $d$, and $\Omega$:
\begin{equation}\label{eq:dec_IC_gen}
\left\Vert u\left(\cdot,t\right) \right\Vert_{H^1\left(\Omega\right)}+\left\Vert \partial_t u\left(\cdot,t\right) \right\Vert_{L^2\left(\Omega\right)}\leq\frac{C}{\left(1+t^2\right)^{\frac{d-1}{2}}},\hspace{1em}t\geq 0.
\end{equation}
\end{prop}

For the case of zero initial data and a localised source term, we have
the following result.

\begin{prop} \label{prop:dec_source}
Let $d\geq2$, $u_0\equiv 0$, $u_1\equiv 0$ and
$\alpha$, $\beta$ satisfy Assumptions~\ref{assm:alph_bet_base}
and \ref{assm:alph_bet_trap}. Additionally, suppose that  {$f\in C^{1}\left(\mathbb{R_+}, L^2\left(\mathbb{R}^d\right)\right)\cap C\left(\mathbb{R_+}, H^1\left(\mathbb{R}^d\right)\right)$}, 
$\underset{t>0}{\cup}\supp f\left(\cdot,t\right)\subset\Omega_f$ for some bounded {domain} $\Omega_f\subset\mathbb{R}^d$, and
there exist constants {$C_f$}, $p>0$ such that  
\begin{equation}\label{eq:assf}
\left\Vert f\left(\cdot,t\right)\right\Vert_{L^2\left(\mathbb{R}^d\right)}+\left\Vert \partial_t f\left(\cdot,t\right)\right\Vert_{L^2\left(\mathbb{R}^d\right)}\leq \frac{{C_f}}{\left(1+t^2\right)^\frac{p}{2}},\hspace{1em}t\geq 0.
\end{equation}
Then, for any bounded domain $\Omega\subset\mathbb{R}^d$, the {unique} solution  {$u\in C^{2}\left(\mathbb{R}_{+},L^2\left(\mathbb{R}^d\right)\right)\cap C^{1}\left(\mathbb{R}_{+},H^1\left(\mathbb{R}^d\right)\right)\cap C\left(\mathbb{R}_{+},H^2\left(\mathbb{R}^d\right)\right)$} of \eqref{eq:wave} obeys the following decay estimates for $t\geq 0$ and some constant $C>0$ depending on $C_f$, $\alpha$, $\beta$, $p$, $d$, and $\Omega$.\\
For $d=2$:
\begin{equation}\label{eq:dec_source_2D}
\left\Vert u\left(\cdot,t\right) \right\Vert_{H^1\left(\Omega\right)}+\left\Vert \partial_t u\left(\cdot,t\right) \right\Vert_{H^1\left(\Omega\right)}\leq C
\begin{cases}
\dfrac{1+\log\left(1+t^2\right)}{\left(1+t^2\right)^\frac{p}{2}},& 0<p\leq1,\\
\dfrac{1}{\left(1+t^2\right)^\frac{1}{2}},& p>1,
\end{cases}
\end{equation}
For $d>2$:
\begin{equation}\label{eq:dec_source_gen}
\left\Vert u\left(\cdot,t\right) \right\Vert_{H^1\left(\Omega\right)}+\left\Vert \partial_t u\left(\cdot,t\right) \right\Vert_{H^1\left(\Omega\right)}\leq C
\begin{cases}
\dfrac{1}{\left(1+t^2\right)^\frac{p}{2}},& 0<p\leq1,\\
\dfrac{1}{\left(1+t^2\right)^\frac{r}{2}},& p>1,
\end{cases}
\end{equation}
where $r:=\min\left(d-1,p\right)$.
\end{prop}

Next, we consider the wave equation~\eqref{eq:wave_c0} with constant
coefficients and $f\equiv 0$.

\begin{lem}\label{lem:gen_dec}
Let $d=2, 3$ 
and $\omega$, $\rho_0>0$, $\rho_1>\rho_0$ be some fixed
constants. Let $\mathbb{S}^{d-1}:=\left\{\vx\in\mathbb{R}^d:
  \left|\vx\right|=1\right\}$ be the $\left(d-1\right)$-dimensional
unit sphere and let $\mathbb{B}_{\rho_0}:=\left\{\vx\in\mathbb{R}^d:
  \left|\vx\right|<\rho_0 \right\}$ be the ball of radius $\rho_0$,
both centered at $\vx=\mathbf{0}$.
Fix $\Omega\Subset\mathbb{B}_{\rho_0}$ meaning that
$\Omega\subseteq\mathbb{B}_{\rho_0-\epsilon}\subset\mathbb{B}_{\rho_0}$
for some $\epsilon>0$.
We make the following assumptions on the initial conditions $v_0$ and $v_1$.
\begin{itemize}
\item {For $d=2$,} we assume that
\begin{equation}\label{eq:gen_ICs}
v_0\left(\mathbf{x}\right)=A_0\left(\left|\mathbf{x}\right|\right)Y_0\left(\frac{\mathbf{x}}{\left|\mathbf{x}\right|}\right)+V_0\left(\vx\right),\hspace{1em}  v_1\left(\mathbf{x}\right)=A_0\left(\left|\mathbf{x}\right|\right)Y_1\left(\frac{\mathbf{x}}{\left|\mathbf{x}\right|}\right)+V_1\left(\vx\right),
\end{equation}
where {$A_0\in
C^6\left(\mathbb{R}_{+}\right)$, $Y_0\in C^6\left(\mathbb{S}^1\right)$, $Y_1\in C^5\left(\mathbb{S}^1\right)$, $V_0\in C^6_b\left(\mathbb{R}^{2}\right)$, $V_1\in C^5_b\left(\mathbb{R}^{2}\right)$} (we use $C_b$ to denote the space of continuous bounded functions)
are such that
$A_0\left(\left|\vx\right|\right)=V_0\left(\vx\right)=V_1\left(\vx\right)\equiv
0$ for $\left|\vx\right|\leq\rho_0$, and that there exists a constant $C_0>0$ such that 
\begin{align}\label{eq:V0V1_dec}
\left|\vx\right|^{5/2} 
\left(\left|V_0\left(\vx\right)\right|+\left|V_1\left(\vx\right)\right|+\left|\nabla V_0\left(\vx\right)\right|+\left|\nabla V_1\left(\vx\right)\right|+\left|\Delta V_0\left(\vx\right)\right|\right.&\\
\left.+\left|\Delta V_1\left(\vx\right)\right|+\left|\nabla\Delta V_0\left(\vx\right)\right|+\left|\nabla\Delta  V_1\left(\vx\right)\right|+\left|\Delta^2 V_0\left(\vx\right)\right|+\left|\Delta^2 V_1\left(\vx\right)\right|\right)&\leq C_0\nonumber
\end{align}
holds true for all $\vx\in\mathbb{R}^{2}$. Moreover,
$A_0\left(\rho\right)=e^{i\frac{\omega}{c_0}\rho}\slash\rho^{3/2}
$
for $\rho>\rho_1$. 
\item {For $d=3$, we assume that {$v_0\in C^{6}(\mR^3)$, $v_1\in C^{5}\left(\mathbb{R}^3\right)$}
and that there exists a constant $C_0>0$ such that
\begin{align}\label{eq:assumption3D}
\left|\vx\right|^{2} 
\left(\left|v_0\left(\vx\right)\right|+\left|v_1\left(\vx\right)\right|+\left|\nabla v_0\left(\vx\right)\right|+\left|\nabla v_1\left(\vx\right)\right|+\left|\Delta v_0\left(\vx\right)\right|\right.&\\
\left.+\left|\Delta v_1\left(\vx\right)\right|+\left|\nabla\Delta v_0\left(\vx\right)\right|+\left|\nabla\Delta  v_1\left(\vx\right)\right|+\left|\Delta^2 v_0\left(\vx\right)\right|+\left|\Delta^2 v_1\left(\vx\right)\right|\right)&\leq C_0\nonumber
\end{align}
holds true for all $\vx\in\mathbb{R}^3$.}

\end{itemize}

Then, there exists a constant $C>0$ such that, for all $\vx\in\Omega$ 
and $t\geq 0$, the {unique} solution {$v\in C^{5}\left(\mathbb{R}^{d}\times\mathbb{R}_{+}\right)$} of \eqref{eq:wave_c0} with the initial
data as above satisfies
\begin{align}\label{eq:sol_gen_bnds}
\left|v\left(\mathbf{x},t\right)\right|+\left|\nabla v\left(\mathbf{x},t\right)\right|+\left|\partial_t v\left(\mathbf{x},t\right)\right|&\\
+\left|\Delta v\left(\mathbf{x},t\right)\right|+\left|\partial_{t}\nabla v\left(\mathbf{x},t\right)\right|+\left|\partial_{t}\Delta v\left(\mathbf{x},t\right)\right|&\leq \frac{C}{\left(1+t^2\right)^{1/2}}.\nonumber
\end{align}
\end{lem}
Note that, in the result of Lemma \ref{lem:gen_dec}, we use smoothness and the presence of the oscillatory exponential {term} in the radial factor {in the case
$d=2$} to deduce the $\mathcal{O}\left(1/t\right)$ decay instead of the more classical $L^\infty$--decay $\mathcal{O}\left(1/t^{1/2}\right)$ of the solution under absolute integrability and some regularity assumptions on the initial data (see e.g. \cite{beals1996}).

In a similar vein, we can obtain the same decay rate as in Lemma \ref{lem:gen_dec} for initial data decaying even slower at infinity. To this effect, we require an additional condition, namely, that $v_1$ is the radial derivative of $v_0$ multiplied by $-c_0$.

\begin{lem}\label{lem:spec_dec}
Let $d=2, 3$ 
and $\omega$, $\rho_0>0$, $\rho_1>\rho_0$ be some fixed
constants.
Using the notation introduced in Lemma \ref{lem:gen_dec}, suppose that $\Omega\Subset\mathbb{B}_{\rho_0}$.
Assume that
\begin{equation}\label{eq:spec_ICs}
v_0\left(\mathbf{x}\right)=A\left(\left|\mathbf{x}\right|\right)Y\left(\frac{\mathbf{x}}{\left|\mathbf{x}\right|}\right),\hspace{1em}  v_1\left(\mathbf{x}\right)=-c_0\partial_{\left|\mathbf{x}\right|}v_0\left(\mathbf{x}\right)=-c_0 A^\prime\left(\left|\mathbf{x}\right|\right)Y\left(\frac{\mathbf{x}}{\left|\mathbf{x}\right|}\right),
\end{equation}
where $\partial_{\left|\vx\right|}$ denotes the derivative in the radial direction of the variable $\vx$, $A\in {C^7}\left(\mathbb{R}_{+}\right)$, $Y\in {C^7}\left(\mathbb{S}^{d-1}\right)$ such that $A\left(\rho\right)\equiv 0$ for $\rho\in\left[0,\rho_0\right]$ and 
{$A\left(\rho\right)=e^{i\frac{\omega}{c_0}\rho}/\rho^{\frac{d-1}{2}}$} for $\rho>\rho_1$. 
	  Then, there exists a constant $C>0$ such that, for all $\vx\in\Omega$ 
and $t\geq 0$, the {unique} solution {$v\in C^{6}\left(\mathbb{R}^{d}\times\mathbb{R}_{+}\right)$} of \eqref{eq:wave_c0} with the initial data \eqref{eq:spec_ICs} satisfies
\begin{align}\label{eq:sol_der_bnds}
\left|v\left(\mathbf{x},t\right)\right|+\left|\nabla v\left(\mathbf{x},t\right)\right|+\left|\partial_t v\left(\mathbf{x},t\right)\right|&\\
+\left|\Delta v\left(\mathbf{x},t\right)\right|+\left|\partial_{t}\nabla v\left(\mathbf{x},t\right)\right|+\left|\partial_{t}\Delta v\left(\mathbf{x},t\right)\right|&\leq \frac{C}{\left(1+t^2\right)^{1/2}}.\nonumber
\end{align}
\end{lem}

In the one-dimensional case, we have the following exponential decay
result{, which is proven in~\cite{paper-decay1D}}.

\begin{prop}[{\cite{paper-decay1D}, Prop. 1.1, Thm. 1.4}]
	\label{prop:1D_decay}
	 Let $d=1$ and $f\equiv 0$.
Suppose that $u_{0}\in H^{1}\left(\mathbb{R}\right)$, $u_{1}\in L^{2}\left(\mathbb{R}\right)$, $\supp u_{0}$, $\supp u_{1}\subset\Omega$ for some bounded $\Omega\subset\mathbb{R}$ and assume $\alpha$, $\beta$, $\Omega_{in}$ be as in Assumption \ref{assm:alph_bet_base_1D}.
Then, for any bounded $\Omega\subset\mathbb{R}$, the {unique} solution {$u\in C^{1}\left(\mathbb{R}_{+},L^2\left(\mathbb{R}\right)\right)\cap C\left(\mathbb{R}_{+},H^1\left(\mathbb{R}\right)\right)$} of \eqref{eq:wave} obeys the decay estimate
	\begin{equation}
	\left\Vert u\left(\cdot,t\right)-u_{\infty}\right\Vert _{H^{1}\left(\Omega\right)}+\left\Vert \partial_{t}u\left(\cdot,t\right)\right\Vert _{L^{2}\left(\Omega\right)}\leq Ce^{-\Lambda t},\hspace{1em}t\geq 0,\label{eq:u_exp_conv}
	\end{equation}
	for some explicit constants $C=C\left(u_{0},u_{1},\alpha,\beta,\left|\Omega\right|\right)$,
	$\Lambda=\Lambda\left(\alpha,\beta\right)>0$  with $\left|\Omega\right|$ denoting the Lebesgue measure of the set $\Omega$, and 
{
	\begin{equation}
	u_{\infty}:=\frac{1}{2\sqrt{\alpha_0\beta_0}}\int_{\Omega}u_{1}\left(x\right)\beta\left(x\right)dx.\label{eq:u_infty_1D}
	\end{equation}
}
\end{prop}

\section{Proof of the LAP (Theorems~\ref{thm:LAP} and~\ref{thm:LAP_1D})}\label{sec:LAP_proof}

In this section, we prove Theorems~\ref{thm:LAP} and~\ref{thm:LAP_1D}
at once.
Without loss of generality, we can assume that $\Omega=\Omega_{in}$,
since both domains could be enlarged to their union without changing the
problem. We also suppose that the origin $\mathbf{x}=\mathbf{0}$ is
chosen to be inside $\Omega$.

The proof is given in two steps. In Step~1, see~Section~\ref{sec:step1}, we transform problem~\eqref{eq:wave_LAP}
into an initial-value problem with zero source term for the difference 
\begin{equation}\label{eq:W_def}
W\left(\mathbf{x},t\right):=u\left(\mathbf{x},t\right)-e^{-i\omega t}U\left(\mathbf{x}\right),
\end{equation}
where $u\left(\mathbf{x},t\right)$ and $U\left(\mathbf{x}\right)$
solve problems \eqref{eq:wave_LAP} and \eqref{eq:U_Helm},
respectively.
In Section~\ref{sec:poorly_localised}, we observe that the problem
introduced in Step~1 has poorly localised initial
data, and we write an integral representation, which will be
useful in what follows.
In Step~2, see Section~\ref{sec:step2}, we decompose the problem from Step~1 into
several subproblems. We distinguish the cases $d=1$ and $d\ge 2$.
In the former case, the arguments are more {transparent and lead} to the
quantitative result of Theorem~\ref{thm:LAP_1D}.
{The higher-dimensional case is more involved, as some of the subproblems do not
have sufficiently localised intitial
data and thus require the more specific time-decay results given in Section \ref{sec:conv}.}

\subsection{Step~1: Transformation into an auxiliary homogeneous problem.}\label{sec:step1}
	
By inspection, we see that $W\left(\vx,t\right)$ defined by \eqref{eq:W_def} satisfies	
	\begin{equation}\label{eq:wave_W}
	\begin{cases}
	\partial_{t}^{2}W\left(\mathbf{x},t\right)-\beta^{-1}\left(\mathbf{x}\right)\nabla\cdot\left(\alpha\left(\mathbf{x}\right)\nabla W\left(\mathbf{x},t\right)\right)=0,& \mathbf{x}\in\mathbb{R}^d, \hspace{1em} t>0,\\
	W\left(\mathbf{x},0\right)=-U\left(\mathbf{x}\right),\hspace{1em}\partial_{t}W\left(\mathbf{x},0\right)=i\omega U\left(\mathbf{x}\right),& \mathbf{x}\in\mathbb{R}^d.
	\end{cases}
	\end{equation}
Completing the proofs of Theorems~\ref{thm:LAP} and~\ref{thm:LAP_1D}
is tantamount to showing that there exists a unique constant $U_\infty\in\mathbb{C}$ explicitly given by \eqref{eq:U_infty_expl} and constants $\Lambda$, $C>0$ depending on $F$, $\alpha$, $\beta$, $\omega$ such that\\
for $d=1$:
\begin{equation}\label{eq:W_dec_estim_1D}
\left\Vert
W\left(\cdot,t\right)-U_\infty\right\Vert_{H^1\left(\Omega\right)}
+\left\Vert \partial_t
W\left(\cdot,t\right)\right\Vert_{L^2\left(\Omega\right)}\le C e^{-\Lambda t},\hspace{1em} t\geq0,
\end{equation}
for $d=2$:
\begin{equation}\label{eq:W_dec_estim_2D}
\left\Vert W\left(\cdot,t\right)\right\Vert_{H^1\left(\Omega\right)}+\left\Vert \partial_t W\left(\cdot,t\right)\right\Vert_{L^2\left(\Omega\right)}\leq C\,\frac{1+\log\left(1+t^2\right)}{\left(1+t^2\right)^{1/2}}, \hspace{1em} t\geq0,
\end{equation}
for $d=3$:
\begin{equation}\label{eq:W_dec_estim_3D}
\left\Vert W\left(\cdot,t\right)\right\Vert_{H^1\left(\Omega\right)}+\left\Vert \partial_t W\left(\cdot,t\right)\right\Vert_{L^2\left(\Omega\right)}\leq \frac{C}{\left(1+t^2\right)^{1/2}}, \hspace{1em} t\geq0.
\end{equation}

\subsection{Slow decay of the initial data of problem~\eqref{eq:wave_W}}
\label{sec:poorly_localised}

One immediate difficulty when dealing with \eqref{eq:wave_W} is that the initial data
        $W{\left(\cdot,0\right)}$ and $\partial_t
        W{\left(\cdot,0\right)}$
        do not belong to $H^1\left(\mathbb{R}^d\right)$ and $L^2\left(\mathbb{R}^d\right)$, respectively.
{The slow decay of the initial conditions in~\eqref{eq:wave_W} can be seen as follows. Let us rewrite \eqref{eq:U_Helm} as} the constant-coefficient problem
	\begin{align}
	-\Delta U\left(\mathbf{x}\right)-\frac{\omega^2}{c_0^2}\, U\left(\mathbf{x}\right)=&\frac{1}{\alpha_0}\left[\beta\left(\mathbf{x}\right) F\left(\mathbf{x}\right)+\left(\beta\left(\mathbf{x}\right)-\beta_0\right)\omega^2 U\left(\mathbf{x}\right)+\right. \nonumber\\
	&\left.+\nabla\cdot(\alpha\left(\mathbf{x}\right)\nabla U\left(\mathbf{x}\right))-\alpha_0\Delta U\left(\mathbf{x}\right)\right]\label{eq:constsantcoeff}\\
	=:&F_1\left(\mathbf{x}\right), \nonumber
	\end{align}
        where we recall that $c_0^2=\alpha_0/\beta_0$.
Assumptions~\ref{assm:alph_bet_base} (or \ref{assm:alph_bet_base_1D} if $d=1$) and~\ref{assm:F_base} 
on the coefficients and on $F$ imply that
$F_1\left(\mathbf{x}\right)=0$ for
$\vx\in\mathbb{R}^d\backslash\bar{\Omega}$.
{Moreover, since the coefficients $\alpha$ and $\beta$ are smooth (for $d\ge2$) and bounded away from zero, and $F\in L^2\left(\mathbb{R}^d\right)$, standard well-posedness results (see e.g. \cite[Sec. 6.3.1]{Evans}) give $U\in H^2(\Omega)$, and hence $F_1\in L^1(\Omega)$.
Therefore, we can write the integral representation {of the
  solution $U$ in $\mathbb{R}^d\backslash\bar{\Omega}$}}
	\begin{equation}
	U\left(\mathbf{x}\right)=\int_{\Omega}K\left(\vx-\vy\right)F_1\left(\mathbf{y}\right)d\vy,\hspace{1em}\vx\in\mathbb{R}^d{\backslash\bar{\Omega}}.\label{eq:U_int_repr}
	\end{equation}
Here
	\begin{equation}\label{eq:K_def}
	K\left(\vx\right):=
	\frac{i}{4}\left(\frac{\omega}{2\pi c_0}\right)^\frac{d-2}{2}
        \frac{1}{\left|\vx\right|^\frac{d-2}{2}}\;
          H_\frac{d-2}{2}^{\left(1\right)}\left(\frac{\omega}{c_0}\left|\vx\right|\right)
	\end{equation}
is the Green's function for the Helmholtz equation (see e.g.~\cite{Engquist-Zhao}) that satisfies the Sommerfeld radiation condition
        $\underset{\left|\mathbf{x}\right|\rightarrow\infty}{\lim}\left|\mathbf{x}\right|^\frac{d-1}{2}\left[\partial_{\left|\mathbf{x}\right|}K\left(\mathbf{x}\right)-i\frac{\omega}{c_0}
          K\left(\mathbf{x}\right)\right]=0$ and $-\Delta
        K\left(\vx\right)-\frac{\omega^2}{c_0^2}
        K\left(\vx\right)=\delta\left(\vx\right)$, with $\delta$ being
        the $d$-dimensional Dirac delta function.
        In~\eqref{eq:K_def},  $H_p^{\left(1\right)}$ denotes the
          Hankel function of the first kind of order $p$.
Since $\mathbf{y}$ in \eqref{eq:U_int_repr} ranges in a bounded set
and $F_1\in L^1\left(\Omega\right)$, we employ
  Lemma~\ref{lem:K_estims} from Appendix~\ref{sec:app} and deduce that
\begin{equation}\label{eq:U_dec_Somm}	
U\left(\vx\right)=\mathcal{O}\left(1/\left|\vx\right|^{\left(d-1\right)/2}\right),\quad \partial_{\left|\vx\right|}U\left(\mathbf{x}\right)-\frac{i\omega}{c_0}U\left(\vx\right)=\mathcal{O}\left(1/\left|\vx\right|^{\left(d+1\right)/2}\right),\quad \left|\vx\right|\gg 1.
\end{equation}
This {implies that $U$, and therefore $W\left(\cdot,0\right)$
  and $\partial_t W\left(\cdot,0\right)$, do not belong {necessarily} to
  $L^2\left(\mathbb{R}^d\right)$. At the same time, this} {gives a precise decay rate in} 
the Sommerfeld radiation condition when the source term $F_1$ {in
  \eqref{eq:constsantcoeff}} is compactly supported.

\subsection{Step~2: Time-decay by decomposition into subproblems.}\label{sec:step2}
{In order} 
to deal with the slowly decaying
        initial data in {problem}~\eqref{eq:wave_W} discussed in
          Section~\ref{sec:poorly_localised}{, we} 
        perform some auxiliary {decompositions} {using the linearity of the problem and the uniqueness of its solution. 
        } 
        As $\mathbf{0}\in\Omega$,
        we can fix $R$ large enough and $\epsilon>0$ such that
        $\Omega$ is contained in the open ball
        $\mathbb{B}_{R-\epsilon}\subset\mathbb{R}^d$ of radius $R-\epsilon$ and center $\vx=\mathbf{0}$.
Let $\left\{\eta_0, \eta_1\right\}$ be a smooth, radial partition
  of unity, i.e. $\eta_0=\eta_0\left(\left|\vx\right|\right)$,
  $\eta_1=\eta_1\left(\left|\vx\right|\right)\in C^\infty\left(\mathbb{R}_{+}\right)$, and
  $\eta_0\left(\left|\vx\right|\right)+\eta_1\left(\left|\vx\right|\right)=1$ for all $\vx\in\mathbb{R}^d$, such that
	\begin{equation}\label{POU}
	\eta_0\left(\left|\vx\right|\right)=\begin{cases}
	0, \quad\left|\vx\right|<R-\epsilon,\\
	1, \quad\left|\vx\right|>R,
	\end{cases}\qquad
	\eta_1\left(\left|\vx\right|\right)=\begin{cases}
	1, \quad\left|\vx\right|<R-\epsilon,\\
	0, \quad\left|\vx\right|>R.
	\end{cases}
	\end{equation}
We proceed
separately with the case $d=1$ and the {cases $d=2, 3$}.
\medskip

\noindent
{\bf $\bullet$ Case $d=1$ (Theorem~\ref{thm:LAP_1D}).}
	 
Note that $H_{-1/2}^{\left(1\right)}\left(x\right)=\left(\frac{2}{\pi x}\right)^\frac{1}{2}e^{i x}$, and hence \eqref{eq:K_def} yields, for $d=1$, $K\left(\left|x\right|\right)=\frac{i}{2} \frac{c_0}{\omega}e^{i\frac{\omega}{c_0}\left|x\right|}$. In this case,
the Green function $K$ does not decay at infinity, but
the radiation conditions {on $K$, and thus on $U$,} are exact,
i.e. for $x\notin\Omega$, we have $c_0\partial_{\left|x\right|}U\left(x\right)=i\omega U\left(x\right)$, where $\partial_{\left|x\right|}\equiv\left(\text{sgn }x\right)\partial_x$. Therefore, we can write
	\begin{equation}
	W\left(x,t\right)=\widetilde{W}_0\left(x,t\right)+\widetilde{W}_1\left(x,t\right),\label{eq:W_decomp_1D}
	\end{equation}
	where $\widetilde{W}_0\left(x,t\right)$, $\widetilde{W}_1\left(x,t\right)$ solve the following initial-value problems, {respectively}:
	\begin{equation}\label{eq:W0_tild}
	\begin{cases}
	\partial_{t}^{2}\widetilde{W}_0\left(x,t\right)-\beta^{-1}\left(x\right)\partial_x\left(\alpha\left(x\right)\partial_x \widetilde{W}_0\left(x,t\right)\right)=0, \qquad x\in\mathbb{R}, \hspace{1em} t>0,\\
	\widetilde{W}_0\left(x,0\right)=-\eta_0\left(\left|x\right|\right)U\left(x\right),\hspace{1em}\partial_{t}\widetilde{W}_0\left(x,0\right)=c_0 \partial_{\left|x\right|}\biggl(\eta_0\left(\left|x\right|\right)U\left(x\right)\biggr),\quad x\in\mathbb{R},
	\end{cases}
	\end{equation}
	\begin{equation}\label{eq:W1_tild}
	\begin{cases}
	\partial_{t}^{2}\widetilde{W}_1\left(x,t\right)-\beta^{-1}\left(x\right)\partial_x\left(\alpha\left(x\right)\partial_x \widetilde{W}_1\left(x,t\right)\right)=0, \qquad x\in\mathbb{R}, \hspace{1em} t>0,\\
	\widetilde{W}_1\left(x,0\right)=-\eta_1\left(\left|x\right|\right)U\left(x\right),\hspace{1em}\partial_{t}\widetilde{W}_1\left(x,0\right)={\biggl(c_0\partial_{\left|x\right|}\eta_1\left(\left|x\right|\right)+i\omega\eta_1\left(\left|x\right|\right)\biggr)U\left(x\right)},\quad x\in\mathbb{R}.
	\end{cases}
	\end{equation}
	
	Observe that problem~\eqref{eq:W0_tild}{, whose initial
        data are supported outside $\mathbb{B}_{R-\epsilon}$,} is solved by a linear combination of two reflection-free outgoing waves
	\begin{align}\label{eq:W0_tild_sol}
	\widetilde{W}_0\left(x,t\right)=&-H\left(x-c_0 t\right)\eta_0\left(|x-c_0 t|\right)U\left(x-c_0 t\right)\\
	&-H\left(-x-c_0 t\right)\eta_0\left(|x+c_0 t|\right)U\left(x+c_0 t\right)\nonumber,
	\end{align}
	where $H$ is the Heaviside step function. Note that the smoothness of the solution is not affected by the discontinuity of the Heaviside function due to the vanishing of $\eta_0$.  
	Because of the support property of $\eta_0$, by inspection of
        \eqref{eq:W0_tild_sol}, we have that
{
\begin{equation}\label{eq:W0_tild_estim}
\widetilde{W}_0\left(x,t\right)=\partial_t\widetilde{W}_0\left(x,t\right)\equiv0, \hspace{1em} x\in\Omega,\hspace{1em} t>0.
\end{equation}
}
        
	To deal with $\widetilde{W}_1$ in \eqref{eq:W_decomp_1D}, we observe that the initial data of \eqref{eq:W1_tild} have compact support. Hence, problem~\eqref{eq:W1_tild} is amenable to the application of Proposition~\ref{prop:1D_decay}, which yields
{
\begin{equation}\label{eq:W1_tild_estim}
\left\Vert
          \widetilde{W}_1\left(\cdot,t\right)-U_\infty\right\Vert_{H^1\left(\Omega\right)}+\left\Vert
          \partial_t\widetilde{W}_1\left(\cdot,t\right)\right\Vert_{L^2\left(\Omega\right)}\le
        Ce^{-\Lambda t},\hspace{1em}t\geq 0,
\end{equation}
\begin{align}\label{eq:U_infty_estim}
U_\infty:=&\frac{1}{2\sqrt{\alpha_0\beta_0}}\int_{{-R}}^{{R}}\biggl(c_0\partial_{\left|x\right|}\eta_1\left(\left|x\right|\right)+i\omega\eta_1\left(\left|x\right|\right)\biggr)U\left(x\right)\beta\left(x\right) dx\\
=&\frac{i\omega}{2\sqrt{\alpha_0\beta_0}}\int_{{-R+\epsilon}}^{{R-\epsilon}}U\left(x\right)\beta\left(x\right)dx-\frac{1}{2}\left[U\left({R-\epsilon}\right)+U\left({-R+\epsilon}\right)\right] \nonumber\\
=&\frac{1}{2i\omega\sqrt{\alpha_0\beta_0}}\int_{{-R+\epsilon}}^{{R-\epsilon}} F\left(x\right)\beta\left(x\right)dx\nonumber
\end{align}
for some constants $C$, $\Lambda>0$. 
Note that in passing from the first to the second line in \eqref{eq:U_infty_estim}, $\eta_1$ disappears upon integration by parts using that $\partial_{\left|x\right|}U\left(x\right)=i\omega/c_0 U\left(x\right)$ and $\beta\left(x\right)\equiv\beta_0$ for $x\in\left[{-R},{-R+\epsilon}\right]\cup\left[{R-\epsilon},{R}\right]$. The passage from the second to the third line of the equality is justified upon integration of \eqref{eq:U_Helm} in $x$ over the interval $\left({-R+\epsilon},{R-\epsilon}\right)$ and using again the exact radiation conditions at its endpoints.}

Together with {\eqref{eq:W0_tild_estim} and ~\eqref{eq:W_decomp_1D},  
estimate \eqref{eq:W1_tild_estim} implies \eqref{eq:W_dec_estim_1D}} which completes the proof
        of {Theorem~\ref{thm:LAP_1D}}.  
	\medskip

\noindent
{\bf $\bullet$ Cases $d=2, 3$ (Theorem~\ref{thm:LAP}).}

We perform a decomposition, which is similar to \eqref{eq:W_decomp_1D}
but contains more terms that have to be treated individually in a more delicate
fashion. 

\medskip
\underline{Decomposition of $W$.}
We write {the unique solution of \eqref{eq:wave_W} as }
	\begin{equation}
	W\left(\mathbf{x},t\right)=\sum_{k=1}^4 W_k\left(\mathbf{x},t\right),\label{eq:W_decomp}
	\end{equation}
	where $W_1$ solves the homogeneous wave equation
\begin{equation*}
\partial_{t}^{2}W_1\left(\mathbf{x},t\right)-\beta^{-1}\left(\mathbf{x}\right)\nabla\cdot\left(\alpha\left(\mathbf{x}\right)\nabla W_1\left(\mathbf{x},t\right)\right)=0,\quad \mathbf{x}\in\mathbb{R}^d,\quad t>0,
	\end{equation*}
subject to the initial conditions on $\mathbb{R}^d$:
$$
  W_1\left(\mathbf{x},0\right)=-\eta_1\left(\left|\vx\right|\right)U\left(\vx\right), \quad\partial_{t}W_1\left(\mathbf{x},0\right)=i\omega \eta_1\left(\left|\vx\right|\right)U\left(\vx\right)-c_0 \eta_0^{\prime}\left(\left|\vx\right|\right)U_{0}\left(\vx\right).
$$
$W_2$ and $W_3$ solve the constant-coefficient problems~\eqref{eq:V12_pbm} and~\eqref{eq:V_pbm},  respectively, and $W_4$ solves the inhomogeneous wave equation \eqref{eq:wave_Z12} {with non-constant coefficients.  
For the initial conditions we use the partition of unity \eqref{POU}. }
{The nonzero right-hand side in~\eqref{eq:wave_Z12} is needed to compensate the fact that the equations in problems~\eqref{eq:V12_pbm} and~\eqref{eq:V_pbm} are different from that in problem~\eqref{eq:wave_W}.}

Here we have introduced {$U_0$, the leading term in the long-range
  asymptotic expansion of~\eqref{eq:U_int_repr}. More precisely,
according to representation~\eqref{eq:U_int_repr} and Lemma~\ref{lem:K_estims}, we have 
}
\begin{equation}
	U_{0}\left(\mathbf{x}\right)=\frac{e^{i\frac{\omega}{c_{0}}\left|\mathbf{x}\right|}}{4\pi\left|\mathbf{x}\right|^{\frac{d-1}{2}}}\left(\frac{\omega}{2\pi i c_0}\right)^{\frac{d-3}{2}}
\int_{\Omega}e^{-i\frac{\omega}{c_{0}}\frac{\mathbf{x}\cdot\mathbf{y}}{\left|\mathbf{x}\right|}}F_{1}\left(\mathbf{y}\right)d\mathbf{y}.\label{eq:U_0_def}
	\end{equation}
{Furthermore, for $\left|\vx\right|\gg1$,}
\begin{align}\label{eq:U_U0}
U\left(\vx\right)-U_0\left(\vx\right)=&\frac{e^{i\frac{\omega}{c_{0}}\left|\mathbf{x}\right|}}{4\pi\left|\mathbf{x}\right|^{\frac{d+1}{2}}}\left(\frac{\omega}{2\pi i c_0}\right)^{\frac{d-3}{2}}
\int_{\Omega}e^{-i\frac{\omega}{c_{0}}\frac{\mathbf{x}\cdot\mathbf{y}}{\left|\mathbf{x}\right|}}\left[\left(d-3\right)\left(d-1\right)\frac{ic_0}{8\omega}\right.\\
&\left.+\frac{d-1}{2}\frac{\vx\cdot\vy}{\left|\vx\right|}+\frac{i\omega}{2c_0}\left(\left|\vy\right|^2-\left(\frac{\vx\cdot\vy}{\left|\vx\right|}\right)^2\right)\right]F_{1}\left(\mathbf{y}\right)d\mathbf{y}+\mathcal{O}\left(\frac{1}{\left|\vx\right|^{\frac{d+3}{2}}}\right)\nonumber,
\end{align}
\begin{align}\label{eq:U_U0_der}
\partial_{\left|\vx\right|}\left[U\left(\vx\right)-U_0\left(\vx\right)\right]=&-\frac{e^{i\frac{\omega}{c_{0}}\left|\mathbf{x}\right|}}{4\pi\left|\mathbf{x}\right|^{\frac{d+1}{2}}}\left(\frac{\omega}{2\pi i c_0}\right)^{\frac{d-3}{2}}
\int_{\Omega}e^{-i\frac{\omega}{c_{0}}\frac{\mathbf{x}\cdot\mathbf{y}}{\left|\mathbf{x}\right|}}\left[\frac{\left(d-3\right)\left(d-1\right)}{8}\right.\\
&\left.-\frac{d-1}{2}\frac{i\omega}{c_0}\frac{\vx\cdot\vy}{\left|\vx\right|}+\frac{\omega^2}{2c_0^2}\left(\left|\vy\right|^2-\left(\frac{\vx\cdot\vy}{\left|\vx\right|}\right)^2\right)\right]F_{1}\left(\mathbf{y}\right)d\mathbf{y}\nonumber\\
&+\mathcal{O}\left(\frac{1}{\left|\vx\right|^{\frac{d+3}{2}}}\right),
                                                                                \nonumber
\end{align}
{\begin{align}\label{eq:U_U_der}
\partial_{\left|\vx\right|}U\left(\mathbf{x}\right)-\frac{i\omega}{c_0}U\left(\vx\right)=&\frac{e^{i\frac{\omega}{c_{0}}\left|\mathbf{x}\right|}}{4\pi\left|\mathbf{x}\right|^{\frac{d+1}{2}}}\left(\frac{\omega}{2\pi i c_0}\right)^{\frac{d-3}{2}}
\frac{1-d}{2}\int_{\Omega}e^{-i\frac{\omega}{c_{0}}\frac{\mathbf{x}\cdot\mathbf{y}}{\left|\mathbf{x}\right|}}F_1\left(\vy\right)d\mathbf{y}+\mathcal{O}\left(\frac{1}{\left|\vx\right|^{\frac{d+3}{2}}}\right).
\end{align}}

\medskip
\underline{Decay of $W_1$.}
In order to apply Proposition \ref{prop:dec_free}, we need to check the regularity of the initial conditions of $W_1$. Since $U\in{H^2\left(\Omega\right)}$,  
we find that $W_1(.,0)$ and the first term of $\partial_t W_1(.,0)$ are in $H^2\left(\mathbb{R}^d\right)$, by recalling that $\supp \eta_1\subset \bar{\mathbb{B}}_{R}$. 
{The second term of~$\partial_t W_1(.,0)$ is {in~$C^\infty(\mR^d)$, } as the integral in~\eqref{eq:U_0_def}
is the Fourier transform ({more precisely}, its restriction to the unit sphere) of a compactly supported function; see the text below definition~\eqref{eq:constsantcoeff}.
}
Moreover, this integral is constant in the radial direction. 
{In addition to being smooth, the second term of $\partial_t W_1(.,0)$ has compact support since $\supp[\eta'_0(|\vx|)]\subset \bar{\mathbb{B}}_{R}\setminus \mathbb{B}_{R-\epsilon}$.}
{
We thus conclude that} $\partial_t W_1(.,0)$ is also {$H^2\left(\mathbb{R}^d\right)$}.

Since all initial data of $W_1$ are compactly supported, the growth estimate \eqref{eq:u0u1_loc_cond} clearly holds. Hence, Proposition \ref{prop:dec_free} applies to give
\begin{equation}\label{eq:W34_dec_estim}
\left\Vert W_1\left(\cdot,t\right)\right\Vert_{H^1\left(\Omega\right)}+\left\Vert \partial_t W_1\left(\cdot,t\right)\right\Vert_{L^2\left(\Omega\right)}\leq \frac{C}{\left(1+t^2\right)^{\frac{d-1}{2}}}, \hspace{1em} t\geq0,
\end{equation}
for some constant $C>0$ depending on $\Omega$.

\medskip
\underline{Decay of $W_2$.}
$W_2\left(\vx,t\right)$ is the unique solution to the constant-coefficient problem
\begin{equation}\label{eq:V12_pbm}
\begin{cases}
\partial_{t}^{2}W_2\left(\mathbf{x},t\right)-c_0^2\Delta W_2\left(\mathbf{x},t\right)=0, \qquad\mathbf{x}\in\mathbb{R}^d, \hspace{1em} t>0,\\
W_2\left(\mathbf{x},0\right)=\eta_0\left(\left|\vx\right|\right)\left(U_{0}\left(\mathbf{x}\right)-U\left(\mathbf{x}\right)\right),\qquad\mathbf{x}\in\mathbb{R}^d,\\
\partial_{t}W_2\left(\mathbf{x},0\right)=c_0 \eta_0\left(\left|\vx\right|\right)\left(\frac{i\omega}{c_0}U\left(\vx\right)-\partial_{\left|\vx\right|}U_{0}\left(\mathbf{x}\right)\right),\quad \mathbf{x}\in\mathbb{R}^d.
\end{cases}
\end{equation}
{Note that even though $U\in H^2\left(\Omega\right)$, it follows from the smoothness of the kernel function in~\eqref{eq:U_int_repr} that $U$ is {arbitrarily} smooth in $\mathbb{R}^d\backslash\bar{\Omega}$. Hence, recalling that $\eta_0$ is zero in $\Omega$, 
and that $U_0$ is also smooth {away from $\vx=0$}, we deduce that $W_2\left(\cdot,0\right)$, $\partial_t W_2\left(\cdot,0\right)\in C^\infty\left(\mathbb{R}^{d}\right)$.}
Moreover, since 
\begin{equation*}
\frac{i\omega}{c_0}U\left(\vx\right)-\partial_{\left|\vx\right|}U_{0}\left(\mathbf{x}\right)=\partial_{\left|\vx\right|}\left[U\left(\mathbf{x}\right)-U_0\left(\mathbf{x}\right)\right]-\left(\partial_{\left|\vx\right|}U\left(\mathbf{x}\right)-\frac{i\omega}{c_0}U\left(\vx\right)\right),
\end{equation*}
we see from \eqref{eq:U_U0}--\eqref{eq:U_U_der}
that the initial conditions of \eqref{eq:V12_pbm} satisfy the assumptions of Lemma~\ref{lem:gen_dec} with
\begin{equation*}
  A_0\left(\left|\vx\right|\right):=\eta_0\left(\left|\vx\right|\right)
  \frac{e^{i\frac{\omega}{c_0}\left|\vx\right|}}{\left|\vx\right|^{\frac{d+1}{2}}},
\end{equation*}
\begin{align*}
Y_0\left(\frac{\vx}{\left|\vx\right|}\right):=&-\frac{1}{4\pi}\left(\frac{\omega}{2\pi i c_0}\right)^{\frac{d-3}{2}}
\int_{\Omega}e^{-i\frac{\omega}{c_{0}}\frac{\mathbf{x}\cdot\mathbf{y}}{\left|\mathbf{x}\right|}}\left[\left(d-3\right)\left(d-1\right)\frac{ic_0}{8\omega}\right.\\
&\left.+\frac{d-1}{2}\frac{\vx\cdot\vy}{\left|\vx\right|}+\frac{i\omega}{2c_0}\left(\left|\vy\right|^2-\left(\frac{\vx\cdot\vy}{\left|\vx\right|}\right)^2\right)\right]F_{1}\left(\mathbf{y}\right)d\mathbf{y}\nonumber,
\end{align*}
\begin{align*}
Y_1\left(\frac{\vx}{\left|\vx\right|}\right):=&-\frac{c_0}{4\pi}\left(\frac{\omega}{2\pi i c_0}\right)^{\frac{d-3}{2}}
\int_{\Omega}e^{-i\frac{\omega}{c_{0}}\frac{\mathbf{x}\cdot\mathbf{y}}{\left|\mathbf{x}\right|}}\left[\frac{\left(d-7\right)\left(d-1\right)}{8}\right.\\
&\left.-\frac{d-1}{2}\frac{i\omega}{c_0}\frac{\vx\cdot\vy}{\left|\vx\right|}+\frac{\omega^2}{2c_0^2}\left(\left|\vy\right|^2-\left(\frac{\vx\cdot\vy}{\left|\vx\right|}\right)^2\right)\right]F_{1}\left(\mathbf{y}\right)d\mathbf{y}\nonumber,
\end{align*}
and 
{$V_0$, $V_1\in C_b^5(\mR^d)$ for $d=2,\,3$. Moreover, the assumptions \eqref{eq:V0V1_dec} and \eqref{eq:assumption3D} on the initial conditions can be verified using \eqref{eq:U_U0}--\eqref{eq:U_U_der}. 
Thus Lemma~\ref{lem:gen_dec} } 
entails that {the solution $W_2\in C^5\left(\mathbb{R}^{d}\times\mathbb{R}_{+}\right)$ obeys} the following decay estimates uniformly in $\vx\in\Omega$ for $t\geq0$:
\begin{align}\label{eq:V12_dec_estim}
\left|W_2\left(\mathbf{x},t\right)\right|+\left|\nabla W_2\left(\mathbf{x},t\right)\right|+\left|\partial_t W_2\left(\mathbf{x},t\right)\right|&\\
+\left|\Delta W_2\left(\mathbf{x},t\right)\right|+\left|\partial_{t}\nabla W_2\left(\mathbf{x},t\right)\right|+\left|\partial_{t}\Delta W_2\left(\mathbf{x},t\right)\right|&\leq \frac{C}{\left(1+t^2\right)^{1/2}}\nonumber
\end{align}
with some constant $C>0$. In particular, \eqref{eq:V12_dec_estim} implies
\begin{equation}\label{eq:V12_dec_estim2}
\left\Vert W_2\left(\cdot,t\right)\right\Vert_{H^1\left(\Omega\right)}+\left\Vert \partial_t W_2\left(\cdot,t\right)\right\Vert_{L^2\left(\Omega\right)}\leq \frac{C}{\left(1+t^2\right)^{1/2}}, \hspace{1em} t\geq0.
\end{equation}

\medskip
\underline{Decay of $W_3$.}
$W_3\left(\vx,t\right)$ is the unique solution to the constant-coefficient problem
	\begin{equation}
	\begin{cases}
	\partial_{t}^{2}W_3\left(\mathbf{x},t\right)-c_0^2\Delta W_3\left(\mathbf{x},t\right)=0, \qquad\mathbf{x}\in\mathbb{R}^d, \hspace{1em} t>0,\label{eq:V_pbm}\\
	W_3\left(\mathbf{x},0\right)=-\eta_0\left(\left|\vx\right|\right)U_{0}\left(\mathbf{x}\right),\hspace{1em}\partial_{t}W_3\left(\mathbf{x},0\right)=c_0 \partial_{\left|\vx\right|}\big(\eta_0\left(\left|\vx\right|\right)U_{0}\left(\mathbf{x}\right)\big),\quad \mathbf{x}\in\mathbb{R}^d.
	\end{cases}
      \end{equation}
{As for $W_2$, the initial conditions satisfy $W_3\left(\cdot,0\right)$, $\partial_t W_3\left(\cdot,0\right)\in C^\infty(\mR^d)$. } 
Moreover, by setting
	\[
	A\left(\left|\vx\right|\right):=\eta_0\left(\left|\vx\right|\right)\frac{e^{i\frac{\omega}{c_{0}}\left|\mathbf{x}\right|}}{\left|\mathbf{x}\right|^{\frac{d-1}{2}}},\hspace{1.5em}
        Y\left(\frac{\vx}{\left|\vx\right|}\right):=
{-\frac{1}{4\pi}\left(\frac{\omega}{2\pi i c_0}\right)^{\frac{d-3}{2}}}
\int_{\Omega}e^{-i\frac{\omega}{c_{0}}\frac{\mathbf{x}\cdot \mathbf{y}}{\left|\mathbf{x}\right|}}F_{1}\left(\mathbf{y}\right)d\mathbf{y},
	\]
it is easy to see that \eqref{eq:V_pbm} satisfies the assumptions of Lemma~\ref{lem:spec_dec}.
Therefore, the {solution  $W_3\in C^6\left(\mathbb{R}^{d}\times\mathbb{R}_{+}\right)$ obeys the} following decay estimate, valid uniformly in $\vx\in\Omega$ for $t\geq0$:
\begin{align}\label{eq:V0_dec_estim}
\left|W_3\left(\mathbf{x},t\right)\right|+\left|\nabla W_3\left(\mathbf{x},t\right)\right|+\left|\partial_t W_3\left(\mathbf{x},t\right)\right|&\\
+\left|\Delta W_3\left(\mathbf{x},t\right)\right|+\left|\partial_{t}\nabla W_3\left(\mathbf{x},t\right)\right|+\left|\partial_{t}\Delta W_3\left(\mathbf{x},t\right)\right|&\leq \frac{C}{\left(1+t^2\right)^{1/2}},\nonumber
\end{align}
with some constant $C>0$.
In particular,
\begin{equation}\label{eq:V0_dec_estim2}
\left\Vert W_3\left(\cdot,t\right)\right\Vert_{H^1\left(\Omega\right)}+\left\Vert \partial_t W_3\left(\cdot,t\right)\right\Vert_{L^2\left(\Omega\right)}\leq \frac{C}{\left(1+t^2\right)^{1/2}}, \hspace{1em} t\geq0.
\end{equation}

\medskip
\underline{Decay of $W_4$.}
$W_4$ solves the inhomogeneous wave problem
	\begin{equation}\label{eq:wave_Z12}
	\begin{cases}
	\partial_{t}^{2}W_4\left(\mathbf{x},t\right)-\beta^{-1}\left(\mathbf{x}\right)\nabla\cdot\left(\alpha\left(\mathbf{x}\right)\nabla W_4\left(\mathbf{x},t\right)\right)=F_{2}\left(\mathbf{x},t\right)+F_{3}\left(\mathbf{x},t\right), \qquad\mathbf{x}\in\mathbb{R}^d, \hspace{1em} t>0,\\
	W_4\left(\mathbf{x},0\right)=0,\hspace{1em}\partial_{t}W_4\left(\mathbf{x},0\right)=0,\quad \mathbf{x}\in\mathbb{R}^d,
	\end{cases}
	\end{equation}
	where
	\begin{equation}
	F_{k}\left(\mathbf{x},t\right):=\beta^{-1}\left(\mathbf{x}\right)\nabla
        \alpha\left(\mathbf{x}\right)\cdot\nabla
        W_k\left(\mathbf{x},t\right)+\left(\beta^{-1}\left(\mathbf{x}\right)\alpha \left(\mathbf{x}\right)-c_0^2\right) \Delta W_k\left(\mathbf{x},t\right),\,k=2,3.
        \label{eq:F12_def}
	\end{equation}

{As already pointed out, the nonzero right-hand side compensates the fact that the equations in problems \eqref{eq:V12_pbm} and \eqref{eq:V_pbm} are different from the equation in problem \eqref{eq:wave_W}.}

        Estimates \eqref{eq:V12_dec_estim} and \eqref{eq:V0_dec_estim} entail the decay of {all
        the terms} 
      entering \eqref{eq:F12_def} and {of their} 
      time derivative. {Hence, recalling the regularity of $W_2$ and $W_3$, we see} that Proposition~\ref{prop:dec_source} is
      applicable {with $p=1$}. This gives {the unique solution $W_4\in C^{2}\left(\mathbb{R}_{+},L^2\left(\Omega\right)\right)\cap C^{1}\left(\mathbb{R}_{+},H^1\left(\Omega\right)\right)\cap C\left(\mathbb{R}_{+},H^2\left(\Omega\right)\right)$ which satisties}\\
for $d=2$:
\begin{equation}\label{eq:Z12_dec_estim_2D}
\left\Vert W_4\left(\cdot,t\right)\right\Vert_{H^1\left(\Omega\right)}+\left\Vert \partial_t W_4\left(\cdot,t\right)\right\Vert_{L^2\left(\Omega\right)}\leq C\frac{1+\log\left(1+t^2\right)}{\left(1+t^2\right)^{1/2}}, \hspace{1em} t\geq0,
\end{equation}
for $d=3$:
\begin{equation}\label{eq:Z12_dec_estim_3D}
\left\Vert W_4\left(\cdot,t\right)\right\Vert_{H^1\left(\Omega\right)}+\left\Vert \partial_t W_4\left(\cdot,t\right)\right\Vert_{L^2\left(\Omega\right)}\leq \frac{C}{\left(1+t^2\right)^{1/2}}, \hspace{1em} t\geq0,
\end{equation}
with some constant $C>0$.

Consequently, by combining \eqref{eq:W34_dec_estim}, \eqref{eq:V12_dec_estim2}, \eqref{eq:V0_dec_estim2}, 
\eqref{eq:Z12_dec_estim_2D}, and~\eqref{eq:Z12_dec_estim_3D}
with~\eqref{eq:W_decomp}, the 
estimates~\eqref{eq:W_dec_estim_2D} and~\eqref{eq:W_dec_estim_3D} follow. This concludes the proof of Theorem~\ref{thm:LAP}.

\section{Proofs of the auxiliary time decay results}
\label{sec:proofs}

\subsection{Proof of Proposition~\ref{prop:dec_free}}\label{sec:4.1}
	This proof is based on an application and an extension of a
        result from~\cite{Boucl-Burq}. We shall focus here only on
          the decay of the solution of~\eqref{eq:wave} with $f\equiv 0$. The existence, uniqueness and regularity results are standard. 	
Indeed, whenever $u_0\in H^{s+1}\left(\mathbb{R}^d\right)$ and $u_1\in
          H^{s}\left(\mathbb{R}^d\right)$ for $s\geq 1$,
a direct application of the result from
\cite[Ch. 6, Thm. 4.9]{Chaz-book} with $f\equiv 0$, together with a bootstrap
  argument for $\partial_t^2 u$, implies
that $u\in C^2\left(\mathbb{R}_{+},H^{s-1}\left(\mathbb{R}^d\right)\right)\cap C^1\left(\mathbb{R}_{+},H^{s}\left(\mathbb{R}^d\right)\right)\cap C\left(\mathbb{R}_{+},H^{s+1}\left(\mathbb{R}^d\right)\right)$ and 
\begin{equation}\label{eq:wave_eq_WP}
\left\Vert u\left(\cdot,t\right)\right\Vert _{H^{s+1}\left(\mathbb{R}^{d}\right)}^{2}+\left\Vert \partial_{t}u\left(\cdot,t\right)\right\Vert _{H^{s}\left(\mathbb{R}^{d}\right)}^{2}+\left\Vert \partial_{t}^{2}u\left(\cdot,t\right)\right\Vert _{H^{s-1}\left(\mathbb{R}^{d}\right)}^{2}\leq C\left(\left\Vert u_{0}\right\Vert _{H^{s+1}\left(\mathbb{R}^{d}\right)}^{2}+\left\Vert u_{1}\right\Vert _{H^{s}\left(\mathbb{R}^{d}\right)}^{2}\right)
\end{equation}	
for any $t>0$ and some constant $C>0$ that is uniform on any time interval $[0,T]$, $T>0$. 
In the present case, since $u_0$,
$u_1\in H^2\left(\mathbb{R}^{d}\right)$, we have~\eqref{eq:wave_eq_WP} with $s=1$. 

Because of {Assumption~\ref{assm:alph_bet_base} on
          $\alpha$ and $\beta$ (positivity and regularity),}
the operator
$P:=-\beta^{-1}\left(\vx\right)\nabla\cdot\left(\alpha\left(\vx\right)\nabla\,\right)$
 {with the domain $\text{Dom } P=H^2\left(\mathbb{R}^d\right)$}
        is self-adjoint in
        $L^2_\beta\left(\mathbb{R}^d\right)$ (the
        $L^2\left(\mathbb{R}^d\right)$ space endowed with the 
        $\beta$-weighted
        $L^2$ inner product).  Note that the sets $L^2_\beta(\mR^d)$ and $L^2(\mR^d)$ coincide since the weight $\beta$ is bounded and uniformly bounded away from zero.
        Moreover, $P$ is positive so that {there exists a unique
          self-adjoint, positive operator $B$ { with $\text{Dom } B=H^1\left(\mathbb{R}^d\right)$} such that $B^2=P$. {We refer to, e.g. \cite[Proof of Prop.~1.1]{paper-decay1D} for a more detailed discussion for the case $d=1$.}
        With the notation $\sqrt{P}:=B$ and $1/\sqrt{P}:=B^{-1}$,}
        we can formally write the solution of~\eqref{eq:wave} with $f\equiv 0$ as
	\begin{equation}\label{eq:wave_inv}
	u\left(\vx,t\right)=\cos\left(t\sqrt{P}\right)u_0\left(\vx\right)+\frac{\sin\left(t\sqrt{P}\right)}{\sqrt{P}}u_1\left(\vx\right), \hspace{1em} t\geq 0.
	\end{equation}
	Under {Assumptions~\ref{assm:alph_bet_base}
          and~\ref{assm:alph_bet_trap}
        on $\alpha$ and $\beta$ (compactly supported derivatives and nontrapping)},
the following operator-norm estimates are obtained in \cite[Thm. 1.5]{Boucl-Burq}. Namely, there exists a constant $C>0$ such that
	\begin{equation}\label{eq:wave_inv_estim1}
	\left\Vert q_\nu^{-1} \frac{\sin\left(t\sqrt{P}\right)}{\sqrt{P}} q_\nu^{-1} \right\Vert_{L^2 \left(\mathbb{R}^d\right)\rightarrow H^1\left(\mathbb{R}^d\right)}\leq \frac{C}{\left(1+t^2\right)^{\frac{d-1}{2}}}, \hspace{1em} t\geq 0,
	\end{equation}
	\begin{equation}\label{eq:wave_inv_estim2}
	 \left\Vert q_\nu^{-1} \cos\left(t\sqrt{P}\right) q_\nu^{-1} \right\Vert_{L^2 \left(\mathbb{R}^d\right)\rightarrow L^2 \left(\mathbb{R}^d\right)}\leq \frac{C}{\left(1+t^2\right)^\frac{d}{2}}, \hspace{1em} t\geq 0,
	\end{equation}
	where $q_\nu:=\left(1+\left|\vx\right|^2\right)^{\nu/2}$ with some $\nu>d+1$. 
	
Set $\mu:=d+1+\epsilon$. According to \eqref{eq:u0u1_loc_cond}, we
have $q_\mu u_0 
$ and $q_\mu u_1\in L^2\left(\mathbb{R}^d\right)$.
Then, we deduce from \eqref{eq:wave_inv}--\eqref{eq:wave_inv_estim2} that, for $t\geq 0$,
\begin{align}\label{eq:u_q_L2_estim}
\left\Vert q_{\mu}^{-1} u \left(\cdot,t\right)\right\Vert_{L^2\left(\mathbb{R}^d\right)}\leq& C\left(\frac{1}{\left(1+t^2\right)^\frac{d}{2}}\left\Vert q_\mu u_0\right\Vert_{L^2\left(\mathbb{R}^d\right)}+\frac{1}{\left(1+t^2\right)^\frac{d-1}{2}}\left\Vert q_\mu u_1\right\Vert_{L^2\left(\mathbb{R}^d\right)}\right)\\
\leq& \frac{C_0}{\left(1+t^2\right)^\frac{d-1}{2}}\nonumber
\end{align}
for some constant $C_0>0$. 
	
To obtain the estimate for the time derivative $\partial_t u$, we note that $w:=\partial_t u$ solves
$\partial_t^2 w+Pw=0$,
$w\left(\vx,0\right)=u_1\left(\vx\right)$, $\partial_t w\left(\vx,0\right)=-Pu_0\left(\vx\right)$. Hence, we have
\begin{equation*}
w\left(\cdot,t\right)=\cos\left(t\sqrt{P}\right)u_1-\frac{\sin\left(t\sqrt{P}\right)}{\sqrt{P}}\left(P u_0\right).
\end{equation*}
Therefore, using \eqref{eq:wave_inv_estim1} and~\eqref{eq:wave_inv_estim2}, 
we estimate, for $t\geq0$,
\begin{align}\label{eq:u_dt_q_L2_estim}
\left\Vert q_{\mu}^{-1}  \partial_t u \left(\cdot,t\right)\right\Vert_{L^2\left(\mathbb{R}^d\right)}\leq& C\left(\frac{1}{\left(1+t^2\right)^\frac{d}{2}}\left\Vert q_\mu u_1\right\Vert_{L^2\left(\mathbb{R}^d\right)}+\frac{1}{\left(1+t^2\right)^\frac{d-1}{2}}\left\Vert q_\mu Pu_0\right\Vert_{L^2\left(\mathbb{R}^d\right)}\right)\\
\leq& \frac{C_1}{\left(1+t^2\right)^\frac{d-1}{2}}\nonumber
\end{align}
for some constant $C_1>0$.

To {complete} the $H^1$-estimate of $u$, we estimate the $L^2$-norm of $\nabla u$. 
First, we observe that $\widetilde{w}:=\partial_t^2 u$ solves
$\partial_t^2 \widetilde{w}+P\widetilde{w}=0$,
$\widetilde{w}\left(\vx,0\right)=-Pu_0\left(\vx\right)$, $\partial_t \widetilde{w}\left(\vx,0\right)=-Pu_1\left(\vx\right)$. Hence, as before, we have, for $t\geq 0$,
\begin{equation*}
\widetilde{w}\left(\cdot,t\right)=-\cos\left(t\sqrt{P}\right)\left(P u_0\right)-\frac{\sin\left(t\sqrt{P}\right)}{\sqrt{P}}\left(P u_1\right),
\end{equation*}
\begin{equation*}
\left\Vert q_{\mu}^{-1} \partial_t^2 u \left(\cdot,t\right)\right\Vert_{L^2\left(\mathbb{R}^d\right)}\leq  C\left(\frac{1}{\left(1+t^2\right)^\frac{d}{2}}\left\Vert q_\mu P u_0\right\Vert_{L^2\left(\mathbb{R}^d\right)}+\frac{1}{\left(1+t^2\right)^\frac{d-1}{2}}\left\Vert q_\mu Pu_1\right\Vert_{L^2\left(\mathbb{R}^d\right)}\right).
\end{equation*}
We thus arrive at
\begin{equation}\label{eq:Pu_q_L2_estim}
\left\Vert q_{\mu}^{-1} P u \left(\cdot,t\right)\right\Vert_{L^2\left(\mathbb{R}^d\right)}=\left\Vert q_{\mu}^{-1} \partial_t^2 u \left(\cdot,t\right)\right\Vert_{L^2\left(\mathbb{R}^d\right)}\leq \frac{C_2}{\left(1+t^2\right)^\frac{d-1}{2}}
\end{equation}
for some constant $C_2>0$. Employing the notation $\overline{\left(\,\cdot\,\right)}$ for the complex conjugate, we consider the following inner product on $L^2_\beta\left(\mathbb{R}^d\right)${:
\begin{equation}\label{eq:Pu_u}
\left\langle q_\mu^{-1} Pu, q_\mu^{-1}u\right\rangle _{L_{\beta}^{2}\left(\mathbb{R}^{d}\right)} 
= -\int_{\mathbb{R}^{d}}\nabla\cdot\Big(\alpha\left(\vx\right)\nabla u\left(\vx,t\right)\Big)\overline{u\left(\vx,t\right)}q_\mu^{-2}\left(\vx\right)d\vx.
\end{equation}
}
{In order to elaborate this expression further, we resort to an approximation argument. Let $t>0$ be fixed. By density of $\mathcal C_0^\infty(\mR^d)$ in
  $H^2(\mR^d)$, for any $n\in\mathbb{N}$, there exists
  $u_n\left(\cdot,t\right)$ such that $\left\Vert
    u\left(\cdot,t\right)-u_n\left(\cdot,t\right)\right\Vert_{H^2(\mR^d)}\le
  1/n$. For $u_n\left(\cdot,t\right)\in \mathcal C_0^\infty(\mR^d)$,
  the divergence theorem gives
%
\[
\int_{\mathbb{R}^{d}}\nabla\cdot\left(\alpha\left(\vx\right)\nabla u_{n}\left(\vx,t\right)\right)\overline{u\left(\vx,t\right)}q_{\mu}^{-2}\left(\vx\right)d\vx=-\int_{\mathbb{R}^{d}}\alpha\left(\vx\right)\nabla u_{n}\left(\vx,t\right)\cdot\nabla\left(\overline{u\left(\vx,t\right)}q_{\mu}^{-2}\left(\vx\right)\right)d\vx.
\]
Therefore, adding and subtracting $u_n\left(\cdot,t\right)$ to $u\left(\cdot,t\right)$ under the divergence sign in the right-hand side of~\eqref{eq:Pu_u} gives
\begin{align*}
\left\langle q_{\mu}^{-1}Pu,q_{\mu}^{-1}u\right\rangle
  _{L_{\beta}^{2}\left(\mathbb{R}^{d}\right)}=
&
                                                 -\int_{\mathbb{R}^{d}}\nabla\cdot\left(\alpha\left(\vx\right)\nabla\left(u\left(\vx,t\right)-u_{n}\left(\vx,t\right)\right)\right)\overline{u\left(\vx,t\right)}q_{\mu}^{-2}\left(\vx\right)d\vx\\
& -\int_{\mathbb{R}^{d}}\alpha\left(\vx\right)\nabla\left(u\left(\vx,t\right)-u_{n}\left(\vx,t\right)\right)\cdot\nabla\left(\overline{u\left(\vx,t\right)}q_{\mu}^{-2}\left(\vx\right)\right)d\vx\\                                                 
& +\int_{\mathbb{R}^{d}}\alpha\left(\vx\right)\nabla u\left(\vx,t\right)\cdot\nabla\left(\overline{u\left(\vx,t\right)}q_{\mu}^{-2}\left(\vx\right)\right)d\vx.\\
\end{align*}
The absolute values of the terms on the first and second lines are estimated by multiples of
$\left\Vert
  u\left(\cdot,t\right)-u_n\left(\cdot,t\right)\right\Vert_{H^2(\mR^d)}\left\Vert
  u\left(\cdot,t\right)\right\Vert_{H^1(\mR^d)}$. Therefore, 
they vanish in the limit as $n\rightarrow\infty$. Consequently, we obtain
\begin{align*}
\left\langle q_{\mu}^{-1}Pu,q_{\mu}^{-1}u\right\rangle _{L_{\beta}^{2}\left(\mathbb{R}^{d}\right)}= & \int_{\mathbb{R}^{d}}\alpha\left(\vx\right)\left|\nabla u\left(\vx,t\right)\right|^{2}q_{\mu}^{-2}\left(\vx\right)d\vx\\
 & +\int_{\mathbb{R}^{d}}\alpha\left(\vx\right)\overline{u\left(\vx,t\right)}\nabla u\left(\vx,t\right)\cdot\nabla q_{\mu}^{-2}\left(\vx\right)d\vx.
\end{align*}
}
{Rearranging the terms and employing}
\[
\left(\nabla u\left(\vx,t\right)\cdot\nabla
  q_{\mu}^{-2}\left(\vx\right)\right)=-2\mu\frac{\left|\vx\right|}{1+\left|\vx\right|^{2}}
{q_{\mu}^{-2}} 
\left(\vx\right)\partial_{\left|\vx\right|}u\left(\vx,t\right),
\]
{we arrive at}
\begin{align}\label{eq:Pu_u_estim1}
\alpha_{\min}\left\Vert q_{\mu}^{-1}\nabla u\right\Vert _{L^{2}\left(\mathbb{R}^{d}\right)}^{2} & \leq\int_{\mathbb{R}^{d}}\alpha\left(\vx\right)\left|\nabla u\left(\vx,t\right)\right|^{2}q_{\mu}^{-2}\left(\vx\right)d\vx\\
 & \leq\left|\left\langle q_{\mu}^{-1}Pu,q_{\mu}^{-1}u\right\rangle
   _{L_{\beta}^{2}\left(\mathbb{R}^{d}\right)}\right|+
   \mu\left\Vert \alpha\right\Vert _{L^{\infty}\left(\mathbb{R}^{d}\right)}\left|\left\langle q_{\mu}^{-1}\partial_{\left|\vx\right|}u,q_{\mu}^{-1}u\right\rangle _{L^{2}\left(\mathbb{R}^{d}\right)}\right|.\nonumber
\end{align}
Furthermore, employing the Cauchy-Schwarz inequality, we can estimate 
\[
\left|\left\langle q_{\mu}^{-1}Pu,q_{\mu}^{-1}u\right\rangle _{L_{\beta}^{2}\left(\mathbb{R}^{d}\right)}\right|\leq\left\Vert \beta\right\Vert _{L^{\infty}\left(\mathbb{R}^{d}\right)}\left\Vert q_{\mu}^{-1}Pu\right\Vert _{L^{2}\left(\mathbb{R}^{d}\right)}\left\Vert q_{\mu}^{-1}u\right\Vert _{L^{2}\left(\mathbb{R}^{d}\right)},
\]
\begin{align*}
\left|\left\langle q_{\mu}^{-1}\partial_{\left|\vx\right|}u,q_{\mu}^{-1}u\right\rangle _{L^{2}\left(\mathbb{R}^{d}\right)}\right| & \leq\left\Vert q_{\mu}^{-1}\nabla u\right\Vert _{L^{2}\left(\mathbb{R}^{d}\right)}\left\Vert q_{\mu}^{-1}u\right\Vert _{L^{2}\left(\mathbb{R}^{d}\right)}\\
& \leq\frac{\alpha_{\min}}{4\mu\left\Vert \alpha\right\Vert _{L^{\infty}\left(\mathbb{R}^{d}\right)}}\left\Vert q_{\mu}^{-1}\nabla u\right\Vert _{L^{2}\left(\mathbb{R}^{d}\right)}^{2}+\frac{\mu\left\Vert \alpha\right\Vert _{L^{\infty}\left(\mathbb{R}^{d}\right)}}{\alpha_{\min}}\left\Vert q_{\mu}^{-1}u\right\Vert _{L^{2}\left(\mathbb{R}^{d}\right)}^{2}.
\end{align*}
Here, on the second line, we used the elementary inequality
$|a|\,|b|\le \frac{{\delta_0}}{2} |a|^2+\frac{1}{2{\delta_0}}|b|^2$,
valid for any ${\delta_0}>0$.
Therefore, estimate \eqref{eq:Pu_u_estim1} entails
\begin{align*}
 {\frac34\, \alpha_{\min}} 
  \left\Vert q_{\mu}^{-1}\nabla u\right\Vert _{L^{2}\left(\mathbb{R}^{d}\right)}^{2}\leq&\left\Vert \beta\right\Vert _{L^{\infty}\left(\mathbb{R}^{d}\right)}\left\Vert q_{\mu}^{-1}Pu\right\Vert _{L^{2}\left(\mathbb{R}^{d}\right)}\left\Vert q_{\mu}^{-1}u\right\Vert _{L^{2}\left(\mathbb{R}^{d}\right)}\\
&+{\frac{\mu^{2}}{\alpha_{\min}}} 
\left\Vert\alpha\right\Vert _{L^{\infty}\left(\mathbb{R}^{d}\right)}^{2}\left\Vert q_{\mu}^{-1}u\right\Vert _{L^{2}\left(\mathbb{R}^{d}\right)}^{2}\nonumber.
\end{align*}
Recalling \eqref{eq:u_q_L2_estim} and~\eqref{eq:Pu_q_L2_estim}, this leads to
\begin{equation}\label{eq:du_q_L2_estim}
\left\Vert q_{\mu}^{-1}\nabla u\right\Vert
_{L^{2}\left(\mathbb{R}^{d}\right)}^{2}\leq
{\frac43}\left(
  {\frac{C_{0}C_{2}}{\alpha_{\min}}} 
  \left\Vert \beta\right\Vert
  _{L^{\infty}\left(\mathbb{R}^{d}\right)}+
 {\frac{\mu^{2}C_{0}^{2}}{\alpha_{\min}^{2}}} 
  \left\Vert \alpha\right\Vert _{L^{\infty}\left(\mathbb{R}^{d}\right)}^{2}\right)\frac{1}{\left(1+t^{2}\right)^{d-1}}.
\end{equation}
Finally, denoting with $\chi_\Omega$ the characteristic function of the bounded set $\Omega$, we have 
\[
\left\Vert u\right\Vert _{L^{2}\left(\Omega\right)}=\left\Vert u\chi_{\Omega}\right\Vert _{L^{2}\left(\mathbb{R}^{d}\right)}\leq C_{\Omega,\mu}\left\Vert q_{\mu}^{-1}u\right\Vert _{L^{2}\left(\mathbb{R}^{d}\right)},
\]
and similarly for $\nabla u$ and $\partial_t u$. Hence, the estimates
\eqref{eq:u_q_L2_estim}, \eqref{eq:u_dt_q_L2_estim}, and
\eqref{eq:du_q_L2_estim} furnish~\eqref{eq:dec_IC_gen}.
\qed

\subsection{Proof of Proposition~\ref{prop:dec_source}}
The existence, uniqueness, and regularity results are standard. 
{Indeed, since $f\in C\left(\mathbb{R}_{+},H^1\left(\mathbb{R}^d\right)\right)$, the result from \cite[Ch. 6, Thm. 4.9]{Chaz-book} implies (directly, and by estimating $\partial_t^2 u$ from the wave equation \eqref{eq:wave}) that $u\in C^2\left(\mathbb{R}_{+},L^{2}\left(\mathbb{R}^d\right)\right)\cap C^1\left(\mathbb{R}_{+},H^{1}\left(\mathbb{R}^d\right)\right)\cap C\left(\mathbb{R}_{+},H^{2}\left(\mathbb{R}^d\right)\right)$.}
 We shall focus here only on the decay of the solution.
	Without loss of generality, we can take $\Omega=\Omega_f$ (by
        enlarging both sets if necessary).
{Let $P$, $\sqrt{P}$, and $1/\sqrt{P}$ be defined as at the
  beginning of Section~\ref{sec:4.1}.}
%
The following operator-norm estimate was obtained in~\cite[Thm. 1.5]{Boucl-Burq}:
	\begin{equation}\label{eq:wave_inv_estim0}
	\left\Vert \frac{\sin\left(t\sqrt{P}\right)}{\sqrt{P}}\chi_\Omega\right\Vert_{L^2 \left(\mathbb{R}^d\right)\rightarrow H^1\left(\Omega\right)}\leq \frac{C_0}{\left(1+t^2\right)^\frac{d-1}{2}},\hspace{1em} t\geq 0, 
	\end{equation}
for some $C_0>0$, where $\chi_\Omega$ denotes the characteristic function of the set $\Omega$.

According to the Duhamel principle, the solution to \eqref{eq:wave}
{with $u_0\equiv 0$, $u_1\equiv 0$} can be written as
\begin{equation}\label{eq:u_Duham}
u\left(\cdot,t\right)=\int_0^t \frac{\sin\left(\left(t-\tau\right)\sqrt{P}\right)}{\sqrt{P}} f\left(\cdot,\tau\right) d\tau.
\end{equation}
Using a basic Bochner integral estimate in $H^1\left(\Omega\right)$ and \eqref{eq:wave_inv_estim0}, we obtain, for $t>0$,
\begin{align}\label{eq:u_H1_estim1}
\left\Vert u\left(\cdot,t\right)\right\Vert _{H^{1}\left(\Omega\right)}&\leq\int_{0}^{t}\left\Vert \frac{\sin\left(\left(t-\tau\right)\sqrt{P}\right)}{\sqrt{P}}f\left(\cdot,\tau\right)\right\Vert _{H^{1}\left(\Omega\right)}d\tau\\
&\leq \int_{0}^{t} \frac{C_0}{\left(1+\left(t-\tau\right)^2\right)^{\frac{d-1}{2}}} \left\Vert f\left(\cdot,\tau\right)\right\Vert_{L^2\left(\Omega\right)} d\tau,\nonumber
\end{align}
where, in the second line, we also took into account the assumption that
the support of $f\left(\cdot,\tau\right)$ is contained in $\Omega$ for
each $\tau>0$.

Employing the assumed estimate{~\eqref{eq:assf} on $f$, namely} $\left\Vert f\left(\cdot,\tau\right)\right\Vert_{L^2\left(\Omega\right)}\leq {C_f}\slash\left(1+\tau^2\right)^{p/2}$ for some {constants $C_f,\,p>0$} and all $\tau>0$, and denoting $C:=C_0{C_f}$, we proceed to estimate
\begin{align}\label{eq:u_H1_estim2}
\left\Vert u\left(\cdot,t\right)\right\Vert _{H^{1}\left(\Omega\right)}\leq&\int_{0}^{t}\frac{C\,d\tilde{\tau}}{\left(1+\left(t-\tilde{\tau}\right)^2\right)^\frac{d-1}{2}\left(1+\tilde{\tau}^{2}\right)^{\frac{p}{2}}}\\
=&\frac{C}{t^{d+p-2}}\left[\int_{0}^{1/2}\frac{d\tau}{\left(1/t^2+\left({1}-\tau\right)^2\right)^\frac{d-1}{2}\left(1/t^2+\tau^{2}\right)^\frac{p}{2}}\right.\nonumber\\
&+\left.\int_{1/2}^{1}\frac{d\tau}{\left(1/t^2+\left({1}-\tau\right)^2\right)^\frac{d-1}{2}\left(1/t^2+\tau^{2}\right)^\frac{p}{2}}\right]\nonumber\\
\leq&\frac{2^{d-1}C}{t^{p-1}}\frac{1}{\left(1+t^2\right)^\frac{d-1}{2}}\int_{0}^{1/2}\frac{d\tau}{\left(1/t^2+\tau^{2}\right)^\frac{p}{2}}\nonumber\\
&+\frac{2^{p}C}{t^{d-2}}\frac{1}{\left(1+t^2\right)^\frac{p}{2}}\int_{0}^{1/2}\frac{d\tau}{\left(1/t^2+\tau^2\right)^\frac{d-1}{2}}\nonumber.
\end{align}
Here we used the change of variable $\tilde{\tau}\mapsto
\tau:=\tilde{\tau}/t$ and employed the estimates 
{
\begin{align*}
\frac{1/t^{d+p-2}}{\left(1/t^2+\left({1}-\tau\right)^2\right)^\frac{d-1}{2}}
  \le
  \frac{1/t^{d+p-2}}{(1/t^2+1/4) ^\frac{d-1}{2}}
  =\frac{2^{d-1}}{t^{p-1}}\frac{1}{(4+t^2) ^\frac{d-1}{2}}
  \le \frac{2^{d-1}}{t^{p-1}}\frac{1}{(1+t^2) ^\frac{d-1}{2}},\\
 \hfill   \hspace{2.5em}0\leq\tau\leq\frac{1}{2},\,t\geq0,
\end{align*}
\begin{align*}
  \frac{1/t^{d+p-2}}{\left(1/t^2+\tau^{2}\right)^\frac{p}{2}}
  \le
  \frac{1/t^{d+p-2}}{(1/t^2+1/4) ^\frac{p}{2}}
  =\frac{2^{p}}{t^{d-2}}\frac{1}{(4+t^2) ^\frac{p}{2}}
  \le \frac{2^{p}}{t^{d-2}}\frac{1}{(1+t^2) ^\frac{p}{2}},\\
  \hfill  \hspace{2.5em}\frac{1}{2}\leq\tau\leq 1,\,t\geq0,
\end{align*}
}
in the integrals over $\left[0,1/2\right]$ and $\left[1/2,1\right]$,
respectively. In the last line {of~\eqref{eq:u_H1_estim2},} we have also made the change of variable $\tau\mapsto 1-\tau$. 
Using Lemma \ref{lem:app_alg_int_estim}, we continue estimate \eqref{eq:u_H1_estim2}:
\begin{align}\label{eq:u_H1_estim3}
\left\Vert u\left(\cdot,t\right)\right\Vert _{H^{1}\left(\Omega\right)}\leq& \frac{2^q C}{\left(1+t^2\right)^{\frac{d-1}{2}}}\begin{cases}
C_{1,p}t^{1-p},& 0<p<1,\\
\log\left(t+\sqrt{1+t^2}\right),& p=1,\\
C_{2,p},& p>1,
\end{cases}\\
&+\frac{2^q C}{\left(1+t^2\right)^{\frac{p}{2}}}\begin{cases}
\log\left(t+\sqrt{1+t^2}\right),& d=2,\\
C_{2,{d-1}},& d>2 \nonumber,
\end{cases}
\end{align}
where $q:=\max\left(d-1,p\right)$, $C_{1,s}:=\dfrac{1}{1-s}$, $C_{2,s}:=\displaystyle{\int_0^\infty\frac{dz}{\left(1+z^2\right)^{s/2}}}$. 
We continue by considering separately the cases $d=2$ and $d>2$.

Since $C_{1,p}>1$, estimate \eqref{eq:u_H1_estim3} for $d=2$ reads
\begin{align}\label{eq:u_H1_estim4}
\left\Vert u\left(\cdot,t\right)\right\Vert _{H^{1}\left(\Omega\right)}\leq& 2^q C
\begin{cases}
C_{1,p}\left(1+C_{3,p}\right)\dfrac{1+\log(1+t^2)}{\left(1+t^2\right)^{\frac{p}{2}}},& 0<p<1,\\
2\dfrac{\log\left(t+\sqrt{1+t^2}\right)}{\left(1+t^2\right)^\frac{1}{2}},& p=1,\\
\left(1+C_{4,p}\right)\max\left(C_{2,p},1\right)\dfrac{1}{\left(1+t^2\right)^{\frac{1}{2}}},& p>1,
\end{cases}\\
\leq&\widetilde{C}_{p} \begin{cases}
\dfrac{1+\log(1+t^2)}{\left(1+t^2\right)^{\frac{p}{2}}},& 0<p\leq 1,\\
\dfrac{1}{\left(1+t^2\right)^{\frac{1}{2}}},& p>1,
\end{cases}\nonumber
\end{align}
where
  $C_{3,p}:=\displaystyle{\sup_{t\geq
        0}\,}\dfrac{t^{1-p}}{\left(1+t^2\right)^{\frac{1-p}{2}}[1+\log(1+t^2)]}$,
  $C_{4,p}:=\displaystyle{\sup_{t\geq
      0}\,}\dfrac{\log\left(t+\sqrt{1+t^2}\right)}{\left(1+t^2\right)^{\frac{p-1}{2}}}$.
In \eqref{eq:u_H1_estim4} we also used the elementary estimate
$$
  \log(t+\sqrt{1+t^2}) \le \log 2+\frac12\log(1+t^2) < 1 + \log(1+t^2),\quad t\ge0.
$$

Similarly, when $d>2$, we have
\begin{align}\label{eq:u_H1_estim5}
\left\Vert u\left(\cdot,t\right)\right\Vert _{H^{1}\left(\Omega\right)}\leq& 2^q C
\begin{cases}
\left(1+C_{5,d,p}\right)\max\left(C_{1,p},C_{2,d-1}\right)\dfrac{1}{\left(1+t^2\right)^{\frac{p}{2}}},& 0<p<1,\\
\left(1+C_{4,d-1}\right)\max\left(C_{2,d-1},1\right)\dfrac{1}{\left(1+t^2\right)^\frac{1}{2}},& p=1,\\
2C_{2,r}\dfrac{1}{\left(1+t^2\right)^{\frac{r}{2}}},& p>1,
\end{cases}\\
\leq&\widetilde{C}_{p,d} \begin{cases}
\dfrac{1}{\left(1+t^2\right)^{\frac{p}{2}}},& 0<p\leq 1,\\
\dfrac{1}{\left(1+t^2\right)^{\frac{r}{2}}},& p>1,
\end{cases}\nonumber
\end{align}
where $r:=\min\left(d-1,p\right)$, $C_{5,d,p}:=\displaystyle{\sup_{t\geq 0}\,}\dfrac{t^{1-p}}{\left(1+t^2\right)^{\frac{d-1-p}{2}}}$.
This completes the estimate of $\left\Vert u\left(\cdot,t\right)\right\Vert _{H^{1}\left(\Omega\right)}$.

To finish the proof, it remains to obtain the estimate for the time derivative $\partial_t u$. To this effect, we note that $w:=\partial_tu$ solves
$\partial_t^2 w+Pw=\partial_t f$,
$w\left(\vx,0\right)=0$, $\partial_t w\left(\vx,0\right)=f\left(\vx,0\right)$.
Hence, we have
\begin{equation*}\label{ut-H1}
  \partial_tu\left(\cdot,t\right)=
  w\left(\cdot,t\right)=\frac{\sin\left(t\sqrt{P}\right)}{\sqrt{P}}f\left(\cdot,0\right)+\int_0^t \frac{\sin\left(\left(t-\tau\right)\sqrt{P}\right)}{\sqrt{P}} \partial_t f\left(\cdot,\tau\right) d\tau,
\end{equation*}
and consequently we obtain from~\eqref{eq:wave_inv_estim0}, 
again with $C=C_0C_f$,
\begin{equation}\label{eq:u_dt_H1_estim1}
\left\Vert \partial_t u\left(\cdot,t\right)\right\Vert _{H^{1}\left(\Omega\right)}\leq\frac{C}{\left(1+t^2\right)^\frac{d-1}{2}}+\int_{0}^{t}\left\Vert \frac{\sin\left(\left(t-\tau\right)\sqrt{P}\right)}{\sqrt{P}}\partial_t f\left(\cdot,\tau\right)\right\Vert _{H^{1}\left(\Omega\right)}d\tau.
\end{equation}
Therefore, {owing to~\eqref{eq:assf},} the estimate for $\partial_t
u$ can be obtained from the estimates for $u$ given in~\eqref{eq:u_H1_estim4} and~\eqref{eq:u_H1_estim5} by only adding an extra term, which is the first term on the right-hand side of \eqref{eq:u_dt_H1_estim1}. Namely, we have, for $d=2$, 
\begin{align}\label{eq:u_dt_H1_estim2}
\left\Vert \partial_t u\left(\cdot,t\right)\right\Vert _{H^{1}\left(\Omega\right)}\leq&\begin{cases}
\max\left(C,\widetilde{C}_p\right)\left[\dfrac{1}{\left(1+t^2\right)^\frac{1}{2}}+\dfrac{1+\log(1+t^2)}{\left(1+t^2\right)^\frac{p}{2}}\right],& 0<p\leq1,\\
\left(C+\widetilde{C}_p\right)\dfrac{1}{\left(1+t^2\right)^\frac{1}{2}},& p>1,
\end{cases}\\
\leq&2\max\left(C,\widetilde{C}_p\right)\begin{cases}
\dfrac{1+\log\left(1+t^2\right)}{\left(1+t^2\right)^\frac{p}{2}},& 0<p\leq1,\\
\dfrac{1}{\left(1+t^2\right)^\frac{1}{2}},& p>1.
\end{cases}\nonumber
\end{align}
Since the first term of the right-hand side of \eqref{eq:u_dt_H1_estim1} decays at least as fast as the second, we have, for $d>2$, 
\begin{align}\label{eq:u_dt_H1_estim3}
\left\Vert \partial_t u\left(\cdot,t\right)\right\Vert _{H^{1}\left(\Omega\right)}\leq& \widehat{C}_{p,d}\begin{cases}
\dfrac{1}{\left(1+t^2\right)^\frac{p}{2}},& 0<p\leq1,\\
\dfrac{1}{\left(1+t^2\right)^\frac{r}{2}},& p>1.
\end{cases}
\end{align}

Altogether, when $d=2$, estimates \eqref{eq:u_dt_H1_estim2} and \eqref{eq:u_H1_estim4} imply \eqref{eq:dec_source_2D}. Analogously, for $d>2$, estimates \eqref{eq:u_dt_H1_estim3} and \eqref{eq:u_H1_estim5} furnish \eqref{eq:dec_source_gen}. 
\qed

\subsection{Proof of Lemma~\ref{lem:gen_dec}}
{Since $v_0\in C^6(\mR^d)$, $v_1\in C^5(\mR^d)$, }
Theorems 2 and 3 in \cite[Par. 2.4.1]{Evans} {applied to $u$ and its derivatives} {(see also \eqref{eq:v_sol_v0v1} and \eqref{eq:v_sol_Kirch_2})}
  imply that {the regularity of} the solution~$v$ to~\eqref{eq:wave_c0} is 
 {$v\in C^{5}\left(\mR^d\times\mR_+\right)$\, and 
  $\partial_t v\in C^{4}\left(\mR^d\times\mR_+\right)$}.

In the main body of the proof, we shall prove that the bound
\begin{equation}\label{eq:v_dec2show}
\left|v\left(\vx,t\right)\right|\leq \frac{C}{\left(1+t^2\right)^{1/2}}, \hspace{1em}\vx\in\Omega,  \hspace{1em}t\geq 0,
\end{equation}
is valid for some constant $C>0$, assuming that $v_0\equiv 0$.

This first part of the proof actually holds true for weaker regularity assumptions than made in \eqref{eq:gen_ICs}-\eqref{eq:assumption3D}. More precisely, for $d=2$, we only need $A_0\in C^1(\mR_+)$, $Y_1\in C^1(\mathbb{S}^1)$, $V_1\in  C(\mR^2)$ with $|\vx|^{5/2}|V_1(\vx)|\le C$, $\vx\in\mR^2$. And, for $d=3$, we only need $v_1\in C(\mR^3)$ with $|\vx|^2|v_1(\vx)|\le C$, $\vx\in\mR^3$. In \eqref{eq:set-A} below, we shall summarize these reduced assumptions by saying that $v_1\in \mathcal A$.

{The case $v_0\not\equiv 0$ and the estimate of the other terms
  in~\eqref{eq:sol_gen_bnds} is discussed in the final part of this
  proof.} We consider {now} separately the cases $d=2$ and $d=3$. 

\medskip
\noindent
{\bf$\bullet$  Case $d=2$.}

The solution of \eqref{eq:wave_c0} with $v_0\equiv0$ is given by Poisson's formula \cite[Par. 2.4.1 (c)]{Evans}
\begin{equation}
\label{eq:v_sol_v1}
v\left(\vx,t\right)=\frac{t}{2\pi}\int_{0}^{1}\frac{r}{\left(1-r^{2}\right)^{1/2}}\int_{\left|\vs\right|=1}v_{1}\left(\vx+\vs rc_{0}t\right)d\sigma_{\vs}dr,
\end{equation}
where $d\sigma_{\vs}$ denotes the surface measure of the unit circle $\mathbb{S}^{1}$.
Introducing $\rho:=\left|\vx+\vs rc_{0}t\right|$, $\phi:=\frac{\vx+\vs rc_{0}t}{\left|\vx+\vs rc_{0}t\right|}$, and using \eqref{eq:gen_ICs}, we can write
\begin{align}\label{eq:v_sol_v1_decomp}
v\left(\vx,t\right)=&\frac{t}{2\pi}\int_{0}^{1}\frac{r}{\left(1-r^{2}\right)^{1/2}}\int_{\left|\vs\right|=1}\left[A_0(\rho)Y_1\left(\phi\right)+V_1\left(\vx+\vs rc_{0}t\right)\right]d\sigma_{\vs}dr\\
=:&P\left(\vx,t\right)+Q\left(\vx,t\right).\nonumber
\end{align}
{As $A_0\left(\rho\right)\equiv 0$ for $\rho\leq\rho_0$, then $\phi$ is well-defined whenever it appears {in the above integral}. In fact, $\rho=\left|\vx+\vs rc_{0}t\right|>\rho_0$ whenever the coefficient $A_0\left(\rho\right)$ in front of~$Y_1\left(\phi\right)$ in~\eqref{eq:v_sol_v1_decomp} is different from zero.}
%

We shall prove that there exists some $t_0>0$ such that the bounds
\begin{equation}\label{eq:boundQP}
    |P\left(\vx,t\right)|\le \frac{\widetilde{C}}{t},\qquad
    |Q\left(\vx,t\right)|\le \frac{\widetilde{C}_0}{t}    
\end{equation}
are valid uniformly in $\vx\in\Omega$ with some constants $\widetilde{C}$, $\widetilde{C}_0>0$  for any $t\ge t_0$.
Since it is evident from~\eqref{eq:v_sol_v1} that the solution $v$ is
bounded for any finite $t\geq 0$, \eqref{eq:v_sol_v1_decomp} and the estimates 
in~\eqref{eq:boundQP} will imply~\eqref{eq:v_dec2show}.

\bigskip 
\noindent
\underline{Estimate of $Q$ for $t\ge t_0$}:

We have
\begin{equation}\label{eq:Q_def}
Q\left(\vx,t\right)=\frac{t}{2\pi}\int_{a_1/t}^{1}\frac{r}{\left(1-r^{2}\right)^{1/2}}\int_{\left|\vs\right|=1}V_1\left(\vx+\vs rc_{0}t\right)d\sigma_{\vs}dr.
\end{equation}
Since $V_1\left(\vx\right)\equiv 0$ for $\left|\vx\right|\leq\rho_0$, we reduced here the integration range in the $r$ variable from $\left(0,1\right)$ to $\left(a_1/t,1\right)$ with
\begin{equation}\label{eq:a1_def}
a_1:=\frac{1}{c_0}\underset{\left|\vs\right|=1,\,\vx\in\Omega}{\text{inf }}\left[\sqrt{\left(\mathbf{x}\cdot\mathbf{s}\right)^{2}+\rho_{0}^{2}-\left|\vx\right|^{2}}-\mathbf{x}\cdot\mathbf{s}\right],
\end{equation}
{which is positive, due to $\Omega\Subset\mathbb{B}_{\rho_0}$.
To justify this reduction of the integration range, we need to prove the implication~$r<a_1/t\ \Rightarrow \, \rho<\rho_0$, from which~$V_1\equiv 0$ follows.} 
{The condition~$r<a_1/t$ means that for }{$\vx\in\Omega$, $\vs\in\mathbb{S}^1$, we have} 
$$
rc_0t+\vx\cdot\vs < \sqrt{\left(\mathbf{x}\cdot\mathbf{s}\right)^{2}+\rho_{0}^{2}-\left|\vx\right|^{2}}.
$$
For $rc_0t+\vx\cdot\vs\ge0$, this is equivalent to
$$
\rho^2=|\vx|^2+2\vx\cdot\vs rc_0t+(rc_0t)^2<\rho_0^2,
$$
and $rc_0t+\vx\cdot\vs<0$ yields directly
$$
  \rho^2=|\vx|^2+2\vx\cdot\vs rc_0t+(rc_0t)^2\le |\vx|^2-(rc_0t)^2\le |\vx|^2< \rho_0^2.
$$

Here and in the sequel we assume that $a_1/t\le1$, i.e. that $t_0\ge a_1$.
{By rearranging the factors, we can write
\begin{equation}\label{eq:auxQ}
    Q\left(\vx,t\right)=\frac{1}{2\pi}\int_{a_1/t}^{1}
    \frac{(rt)^{-3/2}}{\left(1-r^{2}\right)^{1/2}}
    \int_{\left|\vs\right|=1}
    \left(\frac{rt}{\rho}\right)^{5/2}\rho^{5/2}V_1\left(\vx+\vs rc_{0}t\right)d\sigma_{\vs}dr.
\end{equation}
{We can assume $\rho=\left|\vx+\vs rc_{0}\right|>\rho_0$ since the integrand vanishes otherwise due to the support of $V_1$.} Thus, we can estimate
  $rt/\rho$ in~\eqref{eq:auxQ} as follows. 
  From the triangle inequality
\begin{equation}\label{triangle}
    rc_0t=|\vx+\vs rc_{0}t-\vx|\le \rho+|\vx|\quad 
\end{equation}
and $\vx\in\Omega\Subset\mathbb{B}_{\rho_0}$, $\rho>\rho_0$, we have
    
\begin{equation}\label{eq:rt_rho_bnd}
  \dfrac{rt}{\rho}\le\dfrac{1}{c_0}\left(1+\dfrac{|\vx|}{\rho}\right)\le\dfrac{2}{c_0}.  
\end{equation}
Moreover, assumption~\eqref{eq:V0V1_dec} implies 
  $\rho^{5/2}|V_1\left(\vx+\vs rc_{0}t\right)|\le
  C_0$. Using this and~\eqref{eq:rt_rho_bnd}, \eqref{eq:auxQ} gives
  \[
    \left|Q\left(\vx,t\right)\right|\leq
    \frac{2^{5/2}C_0}{c_0^{5/2}t^{3/2}}
    \int_{a_1/t}^{1}\dfrac{dr}{\left(1-r^{2}\right)^{1/2}r^{3/2}}.
  \]
If we choose $t_0:=2a_1$, we obtain for $t\ge t_0$:
}
%
\begin{align}\label{eq:Q_estim1}
  \left|Q\left(\vx,t\right)\right|\leq&
{\frac{2^{5/2}C_0}{c_0^{5/2}t^{3/2}}}  
\int_{a_1/t}^{1}\frac{dr}{\left(1-r^{2}\right)^{1/2}r^{3/2}}=
{\frac{2^{5/2}C_0}{c_0^{5/2}t^{3/2}}}\left(\int_{a_1/t}^{1/2}\ldots+\int_{1/2}^1\ldots\right)\\
  \leq& {\frac{2^{7/2}C_0}{\sqrt{3}c_0^{5/2}t^{3/2}}}
        \int_{a_1/t}^{1/2}\frac{dr}{r^{3/2}}+
    {\frac{2^{5/2}C_0}{c_0^{5/2}t^{3/2}}}    
        \int_{1/2}^1\frac{dr}{\left(1-r^{2}\right)^{1/2}r^{3/2}}
\leq \frac{\widetilde{C}_0}{t}\nonumber
\end{align}
for some constant $\widetilde{C}_0>0$. This completes the proof
  of the estimate of $Q$ in~\eqref{eq:boundQP} with $t_0=2a_1$.

\bigskip 
\noindent\underline{{Estimate of $P$ for $t\ge t_0$}}:
	

In order to prove the estimate of $P$ in~\eqref{eq:boundQP},
let us write 
\begin{align}\label{eq:P_prelim}
  P\left(\mathbf{x},t\right){=}
  &{\frac{t}{2\pi}\int_{0}^{1}
    \frac{e^{ir\omega t}}{\left(1-r\right)^{1/2}}
    \int_{\left|\mathbf{s}\right|=1}
    \left[\dfrac{r e^{-ir\omega
          t}}{(1+r)^{1/2}}A_0\left(\rho\right)Y_1\left(\phi\right)\right.} \\
&  {\left. -\left(\dfrac{e^{-i\omega t}}{\sqrt{2}}-\dfrac{e^{-i\omega t}}{\sqrt{2}}\right)
          A_0\left(\left|\mathbf{x}+\mathbf{s}c_{0}t\right|\right)Y_1\left(\frac{\mathbf{x}+\mathbf{s}c_{0}t}{\left|\mathbf{x}+\mathbf{s}c_{0}t\right|}\right)
          \right]
    d\sigma_{\mathbf{s}}dr} \nonumber\\
  =&\frac{t}{2\pi}\int_{0}^{1}\frac{e^{ir\omega t}}{\left(1-r\right)^{1/2}}\int_{\left|\mathbf{s}\right|=1}\left[\frac{re^{-ir\omega t}}{\left(1+r\right)^{1/2}}A_0\left(\rho\right)Y_1\left(\phi\right)\right. \nonumber\\
 & \left.-\frac{e^{-i\omega t}}{\sqrt{2}}A_0\left(\left|\mathbf{x}+\mathbf{s}c_{0}t\right|\right)Y_1\left(\frac{\mathbf{x}+\mathbf{s}c_{0}t}{\left|\mathbf{x}+\mathbf{s}c_{0}t\right|}\right)\right]d\sigma_{\mathbf{s}}dr\nonumber\\
 & +\frac{t}{2\sqrt{2}\pi}\int_{\left|\mathbf{s}\right|=1}A_0\left(\left|\mathbf{x}+\mathbf{s}c_{0}t\right|\right)Y_1\left(\frac{\mathbf{x}+\mathbf{s}c_{0}t}{\left|\mathbf{x}+\mathbf{s}c_{0}t\right|}\right)d\sigma_{\mathbf{s}}\int_{0}^{1}\frac{e^{-i\left(1-r\right)\omega t}}{\left(1-r\right)^{1/2}}dr\nonumber\\
=: & P_{1}\left(\mathbf{x},t\right)+P_{2}\left(\mathbf{x},t\right).\nonumber
\end{align}
We start with
\begin{equation}\label{eq:P2_def0}
P_{2}\left(\mathbf{x},t\right)=\frac{1}{t^{1/2}}F_{2}\left(\mathbf{x},t\right)\int_{0}^{1}\frac{e^{-ir\omega t}}{r^{1/2}}dr,
\end{equation}
where we made a change of variable $r\mapsto\left(1-r\right)$ and introduced
\begin{equation}\label{eq:F2_def0}
F_{2}\left(\mathbf{x},t\right):=\frac{t^{3/2}}{2^{3/2}\pi}\int_{\left|\mathbf{s}\right|=1}A_0\left(\left|\mathbf{x}+\mathbf{s}c_{0}t\right|\right)Y_1\left(\frac{\mathbf{x}+\mathbf{s}c_{0}t}{\left|\mathbf{x}+\mathbf{s}c_{0}t\right|}\right)d\sigma_{\mathbf{s}}.
\end{equation}
Using \eqref{triangle}, the assumed form of $A_0$ and \eqref{eq:rt_rho_bnd}, both for $\rho>\rho_1$, we have uniformly for $\vx\in\Omega$, $\vs\in\mathbb{S}^1$,
\begin{align}\label{eq:A0_estim}
\left(rt\right)^{3/2}\left|A_0\left(\rho\right)\right|\leq&\begin{cases}
\left(\frac{\rho+\rho_0}{c_0}\right)^{3/2}\left\Vert A_0\right\Vert_{L^\infty\left(\mathbb{R}_+\right)}, \hspace{1em} 0\leq\rho\leq\rho_1,\\
\left(\frac{rt}{\rho}\right)^{3/2}\leq\left(\frac{2}{c_0}\right)^{3/2}, \hspace{1em} \rho>\rho_1,
\end{cases}\\
  \leq& C_1,\hspace{1em}\rho\ge
        0,\nonumber
\end{align}
for some constant $C_1>0$.
Thus, using \eqref{eq:A0_estim} with $r=1$ and recalling the assumptions on $Y_1$, we deduce
\begin{equation}\label{eq:F2_estim}
\underset{\mathbf{x}\in\Omega}{\sup}\left\Vert F_{2}\left(\mathbf{x},\cdot\right)\right\Vert _{L^{\infty}\left(\mathbb{R}_+\right)}=:C_{2}<\infty.
\end{equation}
Finally, employing Lemma~\ref{lem:app_osc_int_estim} from Appendix~\ref{sec:app},
we obtain from \eqref{eq:P2_def0} {and~\eqref{eq:F2_estim}}, for $\vx\in\Omega$ and $t\geq t_0$,
\begin{align}\label{eq:P2_estim0}
\left|P_{2}\left(\mathbf{x},t\right)\right|\leq\frac{\widetilde{C}_2}{t}
\end{align}
with some constant $\widetilde{C}_2>0$ and any $t_0>0$.

\medskip
\underline{Decay of $P_1$.}
{To deal with $P_{1}$, we note that the integrand is a smooth
  function of $r$ in $[0,1)$ and it behaves like $(1-r)^{1/2}$ as $r\to 1$. Integrating by parts in
    the $r$ variable with $e^{ir\omega t}dr$ as differential, both boundary terms vanish (recall also that $A_0\left(\left|\vx\right|\right)\equiv 0$ for
    $\vx\in\Omega$).} We thus arrive at
\begin{align}\label{eq:P1_decomp}
P_1\left(\mathbf{x},t\right)= & -\frac{1}{2\pi i\omega}\int_{0}^{1}\int_{\left|\mathbf{s}\right|=1}e^{ir\omega t}\partial_{r}\left(\frac{1}{\left(1-r\right)^{1/2}}\left[\frac{r e^{-ir\omega t}}{\left(1+r\right)^{1/2}}A_0\left(\rho\right)Y_1\left(\phi\right)\right.\right.\\
 & \left.\left.-\frac{e^{-i\omega t}}{\sqrt{2}}A_0\left(\left|\mathbf{x}+\mathbf{s}c_{0}t\right|\right)Y_1\left(\frac{\mathbf{x}+\mathbf{s}c_{0}t}{\left|\mathbf{x}+\mathbf{s}c_{0}t\right|}\right)\right]\right)d\sigma_{\mathbf{s}}dr\nonumber\\
 =& I_1\left(\mathbf{x},t\right)+I_2\left(\mathbf{x},t\right)+I_3\left(\mathbf{x},t\right)+I_4\left(\mathbf{x},t\right),\nonumber
\end{align}
where
\begin{equation}\label{eq:I1_def0}
I_{1}\left(\mathbf{x},t\right):=\frac{i}{2\pi\omega}\int_{0}^{1}\int_{\left|\mathbf{s}\right|=1}\frac{r}{\left(1-r^2\right)^{1/2}}Y_1\left(\phi\right)e^{ir\omega t}\partial_{r}\left(e^{-ir\omega t}A_0\left(\rho\right)\right)d\sigma_{\mathbf{s}}dr,
\end{equation}
\begin{equation}\label{eq:I2_def0}
I_{2}\left(\mathbf{x},t\right):=\frac{i}{2\pi\omega}\int_{0}^{1}\int_{\left|\mathbf{s}\right|=1}\frac{r}{\left(1-r^2\right)^{1/2}}A_0\left(\rho\right)\partial_{r}Y_1\left(\phi\right)d\sigma_{\mathbf{s}}dr,
\end{equation}
\begin{align}\label{eq:I3_def0}
I_{3}\left(\mathbf{x},t\right):=&\frac{i}{4\pi\omega}\int_{0}^{1}\int_{\left|\mathbf{s}\right|=1}\frac{2+r}{\left(1-r\right)^{1/2}\left(1+r\right)^{3/2}}A_0\left(\rho\right)Y_1\left(\phi\right)d\sigma_{\mathbf{s}}dr,
\end{align}
\begin{align}\label{eq:I4_def0}
I_{4}\left(\mathbf{x},t\right):=&\frac{i}{4\pi\omega}\int_{0}^{1}\int_{\left|\mathbf{s}\right|=1}\frac{1}{\left(1-r\right)^{3/2}}\left[\frac{r}{\left(1+r\right)^{1/2}}A_0\left(\rho\right)Y_1\left(\phi\right)\right.\\
 &\left.-\frac{e^{-i\left(1-r\right)\omega t}}{\sqrt{2}}A_0\left(\left|\mathbf{x}+\mathbf{s}c_{0}t\right|\right)Y_1\left(\frac{\mathbf{x}+\mathbf{s}c_{0}t}{\left|\mathbf{x}+\mathbf{s}c_{0}t\right|}\right)\right]d\sigma_{\mathbf{s}}dr.\nonumber
\end{align}

Here and in the sequel, we use the following notation to avoid having too many brackets: $\partial_r A_0(\rho):=\partial_r (A_0(\rho))$, $\partial_r Y_1(\phi):=\partial_r (Y_1(\phi))$, $\nabla Y_1(\phi):=(\nabla Y_1)(\phi)$.

For the sake of the proof, we shall extend the function $Y_1$ from the circle $\mathbb{S}^1$ to a neighbourhood of it, e.g.\ to the annulus with the inner and outer radii $\frac12$ and $\frac32$, respectively. We choose a constant extension in the radial direction, as this will simplify the proof. In fact, this extension makes $\nabla Y_1$ well defined on $\mathbb{S}^1$. Then, $\nabla Y_1(\phi)$ is a tangent vector to $\mathbb{S}^1$ for each (radial vector)
$\phi\in\mathbb{S}^1$, and hence
\begin{equation}\label{tang-vector}
	\phi\cdot \nabla Y_1(\phi)=0.
\end{equation}

\medskip
\underline{Decay of $I_1$.} 
We have
\begin{equation}\label{eq:A0_dr}
\partial_r A_0\left(\rho\right)=c_0 t A_0^\prime\left(\rho\right) \frac{\vx\cdot\vs+rc_0t}{\left|\vx+\vs rc_0 t\right|},\quad \rho>0.
\end{equation}
Moreover, since $\left|\mathbf{x}+\mathbf{s}rc_{0}t\right|=rc_{0}t\left(1+2\frac{\mathbf{x}\cdot\mathbf{s}}{rc_{0}t}+\frac{\left|\mathbf{x}\right|^{2}}{r^{2}c_{0}^{2}t^{2}}\right)^{1/2}$
for $\left|\vs\right|=1$, 
the estimate
\begin{equation}\label{eq:near_id_estim2}
1-\frac{\mathbf{x}\cdot\mathbf{s}+rc_{0}t}{\left|\mathbf{x}+\mathbf{s}rc_{0}t\right|}=\frac{\left|\vx\right|^2-\left(\vx\cdot\vs\right)^2}{2r^2c_0^2 t^2}+\mathcal{O}\left(\frac{1}{r^3t^3}\right)=\mathcal{O}\left(\frac{1}{r^2t^2}\right)
\end{equation}
is valid for $rt\gg1$. {This can be seen from the Taylor expansion of
  $(1+w)^{-1/2}$ around zero, with $w:=2\frac{\mathbf{x}\cdot\mathbf{s}}{rc_{0}t}+\frac{\left|\mathbf{x}\right|^{2}}{r^{2}c_{0}^{2}t^{2}}$}.

Then, we can write
\begin{align*}
\frac{r^{5/2}t^{3/2}}{c_0}e^{ir\omega t}\partial_{r}\left(e^{-ir\omega t}A_0\left(\rho\right)\right)=&\left(rt\right)^{5/2}\left(A_0^{\prime}\left(\rho\right)-\frac{i\omega}{c_{0}}A_0\left(\rho\right)\right)\\
&-\left(rt\right)^{5/2}A_0^{\prime}\left(\rho\right)\left(1-\frac{\mathbf{x}\cdot\mathbf{s}+rc_{0}t}{\left|\mathbf{x}+\mathbf{s}rc_{0}t\right|}\right),\quad \rho>0,
\end{align*}
where both terms on the right-hand side are uniformly bounded for
$rt>a_1$ (and hence $\rho>\rho_0$), $\vx\in\Omega$, $\left|\vs\right|=1$. This can be deduced
  from~\eqref{eq:near_id_estim2}
{using~\eqref{eq:rt_rho_bnd} and the estimates $\left|A_0^{\prime}\left(\rho\right)-\frac{i\omega}{c_{0}}A_0\left(\rho\right)\right|=3/(2\rho^{5/2})$,
$|A_0^\prime(\rho)|\le
    C/\rho^{3/2}$ for $\rho>\rho_1$ and some constant $C>0$.}
Therefore, we have for 
\begin{equation}\label{eq:F3_def0}
F_{3}\left(\mathbf{x},rt\right):=\frac{ic_{0}}{2\pi\omega}\int_{\left|\mathbf{s}\right|=1}Y_1\left(\phi\right)\frac{r^{5/2}t^{3/2}}{c_0}e^{ir\omega t}\partial_{r}\left(e^{-ir\omega t}A_0\left(\rho\right)\right)d\sigma_{\mathbf{s}}:
\end{equation}
\[
\underset{\mathbf{x}\in\Omega}{\sup}\left\Vert F_{3}\left(\mathbf{x},\cdot\right)\right\Vert _{L^{\infty}\left(a_{1},\infty\right)}=:C_{3}<\infty.
\]
Since both $A_0$, $A_0^\prime$ vanish on $\left[0,\rho_0\right]$, the
{integrals in $r$} 
in each of \eqref{eq:I1_def0}--\eqref{eq:I3_def0} reduces to $\left(a_1/t,1\right)$ (see the discussion before {and after} \eqref{eq:a1_def}). Hence
we can estimate $I_1$ in~\eqref{eq:I1_def0} for $t\ge t_0:=2a_1$ as
\begin{align}\label{eq:I1_estim0}
\left|I_{1}\left(\mathbf{x},t\right)\right|=&\frac{1}{t^{3/2}}\left|\int_{a_1/t}^{1}\frac{1}{r^{3/2} \left(1-r^2\right)^{1/2}}F_{3}\left(\mathbf{x},rt\right)dr\right|\\
\leq&\frac{2C_3}{3^{1/2}t^{3/2}}\int_{a_1/t}^{1/2}\frac{dr}{r^{3/2}}+\frac{2^{3/2}C_3}{t^{3/2}}\int_{1/2}^{1}\frac{dr}{\left(1-r\right)^{1/2}}
\leq\frac{\widetilde{C}_3}{t}\nonumber
\end{align}
with some constant $\widetilde{C}_3>0$. In a similar but simpler fashion we can estimate the terms $I_2$ and $I_3$.

\medskip
\underline{Decay of $I_2$.}
Since
\begin{equation}\label{eq:Y1_dr}
\partial_{r}Y_1\left(\phi\right)=\frac{c_{0}t\mathbf{s}\cdot\nabla Y_1\left(\phi\right)}{\left|\mathbf{x}+\mathbf{s}rc_{0}t\right|} ,
\end{equation}
where we used \eqref{tang-vector}, for
\begin{align}\label{eq:F4_def0}
F_{4}\left(\mathbf{x},rt\right):=&\frac{ic_0\left(rt\right)^{5/2}}{2\pi\omega}\int_{\left|\mathbf{s}\right|=1}A_0\left(\rho\right) 
\frac{\mathbf{s}\cdot\nabla Y_1\left(\phi\right)}{\left|\mathbf{x}+\mathbf{s}rc_{0}t\right|} 
\, d\sigma_{\mathbf{s}},  
\end{align}
we have
\[\left|F_{4}\left(\vx,rt\right)\right|\leq\frac{c_{0}}{\pi\omega}\int_{\left|\vs\right|=1}\frac{rt}{\rho}(rt)^{3/2}\left|A_{0}\left(\rho\right)\right|\left|\nabla Y_{1}\left(\phi\right)\right| d\sigma_{\vs}.
\]
Hence, using~\eqref{eq:rt_rho_bnd} and~\eqref{eq:A0_estim}, we deduce 
\[
\underset{\mathbf{x}\in\Omega}{\sup}\left\Vert F_{4}\left(\mathbf{x},\cdot\right)\right\Vert _{L^{\infty}\left(a_{1},\infty\right)}=:C_{4}<\infty.
\]
Therefore, we obtain for $t\ge t_0=2a_1$
\begin{align}\label{eq:I2_estim0}
\left|I_{2}\left(\mathbf{x},t\right)\right|=&\frac{1}{t^{3/2}}\left|\int_{a_1/t}^{1}\frac{1}{r^{3/2} \left(1-r^2\right)^{1/2}}F_{4}\left(\mathbf{x},rt\right)dr\right|\\
\leq&\frac{2C_4}{3^{1/2}t^{3/2}}\int_{a_1/t}^{1/2}\frac{dr}{r^{3/2}}+\frac{2^{3/2}C_4}{t^{3/2}}\int_{1/2}^{1}\frac{dr}{\left(1-r\right)^{1/2}}
\leq\frac{\widetilde{C}_4}{t}.\nonumber
\end{align}

\medskip
\underline{Decay of $I_3$.}
Similarly,
\begin{equation}\label{eq:F5_def0}
F_{5}\left(\mathbf{x},rt\right):=\frac{i\left(rt\right)^{3/2}}{4\pi\omega}\int_{\left|\mathbf{s}\right|=1}A_0\left(\rho\right)Y_1\left(\phi\right)d\sigma_{\mathbf{s}},
\end{equation}
satisfies
\[
  \left|F_{5}\left(\mathbf{x},rt\right)\right|=\frac{1}{4\pi\omega}\int_{\left|\mathbf{s}\right|=1}
  {\left(rt\right)^{3/2}}
  \left|A_0\left(\rho\right)\right|\left|Y_1\left(\phi\right)\right|d\sigma_{\mathbf{s}}.
\]
Hence, using again~\eqref{eq:A0_estim},
\begin{equation}\label{eq:F5_estim}
\underset{\mathbf{x}\in\Omega}{\sup}\left\Vert F_{5}\left(\mathbf{x},\cdot\right)\right\Vert _{L^{\infty}\left(0,\infty\right)}=:C_{5}<\infty.
\end{equation}
Therefore, we obtain for $t\ge t_0=2a_1$
\begin{align}\label{eq:I3_estim0}
\left|I_{3}\left(\mathbf{x},t\right)\right|=&\frac{1}{t^{3/2}}\left|\int_{a_1/t}^{1}\frac{2+r}{r^{3/2}\left(1-r\right)^{1/2}\left(1+r\right)^{3/2}}F_{5}\left(\mathbf{x},rt\right)dr\right|\\
\leq&\frac{5C_5}{2^{1/2}t^{3/2}}\int_{a_1/t}^{1/2}\frac{dr}{r^{3/2}}+\frac{8C_5}{3^{1/2}t^{3/2}}\int_{1/2}^{1}\frac{dr}{\left(1-r\right)^{1/2}}
\leq\frac{\widetilde{C}_5}{t}\nonumber
\end{align}
with some constants $\widetilde{C}_4$, $\widetilde{C}_5>0$. 

\medskip
\underline{Decay of $I_4$.}
To treat the term $I_4$, 
we introduce
\begin{align}\label{eq:F6_def_tild0}
\widetilde{F}_{6}\left(\mathbf{x},r,t\right):=&\int_{\left|\mathbf{s}\right|=1}\left[\frac{r}{\left(1+r\right)^{1/2}}A_0\left(\rho\right)Y_1\left(\phi\right)\right.\\
 &\left.-\frac{e^{-i\left(1-r\right)\omega t}}{\sqrt{2}}A_0\left(\left|\mathbf{x}+\mathbf{s}c_{0}t\right|\right)Y_1\left(\frac{\mathbf{x}+\mathbf{s}c_{0}t}{\left|\mathbf{x}+\mathbf{s}c_{0}t\right|}\right)\right]d\sigma_{\mathbf{s}},\nonumber
\end{align}
\begin{equation}\label{eq:F6_def0}
F_{6}\left(\mathbf{x},r,t\right):=\frac{1}{1-r}\widetilde{F}_{6}\left(\mathbf{x},r,t\right).
\end{equation}
Using $\widetilde{F}_{6}\left(\mathbf{x},1,t\right)=0$, we rewrite \eqref{eq:F6_def0} as 
\begin{equation}\label{eq:F6_estim}
F_{6}\left(\mathbf{x},r,t\right)=-\frac{1}{1-r}\left(\widetilde{F}_{6}\left(\mathbf{x},1,t\right)-\widetilde{F}_{6}\left(\mathbf{x},r,t\right)\right)=-\frac{1}{1-r}\int_r^1 \partial_r\widetilde{F}_{6}\left(\mathbf{x},\tau,t\right)d\tau.
\end{equation}
Then, for $r\in\left(1-a_1/t,1\right)$, we estimate
\begin{align}\label{Est-F6}
\left|F_{6}\left(\mathbf{x},r,t\right)\right|\leq&\left\Vert\partial_r\widetilde{F}_{6}\left(\mathbf{x},\cdot,t\right)\right\Vert_{L^\infty\left(1-a_1/t,1\right)}\\
\leq&\int_{\left|\mathbf{s}\right|=1}\left[\left(\frac{1}{2\left(1+r\right)^{3/2}}+\frac{1}{\left(1+r\right)^{1/2}}\right)\left|A_0\left(\rho\right)\right|\left|Y_1\left(\phi\right)\right|\right.\nonumber\\
&\left. +\frac{1}{\left(1+r\right)^{1/2}}\left|\partial_r A_0\left(\rho\right)\right|\left|Y_1\left(\phi\right)\right|+\frac{1}{\left(1+r\right)^{1/2}}\left|A_0\left(\rho\right)\right|\left|\partial_r Y_1\left(\phi\right)\right|\right.\nonumber\\
&\left.+\frac{\omega t}{\sqrt{2}}\left|A_0\left(\left|\vx+\vs c_0 t\right|\right)\right|\left|Y_1\left(\frac{\vx+\vs c_0 t}{\left|\vx+\vs c_0 t\right|}\right)\right|\right]d\sigma_{\mathbf{s}}.\nonumber
\end{align}
For the term with $\partial_r A_0$ we first use \eqref{eq:A0_dr} and consider the following estimate which holds uniformly for $\vx\in\Omega$, $\vs\in\mathbb{S}^1$. It is obtained in analogy to \eqref{eq:A0_estim}.
\begin{align*}
\left(rt\right)^{3/2}\left|A_0'\left(\rho\right)\right|\leq&\begin{cases}
\left(\frac{\rho+\rho_0}{c_0}\right)^{3/2}\left\Vert A_0'\right\Vert_{L^\infty\left(\mathbb{R}_+\right)}, \hspace{1em} 0\leq\rho\leq\rho_1,\\
\left(\frac{rt}{\rho}\right)^{3/2}\left(\frac{\omega}{c_0}+\frac{3}{2\rho_1}\right)\leq\left(\frac{2}{c_0}\right)^{3/2}\left(\frac{\omega}{c_0}+\frac{3}{2\rho_1}\right), \hspace{1em} \rho>\rho_1,
\end{cases}\\
  \leq& \widetilde C_1,\hspace{1em}\rho\ge
        0,\nonumber
\end{align*}
for some constant $\widetilde C_1>0$.

For the term with $\partial_r Y_1$ we employ  \eqref{eq:A0_estim}, $r\ge1-\frac{a_1}{t}\ge\frac12$ (for $t\ge t_0=2a_1$), and
$$
  |\partial_r Y_1(\phi)|\le \frac{2c_0t|\nabla Y_1(\phi)|}{\rho}\le \frac{2c_0t}{\rho_0}\|\nabla Y_1\|_{L^\infty(\mathbb{S}^1)},
$$
where we used \eqref{eq:Y1_dr} and the fact that $A_0$ vanishes on $[0,\rho_0]$. 

Therefore, employing in \eqref{Est-F6} the estimate  \eqref{eq:A0_estim} with $r=1$
we deduce that 
\begin{equation}\label{eq:C6_def0}
\underset{\mathbf{x}\in\Omega,\hspace{0.3em}r\in\left(1-a_1/t,\, 1\right),\hspace{0.3em}t>2a_1}{\sup}\left| F_{6}\left(\mathbf{x},r,t\right)t^{1/2} \right|=:C_{6}<\infty.
\end{equation}

Moreover, for $r\in\left(1/2,1-a_1/t\right)$ {and for
  $\epsilon\in(0,1/2]$, taking into account~\eqref{eq:F6_def0}}, we have 
\[
\left|\left(1-r\right)^{1/2+\epsilon}t^{1+\epsilon}F_{6}\left(\mathbf{x},r,t\right)\right|= \left|\widetilde{F}_6\left(\mathbf{x},r,t\right)\right|\frac{t^{1+\epsilon}}{\left(1-r\right)^{1/2-\epsilon}} \leq \left|\widetilde{F}_6\left(\mathbf{x},r,t\right)\right|\frac{t^{3/2}}{a_1^{1/2-\epsilon}}.
\]
Hence, {with the constants $C_2$ and $C_5$ introduced in
  \eqref{eq:F2_estim} and \eqref{eq:F5_estim}, respectively,}
we obtain from \eqref{eq:F6_def_tild0} 
that, for any $\epsilon\in\left(0,1/2\right]$,
\begin{align}\label{eq:C7_def0}
\underset{\mathbf{x}\in\Omega,\hspace{0.3em}r\in\left(1/2,\,
  1-a_1/t\right),\hspace{0.3em}t>2a_1}{\sup}\left|\left(1-r\right)^{1/2+\epsilon}t^{1+\epsilon}F_{6}\left(\mathbf{x},r,t\right)\right|\leq&
                                                                                                                                            \frac{1}{a_1^{1/2-\epsilon}}\left(\dfrac{8\pi\omega}{3^{1/2}} C_{5}+2\pi C_2\right)\\
=:&C_7<\infty.\nonumber
\end{align}
Bounds~\eqref{eq:C7_def0} and~\eqref{eq:C6_def0} imply that,
for $t\ge t_0=2a_1$ and $\epsilon\in\left(0,1/2\right]$, we get
\begin{align}\label{eq:I4_estim0}
\left|I_{4}\left(\mathbf{x},t\right)\right|\leq&\frac{1}{4\pi\omega}\int_{0}^{1}\frac{1}{\left(1-r\right)^{1/2}}\left|F_{6}\left(\mathbf{x},r,t\right)\right|dr\\
=&\frac{1}{4\pi\omega}\left(\int_{0}^{1/2}\ldots+\int_{1/2}^{1-a_1/t}\ldots+\int_{1-a_1/t}^{1}\ldots\right)\nonumber\\
\leq&\frac{2^{3/2}}{t^{3/2}}\left(C_5\int_{0}^{1/2}\frac{dr}{r^{1/2}}+\frac{C_2}{4\omega}\right)+\frac{C_7}{4\pi\omega
      t^{1+\epsilon}}\int_{1/2}^{1-a_1/t}\frac{dr}{\left(1-r\right)^{1+\epsilon}}
      \nonumber\\
&+\frac{C_6}{4\pi\omega
                    t^{1/2}}\int_{1-a_1/t}^{1}\frac{dr}{\left(1-r\right)^{1/2}}
\leq\frac{\widetilde{C}_6}{t}\nonumber
\end{align}
with some constant $\widetilde{C}_6>0$. 
For the interval $\left(0,1/2\right)$, we estimated here the integrand directly from \eqref{eq:I4_def0}, using~\eqref{eq:F5_estim} and~\eqref{eq:F2_estim}. 

From estimates~\eqref{eq:I1_estim0}, \eqref{eq:I2_estim0},
\eqref{eq:I3_estim0}, and \eqref{eq:I4_estim0} of the terms $I_1$,
$I_2$, $I_3$, and $I_4$, respectively, in decomposition \eqref{eq:P1_decomp}, we obtain 
$$
  |P_1(x,t)|\le\frac{C_8}{t}
$$
with some constant $C_8>0$ and $t_0=2a_1$. Together with \eqref{eq:P2_estim0} this gives the estimate of $P$ in~\eqref{eq:boundQP}, again
    with $t_0=2a_1$.
This concludes the proof of~\eqref{eq:v_dec2show} in the
  case $d=2$.

\medskip
\noindent
{\bf$\bullet$  Case $d=3$.}

The solution is given by Kirchhoff's formula \cite[Par. 2.4.1 (c)]{Evans} 
\begin{equation}\label{eq:v_sol_Kirch}
v\left(\vx,t\right)=\frac{1}{4\pi}\int_{\left|\vs\right|=1}t \, v_1\left(\vx+\vs c_0 t\right)d\sigma_\vs,
\end{equation}
where $d\sigma_\vs$ denotes the surface measure of the unit sphere $\mathbb{S}^2$.
{In this case, it is immediate to see that \eqref{eq:v_sol_Kirch}
implies \eqref{eq:v_dec2show}, owing to assumption~\eqref{eq:assumption3D}.}

\medskip
\noindent
{\bf$\bullet$  Conclusion of the proof.}
So far, we have proved the decay of the solution under the assumption
        $v_0\equiv 0$. We now extend the result to the general case
        and show that estimates analogous to \eqref{eq:v_dec2show}
        hold true for the derivatives of the solution.
 
To proceed, it is convenient to introduce the following notation.
Given a function $w$ and constants $\omega$, $c_0$, $\rho_0>0$, $\rho_1>\rho_0$, we say that 
\begin{equation}\label{eq:set-A}
w\in\mathcal{A}\equiv\mathcal{A}_{\omega,c_0,\rho_0,\rho_1}
\end{equation}
if the following conditions are satisfied:
\begin{itemize}
\item[-] When $d=2$, we can write 
\[
w(\vx)=A_w(|\vx|)Y_w\left(\frac{\vx}{|\vx|}\right)+V_w(\vx)
\]
for some functions $A_w\in C^1(\mathbb{R}_{+})$, $Y_w\in C^1(\mathbb{S}^1)$, $V_w\in C(\mathbb{R}^2)$ 
such that
\[
A_w(\left|\vx\right|)\equiv 0\equiv V_w(\vx),\quad |\vx|\leq\rho_0, \qquad A_w(\rho)=
\dfrac{e^{i\frac{\omega}{c_0}\rho}}{\rho^{3/2}},\quad \rho>\rho_1,
\]
and we have
\[
|\vx|^{5/2}|V_w(\vx)|\leq C,\quad \vx\in\mathbb{R}^2,
\]
for some constant $C>0$.
\item[-] When $d=3$, we have $w\in C(\mathbb{R}^3)$ and
\[
|\vx|^{2}|w(\vx)|\leq C,\quad \vx\in\mathbb{R}^3,
\]
for some constant $C>0$.
\end{itemize} 

Let us denote $Z\left[v_0,v_1\right]\equiv Z$ the solution of the wave equation $\partial_t^2 Z\left(\vx,t\right)-c_0^2\Delta Z\left(\vx,t\right)=0$ for $\vx\in\mathbb{R}^d$, $t>0$, subject to the initial conditions $Z\left(\mathbf{x},0\right)=v_0\left(\mathbf{x}\right)$, $\partial_t Z\left(\mathbf{x},0\right)=v_1\left(\mathbf{x}\right)$. {The assumptions $v_0\in C^{6}\left(\mathbb{R}^d\right)$ and $v_1\in C^{5}\left(\mathbb{R}^d\right)$ entail 
\begin{equation}\label{eq:Z_reg}
Z\in C^{5}\left(\mathbb{R}^d\times\mathbb{R}_+\right), \quad
\partial_t Z\in C^{4}\left(\mathbb{R}^d\times\mathbb{R}_+\right);
\end{equation}
recall Theorems 2 and 3 in \cite[Par. 2.4.1]{Evans}  {(applied to $Z$ and its derivatives)}, see also \eqref{eq:v_sol_v0v1} and \eqref{eq:v_sol_Kirch_2}}.

We recall that, in proving the decay for $Z\left[0,v_1\right]$ given by \eqref{eq:v_dec2show}, we have only used that $v_1\in\mathcal{A}$.
But, for $d=3$, due to assumption \eqref{eq:assumption3D}, we also have
$$
  v_0,\,\Delta v_0,\,\Delta^2 v_0,\,v_1,\,\Delta v_1,\,\Delta^2 v_1\in \mathcal{A},
$$
which will be used to prove the decay of $Z$ with such initial conditions.

For $d=2$, we observe that 
$v_1\in\mathcal{A}$ entails that $\Delta v_1$, $\Delta^2 v_1\in\mathcal{A}$ under the regularity assumptions on $v_1$ made in \eqref{eq:gen_ICs}, \eqref{eq:V0V1_dec}.
Indeed, a short computation in polar coordinates yields that 
\begin{align*}
\Delta v_{1}\left(x\right)=\begin{cases}
0, & \left|\mathbf{x}\right|\leq\rho_{0},\\
-\left(\frac{\omega}{c_{0}}\right)^{2}\frac{e^{i\frac{\omega}{c_{0}}\left|\mathbf{x}\right|}}{\left|\mathbf{x}\right|^{3/2}} Y_1\left(\frac{\vx}{|\vx|}\right)+\mathcal{O}\left(\frac{1}{\left|\mathbf{x}\right|^{5/2}}\right), & \left|\mathbf{x}\right|>\rho_{1},
\end{cases}
\end{align*}
so $\Delta v_1\in\mathcal{A}$ with $Y_w=-\left(\omega/c_0\right)^2 Y_1$. Iterating, we also obtain $\Delta^2 v_1\in\mathcal{A}$ (with $Y_w=\left(\omega/c_0\right)^4 Y_1$). 
Since the assumptions on $v_0$ made in \eqref{eq:gen_ICs}, \eqref{eq:V0V1_dec} provide 
$v_0\in\mathcal{A}$, we have similarly $\Delta v_0$, $\Delta^2 v_0\in\mathcal{A}$.

Consequently, we deduce both for $d=2$ and $d=3$ that 
\begin{equation*}
\left|\Delta Z\left[0, v_1\right]\left(\vx,t\right)\right|=\left| Z\left[0, \Delta v_1\right]\left(\vx,t\right)\right|\leq \frac{C}{\left(1+t^2\right)^{1/2}}, \hspace{1em}\vx\in\Omega,  \hspace{1em}t\geq 0,
\end{equation*}
{where the first equality is implied by $\partial_t^2\Delta Z\left[0, v_1\right]\left(\vx,t\right) = \Delta \partial_t^2 Z\left[0, v_1\right]\left(\vx,t\right)$ which is due to $v_1{\in C^5\left(\mathbb{R}^d\right)\subset} C^4\left(\mathbb{R}^d\right)$.
Therefore,} in view of validity of the wave equation, {we deduce}
\begin{equation*}
\left|\partial_t^2 Z\left[0, v_1\right]\left(\vx,t\right)\right|\leq \frac{C}{\left(1+t^2\right)^{1/2}}, \hspace{1em}\vx\in\Omega,  \hspace{1em}t\geq 0.
\end{equation*}
For fixed $\vx\in\Omega$, we interpolate between the decay estimates of $Z[0, v_1]$ and $\partial_t^2 Z[0, v_1]$ by using Lemma~\ref{1D-interpol}. Then, we have
\begin{equation*}
\left|\partial_t Z\left[0, v_1\right]\left(\vx,t\right)\right|\leq \frac{C}{\left(1+t^2\right)^{1/2}}, \hspace{1em}\vx\in\Omega,  \hspace{1em}t\geq 0.
\end{equation*}

We can write $Z\left[v_0,v_1\right]=Z\left[0,v_1\right]+\partial_t Z\left[0,v_0\right]$, and hence we have
\begin{equation}\label{eq:W_dec_estim}
\left|Z\left[v_0, v_1\right]\left(\vx,t\right)\right|\leq \frac{C}{\left(1+t^2\right)^{1/2}}, \hspace{1em}\vx\in\Omega,  \hspace{1em}t\geq 0,
\end{equation}
due to $v_1$, $v_0$, $\Delta v_0\in\mathcal{A}$.

Also, writing $\Delta Z\left[v_0,v_1\right]=Z\left[0,\Delta v_1\right]+\partial_t Z\left[0,\Delta v_0\right]$, we have
\begin{equation}\label{eq:W_Lapl_dec_estim}
\left|\Delta Z\left[v_0, v_1\right]\left(\vx,t\right)\right|\leq \frac{C}{\left(1+t^2\right)^{1/2}}, \hspace{1em}\vx\in\Omega,  \hspace{1em}t\geq 0,
\end{equation}
since $\Delta v_1$, $\Delta v_0$, $\Delta^2 v_0\in\mathcal{A}$ and $v_0{\in C^6\left(\mathbb{R}^d\right)\subset} C^{5}\left(\mathbb{R}^d\right)$, $v_1{\in C^5\left(\mathbb{R}^d\right)\subset} C^{4}\left(\mathbb{R}^d\right)$. 
Using the wave equation, we interpolate between the decay estimates of $Z[v_0, v_1]$ and $\partial_t^2 Z[v_0, v_1]$ by employing Lemma \ref{1D-interpol}. We thus obtain
\begin{equation}\label{eq:W_t_dec_estim}
\left|\partial_t Z\left[v_0, v_1\right]\left(\vx,t\right)\right|\leq \frac{C}{\left(1+t^2\right)^{1/2}}, \hspace{1em}\vx\in\Omega,  \hspace{1em}t\geq 0.
\end{equation}


Moreover, {since $\Delta Z\left(0,\Delta v_0\right)=Z\left(0,\Delta^2 v_0\right)$ due to $v_0\in C^{6}\left(\mathbb{R}^d\right)$, we can write} 
\begin{align*}
\partial_t \Delta Z\left[v_0,v_1\right]=&\partial_t Z\left[0,\Delta v_1\right]+\partial_t^2 Z\left[0,\Delta v_0\right]\\
=&\partial_t Z\left[0,\Delta v_1\right]+c_0^2 Z\left[0,\Delta^2 v_0\right].
\end{align*}
{Consequently, }
\begin{equation}\label{eq:W_t_Lapl_dec_estim}
\left|\partial_t\Delta Z\left[v_0, v_1\right]\left(\vx,t\right)\right|\leq \frac{C}{\left(1+t^2\right)^{1/2}}, \hspace{1em}\vx\in\Omega,  \hspace{1em}t\geq 0,
\end{equation}
as $v_1$, $\Delta v_1$, $\Delta^2 v_1$, $\Delta^2 v_0\in\mathcal{A}$ {and $v_0\in C^{6}\left(\mathbb{R}^d\right)$, $v_1\in C^{5}\left(\mathbb{R}^d\right)$}.


To deduce the estimates analogous to \eqref{eq:W_Lapl_dec_estim} and \eqref{eq:W_t_Lapl_dec_estim}
but involving the gradient of $Z$ instead of the Laplacian, we use an
interpolation argument. 
In particular, for a function $u\in
  C^2\big(\widetilde{\Omega}\big)$, with some $\widetilde{\Omega}$
  such that $\Omega\Subset\widetilde{\Omega}$,
using interior elliptic regularity results, we obtain
\begin{align}\label{eq:inter_reg}
\left\Vert \nabla u\right\Vert _{L^{\infty}\left(\Omega\right)} & \leq\sup_{x\in\Omega}\left(\frac{1}{d_{x}}\right)\sup_{x\in\Omega}\left(d_{x}\left|\nabla u\left(x\right)\right|\right)\leq C_{01}\sup_{x\in\widetilde{\Omega}}\left(d_{x}\left|\nabla u\left(x\right)\right|\right)\\
 & \leq C_{01}C_{02}\left[\left\Vert u\right\Vert _{L^{\infty}\left(\widetilde{\Omega}\right)}+\sup_{x\in\widetilde{\Omega}}d_{x}^2\left\Vert \Delta u\right\Vert _{L^{\infty}\left(\widetilde{\Omega}\right)}\right]\nonumber\\
 & \leq C_{03}\left[\left\Vert u\right\Vert _{L^{\infty}\left(\widetilde{\Omega}\right)}+\left\Vert \Delta u\right\Vert _{L^{\infty}\left(\widetilde{\Omega}\right)}\right],\nonumber
\end{align}
with some constants $C_{01}$, $C_{02}$, $C_{03}>0$. Here, we employed
the notation
$d_x:=\textrm{dist}\big(x,\partial\widetilde{\Omega}\big)$ and,
in the second line, we used the first estimate of \cite[Thm
3.9]{Gilbarg}. Observe now that, since in the statement of the
present lemma, the domain $\Omega$ was arbitrary, all the
previous steps of this proof remain valid (with different
constants) for the larger domain $\widetilde{\Omega}$. In particular,
estimates \eqref{eq:W_dec_estim}--\eqref{eq:W_t_Lapl_dec_estim} are valid with $\Omega$ replaced by $\widetilde{\Omega}$. Consequently,
according to \eqref{eq:inter_reg} applied to $Z$ and to
  $\partial_t Z$ {(permitted by the regularity~\eqref{eq:Z_reg}), 
  } 
we deduce
\[
\left|\nabla Z\left[v_0, v_1\right]\left(\vx,t\right)\right|+\left|\partial_t\nabla Z\left[v_0, v_1\right]\left(\vx,t\right)\right|\leq \frac{C}{\left(1+t^2\right)^{1/2}}, \hspace{1em}\vx\in\Omega,  \hspace{1em}t\geq 0,
\]
and the proof is complete.

\qed


\subsection{A variant of Lemma~\ref{lem:gen_dec}}
In this subsection, we present a small variation of Lemma~\ref{lem:gen_dec}, which provides a weaker result under weaker assumptions. This lemma will be used as an auxiliary tool to prove Lemma~\ref{lem:spec_dec}.
\begin{lem}\label{lem:Y0Y1_zero} 
Let $d=2, 3$ and $\omega$, $\rho_0>0$, $\rho_1>\rho_0$ be some fixed
constants. Using the notation introduced in Lemma \ref{lem:gen_dec}, suppose that $\Omega\Subset\mathbb{B}_{\rho_0}$.
When $d=2$, we assume that $v_0\in C^1\left(\mR^2\right)$, $v_1\in C\left(\mR^2\right)$ are such that
\begin{equation}\label{eq:v0_dv0_v1_ass_d2}
\left|\vx\right|^{5/2} 
\big(\left|v_0\left(\vx\right)\right|+\left|v_1\left(\vx\right)\right|+\left|\nabla v_0\left(\vx\right)\right|\big)\leq C_0,\hspace{1em}\vx\in\mathbb{R}^{2},
\end{equation}
and $v_0\left(\vx\right)=v_1\left(\vx\right)\equiv
0$ for $\left|\vx\right|\leq\rho_0$.
When  $d=3$, we assume that $v_0\in C^{1}\left(\mR^3\right)$, $v_1\in C\left(\mR^3\right)$ are such that 
\begin{equation}\label{eq:v0_dv0_v1_ass_d3}
\left|\vx\right|^{2} 
\big(
\left|v_0\left(\vx\right)\right|+
\left|v_1\left(\vx\right)\right|+\left|\nabla v_0\left(\vx\right)\right|
\big) \leq C_0,\hspace{1em}\vx\in\mathbb{R}^3.
\end{equation}
Then, there exists a constant $C>0$ such that, for all $\vx\in\Omega$ 
and $t\geq 0$, the solution of \eqref{eq:wave_c0}  satisfies
\begin{equation}
\left|v\left(\vx,t\right)\right|\leq\frac{C}{\left(1+t^2\right)^{1/2}},\hspace{1em}\vx\in\Omega,\hspace{1em} t\geq0.\label{eq:v_only_estim}
\end{equation}
\end{lem}
\begin{proof}
We treat separately the cases $d=2$ and $3$.

\begin{itemize}
\item{\bf Case $d=2$.}
\end{itemize}

According to Poisson's formula  \cite[Par. 2.4.1 (c)]{Evans} (see also \eqref{eq:v_sol_v0v1} below), we have 
\begin{align}
v\left(\mathbf{x},t\right)= & \frac{t}{2\pi}\int_{0}^{1}\frac{r}{\left(1-r^{2}\right)^{1/2}}\int_{\left|\mathbf{s}\right|=1}v_1\left(\mathbf{x}+\mathbf{s}rc_0t\right)d\sigma_{\mathbf{s}}dr\label{eq:v_sol_v0v1_Poiss}\\
&+\frac{1}{2\pi}\int_{0}^{1}\frac{r}{\left(1-r^{2}\right)^{1/2}}\int_{\left|\mathbf{s}\right|=1}v_0\left(\mathbf{x}+\mathbf{s}rc_0t\right)d\sigma_{\mathbf{s}}dr\nonumber\\
&+\frac{c_0t}{2\pi}\int_{0}^{1}\frac{r^2}{\left(1-r^{2}\right)^{1/2}}\int_{\left|\mathbf{s}\right|=1}\vs\cdot\nabla v_0\left(\mathbf{x}+\mathbf{s}rc_0t\right)d\sigma_{\mathbf{s}}dr\nonumber\\
=:&Q_1\left(\vx,t\right)+Q_2\left(\vx,t\right)+Q_3\left(\vx,t\right).\nonumber
\end{align}
Using the support assumption on $v_0$, we recall the definition of
$a_1$ given by \eqref{eq:a1_def} and realise that the term $Q_1$ is
identical to $Q$ in the proof of Lemma~\ref{lem:gen_dec} (see
\eqref{eq:Q_def}). Hence, \eqref{eq:Q_estim1} furnishes the required
estimate of $Q_1$ due to assumption \eqref{eq:v0_dv0_v1_ass_d2}. The
term $Q_2$ only differs from $Q_1$ by the absence of the $t$ factor and thus obeys an analogous estimate (in fact it is even $\mathcal O(t^{-2})$). Therefore, we immediately obtain
\begin{equation}\label{eq:Q1Q2_estim} 
\left|Q_1\left(\vx,t\right)\right|+\left|Q_2\left(\vx,t\right)\right|\leq\frac{C}{\left(1+t^2\right)^{1/2}},\hspace{1em}\vx\in\Omega,\hspace{1em} t\geq t_0,
\end{equation}  
with some constant $C>0$ and $t_0=2a_1$. 

It thus remains to deal with $Q_3$. To this effect, we rewrite 
\[
Q_{3}\left(\vx,t\right)=\frac{c_{0}}{2\pi t^{3/2}}\int_{a_{1}/t}^{1}\frac{1}{r^{1/2}\left(1-r^{2}\right)^{1/2}}\int_{\left|\vs\right|=1}\left(\frac{rt}{\rho}\right)^{5/2}\rho^{5/2}\vs\cdot\nabla v_{0}\left(\vx+\vs rc_{0}t\right)d\sigma_{\vs}dr,
\]
where $\rho:=\left|\vx+\vs rc_{0}t\right|$, and we reduced the $r$-integration range, following the discussion ``around'' \eqref{eq:a1_def}. Employing \eqref{eq:rt_rho_bnd} and \eqref{eq:v0_dv0_v1_ass_d2}, we can estimate
\[
\left|Q_{3}\left(\vx,t\right)\right|\leq\frac{2^{5/2}C_{0}}{\left(c_{0}t\right)^{3/2}}\int_{a_{1}/t}^{1}\frac{dr}{r^{1/2}\left(1-r^{2}\right)^{1/2}}\leq\frac{2^{5/2}C_{0}}{c_{0}^{3/2}a_{1}^{1/2}}\frac{1}{t}\int_{0}^{1}\frac{dr}{\left(1-r^{2}\right)^{1/2}}
\]
for $\vx\in\Omega$ and $t\geq t_0$. Combined with \eqref{eq:Q1Q2_estim}, this furnishes the bound on $Q$. Continuity of $v$ (as follows from \eqref{eq:v_sol_v0v1_Poiss} due to the regularity assumptions on the initial data) implies that the bound can be extended to $t\geq 0$. Therefore, we conclude \eqref{eq:v_only_estim}.

\begin{itemize}
\item{\bf Case $d=3$.}
\end{itemize}
Kirchhoff's formula \cite[Par. 2.4.1 (c)]{Evans} (see also \eqref{eq:v_sol_Kirch_2} below) yields
\begin{equation}\label{eq:v_sol_v0v1_Kirch}
v\left(\vx,t\right)=\frac{1}{4\pi}\int_{\left|\vs\right|=1}\left[t \,
  v_1\left(\vx+\vs c_0 t\right)+v_0\left(\vx+\vs c_0 t\right)+ t c_0 \vs\cdot\nabla v_0\left(\vx+\vs c_0 t\right)\right]d\sigma_\vs.
\end{equation}
Hence, estimating each term in \eqref{eq:v_sol_v0v1_Kirch} using \eqref{eq:rt_rho_bnd} and \eqref{eq:v0_dv0_v1_ass_d3} directly implies \eqref{eq:v_only_estim}.
\end{proof}

\subsection{Proof of Lemma~\ref{lem:spec_dec}}
{Since by our assumptions $v_0\in C^7(\mR^d)$ and $v_1\in C^6(\mR^d)$, Theorems 2 and 3 in \cite[Par. 2.4.1]{Evans} {(applied to $v$ and its derivatives)} 
imply that the solution~$v$ to~\eqref{eq:wave_c0} satisfies
$v\in C^{6}(\mR^d\times\mR_{+})$\, and thus
$\partial_t v\in C^5(\mR^d\times\mR_{+})$.}

Similarly to Lemma~\ref{lem:gen_dec}, in the main body of the proof, we shall prove the estimate
\begin{equation}\label{eq:v_dec2show_2}
\left|v\left(\vx,t\right)\right|\leq \frac{C}{\left(1+t^2\right)^{1/2}}, \hspace{1em}\vx\in\Omega,\hspace{0.5em} t\geq 0,
\end{equation}
for some constant $C>0$, and the estimate of the other terms 
in~\eqref{eq:sol_der_bnds} is discussed at the end. Note that the
    estimate \eqref{eq:v_dec2show_2} can be proven under weaker
    regularity assumptions on $v_0$ and $v_1$ than those in the
    formulation of the present lemma but, on the other hand, it also
    holds true for a more general form of the initial conditions
    (allowing for the presence of faster decaying extra terms). This
    class of initial conditions shall be described in the definition
    of $\mathcal{B}$ given by \eqref{eq:B-set}.


We extend the function $Y$ from the sphere $\mathbb{S}^{d-1}$ to the
    spherical shell (annulus for $d=2$) with the inner and outer radii
    $\frac12$ and $\frac32$, respectively. We choose a constant
    extension in the radial direction. This extension makes $\nabla Y$ well defined on $\mathbb{S}^{d-1}$. Moreover, $\nabla Y(\phi)$ is a tangent vector to $\mathbb{S}^{d-1}$ for each (radial vector)
$\phi\in\mathbb{S}^{d-1}$, and hence
\begin{equation}\label{tang-vector1}
	\phi\cdot \nabla Y(\phi)=0.
\end{equation}

\begin{itemize}
\item{\bf Case $d=2$.}
\end{itemize}

The solution is given explicitly by Poisson's formula  \cite[Par. 2.4.1 (c)]{Evans} 
\begin{align}
v\left(\mathbf{x},t\right)= & \frac{1}{2\pi}\left[\int_{0}^{1}\frac{rt}{\left(1-r^{2}\right)^{1/2}}\int_{\left|\mathbf{s}\right|=1}v_1\left(\mathbf{x}+\mathbf{s}rc_0t\right)d\sigma_{\mathbf{s}}dr\right.\label{eq:v_sol_v0v1}\\
& \left.+\partial_{t}\left(\int_{0}^{1}\frac{rt}{\left(1-r^{2}\right)^{1/2}}\int_{\left|\mathbf{s}\right|=1}v_0\left(\mathbf{x}+\mathbf{s}rc_0t\right)d\sigma_{\mathbf{s}}dr\right)\right].\nonumber
\end{align}
Upon insertion of \eqref{eq:spec_ICs} into \eqref{eq:v_sol_v0v1},
{and setting} $\rho:=\left|\mathbf{x}+\mathbf{s}rc_0t\right|$, $\phi:=\frac{\mathbf{x}+\mathbf{s}rc_0t}{\left|\mathbf{x}+\mathbf{s}rc_0t\right|}$,
we rearrange the terms to obtain  
	\begin{align}\label{eq:v_sol_2D}
	v\left(\mathbf{x},t\right)=&\frac{1}{2\pi}\int_{\left|\mathbf{s}\right|=1}\int_{0}^{1}\frac{rt}{\left(1-r^{2}\right)^{1/2}}\left[\partial_{t}A\left(\rho\right)-c_0A^{\prime}\left(\rho\right)\right]Y\left(\phi\right)drd\sigma_{\mathbf{s}}\\
	& +\frac{1}{2\pi}\int_{\left|\mathbf{s}\right|=1}\int_{0}^{1}\frac{r}{\left(1-r^{2}\right)^{1/2}}A\left(\rho\right)\left[t\partial_{t}Y\left(\phi\right)+Y\left(\phi\right)\right]drd\sigma_{\mathbf{s}}\nonumber.
	\end{align}
{As $A\left(\rho\right)\equiv 0$ for $\rho\leq\rho_0$, then $\phi$ is well-defined whenever it appears {in the above integrands} 
(i.e.~$\rho=\left|\vx+\vs rc_{0}t\right|>\rho_0$ whenever $A\left(\rho\right)\ne 0$).
%
}

Here and in the sequel, we use the following notation to avoid having too many brackets: $\partial_t A(\rho):=\partial_t (A(\rho))$, $\partial_t Y(\phi):=\partial_t (Y(\phi))$, $\nabla Y(\phi):=(\nabla Y)(\phi)$.

Since $\left|\vs\right|=1$, we have 
	\begin{equation}\label{eq:A_drdt}
	\partial_t A\left(\rho\right)=rc_0 A^\prime\left(\rho\right)
        \frac{\vx\cdot\vs+rc_0t}{\left|\vx+\vs rc_0
            t\right|},\hspace{2em}
{r \partial_r A\left(\rho\right) = t \partial_t A\left(\rho\right),}
	\end{equation}
	\begin{equation}\label{eq:Y_dt}
          \begin{split}
	&\partial_{t}Y\left(\phi\right)=\frac{rc_{0}\mathbf{s}\cdot\nabla
          Y\left(\phi\right)}{\left|\mathbf{x}+\mathbf{s}rc_{0}t\right|},
        \qquad
        {r\partial_r Y\left(\phi\right)
          =t\partial_t Y\left(\phi\right)},
        \end{split}
	\end{equation}
where we have used \eqref{tang-vector1}. 
Note that{, using~\eqref{eq:A_drdt},} we can rewrite, for $t>0$,
	\begin{align}\label{eq:A_der_comb}
	\partial_{t}A\left(\rho\right)-c_{0}A^{\prime}\left(\rho\right)=&\frac{1}{t}\left(r-1\right)\partial_{r}A\left(\rho\right)-\frac{1}{t}\left(\frac{\left|\mathbf{x}+\mathbf{s}rc_{0}t\right|}{\mathbf{x}\cdot\mathbf{s}+rc_{0}t}-1\right)\partial_{r}A\left(\rho\right)\nonumber\\
	=&\frac{1}{t}\left(r-1\right)\partial_{r}A\left(\rho\right)+c_0\left(\frac{\vx\cdot\vs+rc_0t}{\left|\vx+\vs rc_0 t\right|}-1\right)A^\prime\left(\rho\right).
	\end{align}
{Note that $\mathbf{x}\cdot\mathbf{s}+rc_{0}t$ cannot be zero on $\supp A$ if $\vx\in\Omega$ {and $\vs\in\mathbb{S}^1$}. Indeed, $\mathbf{x}\cdot\mathbf{s}+rc_{0}t{=0}$ would imply
$$
  \rho^2=|\vx|^2+2rc_0t\vx\cdot\vs+r^2c_0^2t^2=|\vx|^2-r^2c_0^2t^2\le|\vx|^2<\rho_0^2,
$$
and hence $A(\rho)=0$.
}
Plugging \eqref{eq:A_der_comb} into \eqref{eq:v_sol_2D}, we observe
that the term with $\partial_{r}A$ can be integrated by parts {in
  the variable $r$,} due to the cancellation of the singularity at $r=1$. 
	In doing so, both boundary terms at $r=0$ and $r=1$, respectively, vanish.
With some simplifications that employ \eqref{eq:Y_dt}
and the identity
	\[
\left(\frac{r\left(1-r\right)}{\left(1-r^{2}\right)^{1/2}}\right)^{\prime}+\frac{r}{\left(1-r^{2}\right)^{1/2}}=\frac{1}{\left(1-r^{2}\right)^{1/2}\left(1+r\right)},
	\]
we arrive at
\begin{align}\label{eq:v_sol_2D_simpl}
v\left(\mathbf{x},t\right)= & \frac{1}{2\pi}\int_{\left|\mathbf{s}\right|=1}\int_{0}^{1}\frac{1}{\left(1-r^{2}\right)^{1/2}}\left[\frac{1}{1+r}A\left(\rho\right)Y\left(\phi\right)\right.\\
	&\left.+rc_{0}t\left(\frac{\vx\cdot\vs+rc_0t}{\left|\vx+\vs rc_0 t\right|}-1\right)A^{\prime}\left(\rho\right)Y\left(\phi\right)\right]drd\sigma_{\mathbf{s}}+Q\left(\mathbf{x},t\right).\nonumber
	\end{align}
Here, we have set
	\begin{align}\label{eq:Q1_def}
	Q\left(\mathbf{x},t\right):= &  \frac{1}{2\pi}\int_{\left|\mathbf{s}\right|=1}\int_{0}^{1}\frac{rc_{0}t}{\left(1-r^{2}\right)^{1/2}} 
\,\frac{1}{\left|\mathbf{x}+\mathbf{s}rc_{0}t\right|}A\left(\rho\right)\mathbf{s}\cdot\nabla Y\left(\phi\right) 
drd\sigma_{\mathbf{s}} \nonumber \\
	= & -\frac{1}{2\pi}\int_{\left|\mathbf{s}\right|=1}\int_{0}^{1}\frac{A\left(\rho\right)}{\left(1-r^{2}\right)^{1/2}} 
	\,\frac{\mathbf{x}\cdot\nabla Y\left(\phi\right)}{\left|\mathbf{x}+\mathbf{s}rc_{0}t\right|} 
	drd\sigma_{\mathbf{s}} ,
	\end{align}
where we used
\begin{equation}\label{tang-vector2}
	\mathbf{x}\cdot \nabla Y(\phi)=-rc_0t\mathbf{s}\cdot \nabla Y(\phi),
\end{equation}
which follows from \eqref{tang-vector1}.

For later reference, we note that, for $rt\gg 1$, uniformly in $\vx\in\Omega$,
$\vs\in\mathbb{S}^{{1}}$,
\begin{equation}\label{eq:Q1_estim_prelim}
-\frac{\mathbf{x}\cdot\nabla Y\left(\phi\right)}{\left|\mathbf{x}+\mathbf{s}rc_{0}t\right|}=\mathcal{O}\left(\frac{1}{rt}\right),
\end{equation}
where we used $|\mathbf{x}|\le\rho_0$ and
$$
  \rho\ge|\mathbf{s}rc_{0}t|-|\mathbf{x}|\ge rc_0t-\rho_0\ge\frac12 rc_0t\qquad \mbox{ for } rt\ge\frac{2\rho_0}{c_0}.
$$

By setting
\begin{equation}\label{eq:P1_def}
P_{1}\left(\mathbf{x},t\right):=\frac{1}{2\pi}\int_{0}^{1}\frac{1}{\left(1-r^{2}\right)^{1/2}}\int_{\left|\mathbf{s}\right|=1}rc_{0}t\left(\frac{\vx\cdot\vs+rc_0t}{\left|\vx+\vs rc_0 t\right|}-1\right)A^{\prime}\left(\rho\right)Y\left(\phi\right)d\sigma_{\mathbf{s}}dr,
\end{equation}
\begin{equation}\label{eq:P2_def}
P_{2}\left(\mathbf{x},t\right):=\frac{1}{2\pi}\int_{0}^{1}\frac{1}{\left(1-r^{2}\right)^{1/2}\left(1+r\right)}\int_{\left|\mathbf{s}\right|=1}A\left(\rho\right)Y\left(\phi\right)d\sigma_{\mathbf{s}}dr,
\end{equation}
{we can rewrite~\eqref{eq:v_sol_2D_simpl} as}
\begin{equation}\label{eq:v_sol_2D_abbr}
v\left(\mathbf{x},t\right)=P_1\left(\mathbf{x},t\right)+P_2\left(\mathbf{x},t\right)+Q\left(\mathbf{x},t\right).
\end{equation}


{As in the proof of Lemma~\ref{lem:gen_dec}, we shall prove that
  there exists $t_0>0$ such that the bounds
\begin{equation}\label{eq:boundP1P2Q1Q2}
  \left|P_{1}\left(\mathbf{x},t\right)\right|\le\dfrac{\widetilde C_0}{t},\quad
  \left|P_{2}\left(\mathbf{x},t\right)\right|\le\dfrac{\widetilde C_1}{t},\quad
  \left|Q\left(\mathbf{x},t\right)\right|\le\dfrac{\widetilde C_2}{t},\quad
\end{equation}
are valid uniformly in $\vx\in\Omega$ with some constants $\widetilde C_0, \widetilde C_1, \widetilde C_2 
 >0$ for
any $t\ge t_0$. Then, due to the uniform boundedness of the solution $v$ 
on $\Omega\times[0,t_0]$ (see~\eqref{eq:v_sol_v0v1}), the bounds in~\eqref{eq:boundP1P2Q1Q2} imply~\eqref{eq:v_dec2show_2}.
}

The functions $A\left(\rho\right)$ and $A^\prime\left(\rho\right)$ may be different from zero only for $\rho>\rho_0$, i.e. for $rc_{0}t>\left(\left(\mathbf{x}\cdot\mathbf{s}\right)^{2}+\rho_{0}^{2}-\left|\vx\right|^{2}\right)^{1/2}-\mathbf{x}\cdot\mathbf{s}$, obtained from $\rho^2=(rc_0t+\mathbf{x}\cdot\mathbf{s})^2-(\mathbf{x}\cdot\mathbf{s})^2 +|\mathbf{x}|^2>\rho_0^2$. Thus, the integration range in the $r$ variable in each term of \eqref{eq:v_sol_2D_abbr} 
effectively
reduces from $\left(0,1\right)$ to $\left(a_1/t,1\right)$ with $a_1>0$
defined in~\eqref{eq:a1_def}. {With this argument, we are implicity
assuming that $t\ge a_1$. We will actually prove~\eqref{eq:boundP1P2Q1Q2} with $t_0:=2a_1$.}



\bigskip \underline{{Estimate of $P_1$ for $t\ge t_0$}}:
	
Let us introduce 
\begin{equation}\label{eq:F1_def}
F_{1}\left(\mathbf{x},rt\right):=\frac{c_{0}}{2\pi}\int_{\left|\mathbf{s}\right|=1}\left(rt\right)^{5/2}\left(\frac{\mathbf{x}\cdot\mathbf{s}+rc_{0}t}{\left|\mathbf{x}+\mathbf{s}rc_{0}t\right|}-1\right)A^{\prime}\left(\rho\right)Y\left(\phi\right)d\sigma_{\mathbf{s}},
\end{equation}
so that we can write
\begin{align*}
P_{1}\left(\mathbf{x},t\right)= & \int_{a_{1}/t}^{1}\frac{1}{\left(1-r^{2}\right)^{1/2}\left(rt\right)^{3/2}}F_{1}\left(\mathbf{x},rt\right)dr = \int_{a_{1}/t}^{1/2}\dots\phantom{.}+\int_{1/2}^{1}\dots\nonumber\\
=: & P_{1,1}\left(\mathbf{x},t\right)+P_{1,2}\left(\mathbf{x},t\right),\nonumber
\end{align*}
assuming $a_1/t \leq 1/2$, i.e. $t\geq 2a_1$.

Since 
we consider only $\rho=\left|\vx+\vs rc_0 t\right|>\rho_0>0$ (as the integrand of $F_1\left(\vx,rt\right)$ vanishes otherwise),
the denominators in \eqref{eq:near_id_estim2} and \eqref{eq:F1_def}
are bounded away from zero. Moreover, since 
\begin{equation}\label{rho-est}
  rc_0t=\left|\vs   rc_0t\right|\leq \rho+\left|\vx\right|\leq\rho+\rho_0,
\end{equation} 
we have,
uniformly for $\vx\in\Omega$,
$\vs\in\mathbb{S}^{{1}}$,
\begin{align}\label{eq:A_rt_estim}
\left(rt\right)^{1/2}\left|A^\prime\left(\rho\right)\right|\leq&\begin{cases}
\left(\frac{\rho+\rho_0}{c_0}\right)^{1/2}\left\Vert A^\prime\right\Vert_{L^\infty\left(\mathbb{R}_+\right)}, \hspace{1em} 0\leq\rho\leq\rho_1,\\
\left(\frac{rt}{\rho}\right)^{1/2}{\left(\frac{\omega}{c_0}+\frac{1}{2\rho_1}\right)}\leq\left(\frac{\rho+\rho_0}{c_0\rho}\right)^{1/2}{\left(\frac{\omega}{c_0}+\frac{1}{2\rho_1}\right)}, \hspace{1em} \rho>\rho_1,
\end{cases}\\
\leq& C_0,\hspace{1em}\rho>0,\nonumber
\end{align}
for some constant $C_0>0$. This, together with \eqref{eq:near_id_estim2},
implies that
\[
\underset{\mathbf{x}\in\Omega}{\sup}\left\Vert F_{1}\left(\mathbf{x},\cdot\right)\right\Vert _{L^{\infty}\left(a_{1},\infty\right)}=:C_{1}<\infty,
\]
for some constant $C_1>0$.
Therefore, we can estimate, for {$t\ge t_0=2a_1$} {and
  $\vx\in\Omega$},
\[
\left|P_{1,1}\left(\mathbf{x},t\right)\right|\leq 2^{1/2}\frac{C_1}{t^{3/2}}\int_{a_1/t}^{1/2}\frac{dr}{r^{3/2}}<
2^{1/2}\frac{C_1}{t^{3/2}}\int_{a_1/t}^{\infty}\frac{dr}{r^{3/2}}=
\frac{2^{3/2}C_1}{a_1^{1/2}t},
\]
\[
\left|P_{1,2}\left(\mathbf{x},t\right)\right|\leq 2^{3/2}\frac{C_1}{t^{3/2}}\int_{1/2}^{1}\frac{dr}{\left(1-r\right)^{1/2}}=\frac{4C_1}{t^{3/2}},
\]
and thus we have {the bound for $P_1$ in~\eqref{eq:boundP1P2Q1Q2}
  with $t_0=2a_1$ and}
some constant ${\widetilde{C}_0}>0$.

\bigskip \underline{{Estimate of $Q$ for $t\ge t_0$}}:

As above, we note that all the denominators in \eqref{eq:Q1_def} are bounded away from zero. Therefore, by setting 
\begin{align}\label{eq:F2_def}
F_{2}\left(\mathbf{x},rt\right):=&-\frac{\left(rt\right)^{3/2}}{2\pi}\int_{\left|\mathbf{s}\right|=1} 
\frac{\mathbf{x}\cdot\nabla Y\left(\phi\right)}{\left|\mathbf{x}+\mathbf{s}rc_{0}t\right|}
A\left(\rho\right)d\sigma_{\mathbf{s}}, 
\end{align}
and recalling \eqref{eq:Q1_estim_prelim} and~\eqref{rho-est}, we have
\begin{equation}\label{eq:F2_F3_estim}
\underset{\mathbf{x}\in\Omega}{\sup}\left\Vert F_{2}\left(\mathbf{x},\cdot\right)\right\Vert _{L^{\infty}\left(a_{1},\infty\right)}=:C_{2}<\infty.
\end{equation}
Consequently, we estimate as before, for $\vx\in\Omega$, {$t\ge t_0=2a_1$},
\begin{equation*}
\left|Q\left(\mathbf{x},t\right)\right|\leq  2^{1/2}\frac{C_2}{t^{3/2}}\int_{a_1/t}^{1/2}\frac{dr}{r^{3/2}}+2^{3/2}\frac{C_2}{t^{3/2}}\int_{1/2}^{1}\frac{dr}{\left(1-r\right)^{1/2}}<\frac{\widetilde{C}_2}{t},
\end{equation*}
with some constant $\widetilde{C}_2 
>0$.
{This proves the bound for $Q$ 
  in~\eqref{eq:boundP1P2Q1Q2} again
  with $t_0=2a_1$.}

\bigskip \underline{{Estimate of $P_2$ for $t\ge t_0$}}:


{For the term $P_2$, we proceed as in the estimate
  of the term $P$ in Lemma~\ref{lem:gen_dec}.}
Let us rewrite \eqref{eq:P2_def} as
\begin{align}\label{eq:P2_estim_prelim}
P_{2}\left(\mathbf{x},t\right)= & \frac{1}{2\pi}\int_{0}^{1}\frac{e^{ir\omega t}}{\left(1-r\right)^{1/2}}\int_{\left|\mathbf{s}\right|=1}\left[\frac{e^{-ir\omega t}}{\left(1+r\right)^{3/2}}A\left(\rho\right)Y\left(\phi\right)\right.\\
 & \left.-\frac{e^{-i\omega t}}{2^{3/2}}A\left(\left|\mathbf{x}+\mathbf{s}c_{0}t\right|\right)Y\left(\frac{\mathbf{x}+\mathbf{s}c_{0}t}{\left|\mathbf{x}+\mathbf{s}c_{0}t\right|}\right)\right]d\sigma_{\mathbf{s}}dr\nonumber\\
 & +\frac{1}{2^{5/2}\pi}\int_{\left|\mathbf{s}\right|=1}A\left(\left|\mathbf{x}+\mathbf{s}c_{0}t\right|\right)Y\left(\frac{\mathbf{x}+\mathbf{s}c_{0}t}{\left|\mathbf{x}+\mathbf{s}c_{0}t\right|}\right)d\sigma_{\mathbf{s}}\int_{0}^{1}\frac{e^{-i\left(1-r\right)\omega t}}{\left(1-r\right)^{1/2}}dr\nonumber\\
=: & P_{2,1}\left(\mathbf{x},t\right)+P_{2,2}\left(\mathbf{x},t\right).\nonumber
\end{align}
We start with
\[
P_{2,2}\left(\mathbf{x},t\right)=\frac{1}{t^{1/2}}F_{4}\left(\mathbf{x},t\right)\int_{0}^{1}\frac{e^{-ir\omega t}}{r^{1/2}}dr,
\]
where we made a change of variable $r\mapsto\left(1-r\right)$ and introduced
\begin{equation}\label{eq:F4_def}
F_{4}\left(\mathbf{x},t\right):=\frac{t^{1/2}}{2^{5/2}\pi}\int_{\left|\mathbf{s}\right|=1}A\left(\left|\mathbf{x}+\mathbf{s}c_{0}t\right|\right)Y\left(\frac{\mathbf{x}+\mathbf{s}c_{0}t}{\left|\mathbf{x}+\mathbf{s}c_{0}t\right|}\right)d\sigma_{\mathbf{s}}.
\end{equation}
Similarly to \eqref{eq:A_rt_estim}, but with setting $r=1$, we have $t^{1/2}A\left(\left|\vx+\vs c_0 t\right|\right)\leq {C_{00}}$ for some constant ${C_{00}}>0$, and thus
\begin{equation}\label{eq:F4_estim}
\underset{\mathbf{x}\in\Omega}{\sup}\left\Vert F_{4}\left(\mathbf{x},\cdot\right)\right\Vert _{L^{\infty}\left(\mathbb{R}_+\right)}=:C_{4}<\infty.
\end{equation}
Hence, employing Lemma \ref{lem:app_osc_int_estim}, we obtain uniformly for $\vx\in\Omega$ and sufficiently large $t>0$,
\begin{align}\label{eq:P22_estim}
\left|P_{2,2}\left(\mathbf{x},t\right)\right|\leq\frac{\widetilde{C}_4}{t}.
\end{align}

\bigskip \underline{{Decay of $P_{2,1}$}}. 
To deal with $P_{2,1}$, {we note that the integrand is a smooth
function of $r$ in $[0,1)$ and it behaves like $(1-r)^{1/2}$ as $r\to
1$.
Integrating by parts in the $r$ variable with $e^{ir\omega t}dr$ as
differential, both boundary terms vanish (recall also
that $A\left(\left|\vx\right|\right)\equiv 0$ for $\vx\in\Omega$).}
We thus have
\begin{align}\label{eq:P21_def}
P_{2,1}\left(\mathbf{x},t\right)= & -\frac{1}{2\pi i\omega t}\int_{0}^{1}\int_{\left|\mathbf{s}\right|=1}e^{ir\omega t}\partial_{r}\left(\frac{1}{\left(1-r\right)^{1/2}}\left[\frac{e^{-ir\omega t}}{\left(1+r\right)^{3/2}}A\left(\rho\right)Y\left(\phi\right)\right.\right.\\
 & \left.\left.-\frac{e^{-i\omega t}}{2^{3/2}}A\left(\left|\mathbf{x}+\mathbf{s}c_{0}t\right|\right)Y\left(\frac{\mathbf{x}+\mathbf{s}c_{0}t}{\left|\mathbf{x}+\mathbf{s}c_{0}t\right|}\right)\right]\right)d\sigma_{\mathbf{s}}dr\nonumber\\
 =& I_1\left(\mathbf{x},t\right)+I_2\left(\mathbf{x},t\right)+I_3\left(\mathbf{x},t\right)+I_4\left(\mathbf{x},t\right),\nonumber
\end{align}
where
\begin{equation}\label{eq:I1_def}
I_{1}\left(\mathbf{x},t\right):=\frac{i}{2\pi\omega t}\int_{0}^{1}\int_{\left|\mathbf{s}\right|=1}\frac{1}{\left(1-r\right)^{1/2}\left(1+r\right)^{3/2}}Y\left(\phi\right)e^{ir\omega t}\partial_{r}\left(e^{-ir\omega t}A\left(\rho\right)\right)d\sigma_{\mathbf{s}}dr,
\end{equation}
\begin{equation}\label{eq:I2_def}
I_{2}\left(\mathbf{x},t\right):=\frac{i}{2\pi\omega t}\int_{0}^{1}\int_{\left|\mathbf{s}\right|=1}\frac{1}{\left(1-r\right)^{1/2}\left(1+r\right)^{3/2}}A\left(\rho\right)\partial_{r}Y\left(\phi\right)d\sigma_{\mathbf{s}}dr,
\end{equation}
\begin{align}\label{eq:I3_def}
I_{3}\left(\mathbf{x},t\right):=&-\frac{3i}{4\pi\omega t}\int_{0}^{1}\int_{\left|\mathbf{s}\right|=1}\frac{1}{\left(1-r\right)^{1/2}\left(1+r\right)^{5/2}}A\left(\rho\right)Y\left(\phi\right)d\sigma_{\mathbf{s}}dr,
\end{align}
\begin{align}\label{eq:I4_def}
I_{4}\left(\mathbf{x},t\right):=&\frac{i}{4\pi\omega t}\int_{0}^{1}\int_{\left|\mathbf{s}\right|=1}\frac{1}{\left(1-r\right)^{3/2}}\left[\frac{1}{\left(1+r\right)^{3/2}}A\left(\rho\right)Y\left(\phi\right)\right.\\
 &\left.-\frac{e^{-i\left(1-r\right)\omega t}}{2^{3/2}}A\left(\left|\mathbf{x}+\mathbf{s}c_{0}t\right|\right)Y\left(\frac{\mathbf{x}+\mathbf{s}c_{0}t}{\left|\mathbf{x}+\mathbf{s}c_{0}t\right|}\right)\right]d\sigma_{\mathbf{s}}dr.\nonumber
\end{align}

\bigskip \underline{{Decay of $I_1$}}. 
For the term $I_1$ observe that, using \eqref{eq:A_drdt}, we have
\begin{align*}
\frac{r^{3/2}t^{1/2}}{c_0}e^{ir\omega t}\partial_{r}\left(e^{-ir\omega t}A\left(\rho\right)\right)=&\left(rt\right)^{3/2}\left(A^{\prime}\left(\rho\right)-\frac{i\omega}{c_{0}}A\left(\rho\right)\right)\\
&-\left(rt\right)^{3/2}A^{\prime}\left(\rho\right)\left(1-\frac{\mathbf{x}\cdot\mathbf{s}+rc_{0}t}{\left|\mathbf{x}+\mathbf{s}rc_{0}t\right|}\right),
\end{align*}
where both terms on the right-hand side are uniformly bounded for $rt>0$, $\vx\in\Omega$, $\left|\vs\right|=1$, due to \eqref{eq:near_id_estim2}, \eqref{rho-est}, and the assumption on the form of $A$. Therefore, {for}
\begin{equation*}
F_{5}\left(\mathbf{x},rt\right):=\frac{ic_{0}}{2\pi\omega}\int_{\left|\mathbf{s}\right|=1}Y\left(\phi\right)\frac{r^{3/2}t^{1/2}}{c_0}e^{ir\omega t}\partial_{r}\left(e^{-ir\omega t}A\left(\rho\right)\right)d\sigma_{\mathbf{s}},
\end{equation*}
{we have}
\[
\underset{\mathbf{x}\in\Omega}{\sup}\left\Vert F_{5}\left(\mathbf{x},\cdot\right)\right\Vert _{L^{\infty}\left(a_{1},\infty\right)}=:C_{5}<\infty.
\]
Since both $A$, $A^\prime$ vanish on $\left[0,\rho_0\right]$, the
interval of integration for the $r$-integrals in each of \eqref{eq:I1_def}--\eqref{eq:I3_def} reduces to $\left(a_1/t,1\right)$ (see the discussion before \eqref{eq:a1_def}). Hence
we can estimate {$I_1$ in}~\eqref{eq:I1_def} for {$t\ge t_0=2a_1$} as
\begin{align}\label{eq:I1_estim}
\left|I_{1}\left(\mathbf{x},t\right)\right|=&\frac{1}{t^{3/2}}\left|\int_{a_1/t}^{1}\frac{1}{r^{3/2} \left(1-r\right)^{1/2}\left(1+r\right)^{3/2}}F_{5}\left(\mathbf{x},rt\right)dr\right|\\
\leq&\frac{2^{1/2}C_5}{t^{3/2}}\int_{a_1/t}^{1/2}\frac{dr}{r^{3/2}}+\frac{2^{3/2}C_5}{t^{3/2}}\int_{1/2}^{1}\frac{dr}{\left(1-r\right)^{1/2}}
\leq\frac{\widetilde{C}_5}{t}\nonumber
\end{align}
with some constant $\widetilde{C}_5>0$.

\bigskip \underline{{Decay of $I_2$, $I_3$}}. 
In a similar but simpler fashion, we can estimate the terms $I_2$ and $I_3$. To this end we introduce

\begin{equation*}
F_{6}\left(\mathbf{x},rt\right):=-\frac{3i\left(rt\right)^{1/2}}{4\pi\omega}\int_{\left|\mathbf{s}\right|=1}A\left(\rho\right)Y\left(\phi\right)d\sigma_{\mathbf{s}},
\end{equation*}
which satisfies
\[
\left|F_{6}\left(\mathbf{x},rt\right)\right|\leq\frac{3}{4\pi\omega}\int_{\left|\mathbf{s}\right|=1}\left(\frac{rt}{\rho}\right)^{1/2}\rho^{1/2}\left|A\left(\rho\right)\right|\left|Y\left(\phi\right)\right|d\sigma_{\mathbf{s}},
\]
and thus
\begin{equation}\label{eq:F7_def}
\underset{\mathbf{x}\in\Omega}{\sup}\left\Vert F_{6}\left(\mathbf{x},\cdot\right)\right\Vert _{L^{\infty}\left(\mathbb{R}_+\right)}=:C_{6}<\infty.
\end{equation}
Then, using~\eqref{eq:F2_F3_estim} and \eqref{eq:F7_def}, we obtain for {$t\ge t_0=2a_1$},
\begin{align}\label{eq:I2_estim}
\left|I_{2}\left(\mathbf{x},t\right)\right|=&\frac{1}{\omega t^{5/2}}\left|\int_{a_1/t}^{1}\frac{1}{r^{5/2} \left(1-r\right)^{1/2}\left(1+r\right)^{3/2}}F_{2}\left(\mathbf{x},rt\right)dr\right|\\
\leq&\frac{2^{1/2}C_2}{\omega t^{5/2}}\int_{a_1/t}^{1/2}\frac{dr}{r^{5/2}}+\frac{2^{5/2}C_2}{\omega t^{5/2}}\int_{1/2}^{1}\frac{dr}{\left(1-r\right)^{1/2}}
\leq\frac{\bar{C}_2}{t},\nonumber
\end{align}
\begin{align}\label{eq:I3_estim}
\left|I_{3}\left(\mathbf{x},t\right)\right|=&\frac{1}{t^{3/2}}\left|\int_{a_1/t}^{1}\frac{1}{r^{1/2}\left(1-r\right)^{1/2}\left(1+r\right)^{5/2}}F_{6}\left(\mathbf{x},rt\right)dr\right|\\
\leq&\frac{2^{1/2}C_6}{t^{3/2}}\left(\int_{a_1/t}^{1/2}\frac{dr}{r^{1/2}}+\int_{1/2}^{1}\frac{dr}{\left(1-r\right)^{1/2}}\right)
\leq\frac{\widetilde{C}_6}{t^{3/2}}\nonumber
\end{align}
with some constants $\bar{C}_2$, $\widetilde{C}_6>0$. 

\bigskip \underline{{Decay of $I_4$}}. To treat the term $I_4$, we introduce
\begin{align}\label{eq:F8_def_tild}
\widetilde{F}_{7}\left(\mathbf{x},r,t\right):=&\int_{\left|\mathbf{s}\right|=1}\left[\frac{1}{\left(1+r\right)^{3/2}}A\left(\rho\right)Y\left(\phi\right)\right.\\
 &\left.-\frac{e^{-i\left(1-r\right)\omega t}}{2^{3/2}}A\left(\left|\mathbf{x}+\mathbf{s}c_{0}t\right|\right)Y\left(\frac{\mathbf{x}+\mathbf{s}c_{0}t}{\left|\mathbf{x}+\mathbf{s}c_{0}t\right|}\right)\right]d\sigma_{\mathbf{s}},\nonumber
\end{align}
\begin{equation}\label{eq:F8_def}
F_{7}\left(\mathbf{x},r,t\right):=\frac{1}{1-r}\widetilde{F}_{7}\left(\mathbf{x},r,t\right).
\end{equation}
Using $\widetilde{F}_{7}\left(\mathbf{x},1,t\right)=0$, \eqref{eq:F8_def} can be rewritten as
\begin{equation}\label{eq:F8_def_ext}
F_{7}\left(\mathbf{x},r,t\right)=-\frac{1}{1-r}\left(\widetilde{F}_{7}\left(\mathbf{x},1,t\right)-\widetilde{F}_{7}\left(\mathbf{x},r,t\right)\right)=-\frac{1}{1-r}\int_r^1 \partial_r\widetilde{F}_{7}\left(\mathbf{x},\tau,t\right)d\tau.
\end{equation}
From this we estimate for $r\in\left(1-a_1/t,1\right)$,
\begin{align*}
\left|F_{7}\left(\mathbf{x},r,t\right)\right|\leq&\left\Vert\partial_r\widetilde{F}_{7}\left(\mathbf{x},\cdot,t\right)\right\Vert_{L^\infty\left(1-a_1/t,1\right)}\\
\leq& \underset{r\in\left(1-a_1/t,\, 1\right)}{\sup}
\int_{\left|\mathbf{s}\right|=1}\left[\frac{3}{2\left(1+r\right)^{5/2}}\left|A\left(\rho\right)\right|\left|Y\left(\phi\right)\right|+\frac{1}{\left(1+r\right)^{3/2}}\left|\partial_r A\left(\rho\right)\right|\left|Y\left(\phi\right)\right|\right.\nonumber\\
&+\left.\frac{1}{\left(1+r\right)^{3/2}}\left|A\left(\rho\right)\right|\left|\partial_r Y\left(\phi\right)\right|+\frac{\omega t}{2^{3/2}}\left|A\left(\left|\vx+\vs c_0 t\right|\right)\right|\left|Y\left(\frac{\vx+\vs c_0 t}{\left|\vx+\vs c_0 t\right|}\right)\right|\right]d\sigma_{\mathbf{s}}.\nonumber
\end{align*}
Therefore, employing \eqref{rho-est}, \eqref{eq:A_drdt} and \eqref{eq:Y_dt}, {taking into account the behaviour of $A(\rho)$ and $A^\prime(\rho)$ for $\rho>\rho_1$,} we deduce that 
\begin{equation}\label{eq:C8_def}
\underset{\mathbf{x}\in\Omega,\hspace{0.3em}r\in\left(1-a_1/t,\, 1\right),\hspace{0.3em}t>2a_1}{\sup}\left| F_{7}\left(\mathbf{x},r,t\right)/t^{1/2} \right|=:C_{7}<\infty.
\end{equation}
{Moreover,} for $r\in\left(1/2,1-a_1/t\right)$ {and
  $\epsilon\in\left(0,1/2\right]$, taking into
  account~\eqref{eq:F8_def}, we have}
\[
\left|\left(1-r\right)^{1/2+\epsilon}t^{\epsilon}F_{7}\left(\mathbf{x},r,t\right)\right|= \left|\widetilde{F}_7\left(\mathbf{x},r,t\right)\right|\frac{t^\epsilon}{\left(1-r\right)^{1/2-\epsilon}} \leq \left|\widetilde{F}_7\left(\mathbf{x},r,t\right)\right|\frac{t^{1/2}}{a_1^{1/2-\epsilon}}.
\]
Hence, {with $C_4$ and $C_6$ defined in~\eqref{eq:F4_estim}
  and~\eqref{eq:F7_def}, respectively,}
we obtain from~\eqref{eq:F8_def_tild} 
that, for any $\epsilon\in\left(0,1/2\right]$,
\begin{align}\label{eq:C9_def}
\underset{\mathbf{x}\in\Omega,\hspace{0.3em}r\in\left(1/2,\, 1-a_1/t\right),\hspace{0.3em}t>2a_1}{\sup}\left|\left(1-r\right)^{1/2+\epsilon}t^{\epsilon}F_{7}\left(\mathbf{x},r,t\right)\right|\leq& \frac{1}{a_1^{1/2-\epsilon}}\left(\frac{2^{1/2}4\pi\omega}{3}C_{6}+2\pi C_4\right)\\
=:&C_8<\infty.\nonumber
\end{align}
Altogether \eqref{eq:C9_def} and \eqref{eq:C8_def} imply that,
for {$t\ge t_0=2a_1$ and} $\epsilon\in\left(0,1/2\right]$,
\begin{align}\label{eq:I4_estim}
\left|I_{4}\left(\mathbf{x},t\right)\right|\leq&\frac{1}{4\pi\omega t}\int_{0}^{1}\frac{1}{\left(1-r\right)^{1/2}}\left|F_{7}\left(\mathbf{x},r,t\right)\right|dr\\
=&\frac{1}{4\pi\omega t}\left(\int_{0}^{1/2}\ldots+\int_{1/2}^{1-a_1/t}\ldots+\int_{1-a_1/t}^{1}\ldots\right)\nonumber\\
\leq&\frac{2^{3/2}}{t^{3/2}}\left(\frac{C_6}{3}\int_{0}^{1/2}\frac{dr}{r^{1/2}}+\frac{C_4}{4\omega}\right)+\frac{C_8}{4\pi\omega t^{1+\epsilon}}\int_{1/2}^{1-a_1/t}\frac{dr}{\left(1-r\right)^{1+\epsilon}}\nonumber\\
&+\frac{C_7}{4\pi\omega t^{1/2}}\int_{1-a_1/t}^{1}\frac{dr}{\left(1-r\right)^{1/2}}
\leq\frac{\widetilde{C}_7}{t}\nonumber
\end{align}
with some constant $\widetilde{C}_7>0$. 
Here, for the interval $\left(0,1/2\right)$, we have estimated the
integrand of $F_7$ directly, 
using again \eqref{eq:F7_def} and \eqref{eq:F4_estim}.

{From estimates~\eqref{eq:I1_estim}, \eqref{eq:I2_estim},
  \eqref{eq:I3_estim}, and \eqref{eq:I4_estim}, of $I_1$, $I_2$,
  $I_3$, and $I_4$, respectively, in decomposition~\eqref{eq:P21_def},
we obtain for $P_{2,1}$ the same estimate as~\eqref{eq:P22_estim} for
$P_{2,2}$. 

The estimate for $P_2$ in~\eqref{eq:boundP1P2Q1Q2} with $t_0=2a_1$ readily
follows from~\eqref{eq:P2_estim_prelim}. This completes the proof of~\eqref{eq:v_dec2show_2} in the case $d=2$.}

\begin{itemize}
\item{\bf Case $d=3$.}
\end{itemize}

In this case, the solution is given by Kirchhoff's formula  \cite[Par. 2.4.1 (c)]{Evans} 
\begin{equation}\label{eq:v_sol_Kirch_2}
v\left(\mathbf{x},t\right)=\frac{t}{4\pi}\int_{\left|\mathbf{s}\right|=1}\left[v_1\left(\mathbf{x}+\mathbf{s}c_0 t\right)+\partial_t v_0\left(\mathbf{x}+\mathbf{s}c_0 t\right)+\frac{1}{t}v_0\left(\mathbf{x}+\mathbf{s}c_0 t\right)\right]d\sigma_{\mathbf{s}}.
\end{equation}
The assumed form of the initial conditions yields
\begin{align}\label{eq:v_sol_Kirch_2_simpl}
v\left(\mathbf{x},t\right)=&\frac{t}{4\pi}\int_{\left|\mathbf{s}\right|=1}\left[\left(\frac{\vx\cdot\vs+c_0 t}{\left|\vx+\vs c_0 t\right|}-1\right) c_0 A^\prime\left(|\mathbf{x}+\mathbf{s}c_0 t|\right)Y\left(\frac{\mathbf{x}+\mathbf{s}c_0 t}{\left|\mathbf{x}+\mathbf{s}c_0 t\right|}\right)\right.\\
&\left.+A\left(|\mathbf{x}+\mathbf{s}c_0 t|\right)\partial_t Y\left(\frac{\mathbf{x}+\mathbf{s}c_0 t}{\left|\mathbf{x}+\mathbf{s}c_0 t\right|}\right) 
+\frac{1}{t}A\left(|\mathbf{x}+\mathbf{s}c_0 t|\right)Y\left(\frac{\mathbf{x}+\mathbf{s}c_0 t}{\left|\mathbf{x}+\mathbf{s}c_0 t\right|}\right)\right]d\sigma_{\mathbf{s}}.\nonumber
\end{align}
{For} $d=3$, {both} $A\left(|\mathbf{x}+\mathbf{s}c_0
  t|\right)$ {and} $A^\prime\left(|\mathbf{x}+\mathbf{s}c_0 t|\right)$ {are}
$\mathcal{O}\left(1/t\right)$ for $t\gg 1$ uniformly for $\vx\in\Omega$, $\vs\in\mathbb{S}^2$ (due to \eqref{rho-est} with $r=1$). {Hence, the last term in the integrand is $\mathcal{O}\left(1/t^2\right)$ and employing \eqref{eq:near_id_estim2} and \eqref{eq:Y_dt} (both for $r=1$), we observe that the first and the second terms are $\mathcal{O}\left(1/t^3\right)$ and $\mathcal{O}\left(1/t^2\right)$, respectively. Therefore, the whole integrand is $\mathcal{O}\left(1/t^2\right)$.} The estimate \eqref{eq:v_dec2show_2} hence follows. 

\begin{itemize}
\item{\bf Conclusion of the proof.}
\end{itemize}
Given a pair of functions $\left(w_0,w_1\right)$ and constants $\omega$, $c_0$, $\rho_0>0$, $\rho_1>\rho_0$, we say that 
\begin{equation}\label{eq:B-set}
\left(w_0,w_1\right)\in\mathcal{B},
\end{equation}
if we can write
\begin{equation}\label{eq:w0_def}
w_0(\vx)=A_w(|\vx|)Y_w\left(\frac{\vx}{|\vx|}\right)+V_w^{\left(1\right)}(\vx),
\end{equation}
\begin{equation}\label{eq:w1_def}
w_1(\vx)=-c_0 A_w^\prime(|\vx|)Y_w\left(\frac{\vx}{|\vx|}\right)+V_w^{\left(2\right)}(\vx)
\end{equation}
for some functions 
\begin{equation}\label{def-w1}
A_w\in C^1(\mathbb{R}_{+}),\quad
Y_w\in C^1(\mathbb{S}^{d-1}),\quad
V_w^{\left(1\right)}\in C^1(\mathbb{R}^d),\quad
V_w^{\left(2\right)}\in C(\mathbb{R}^d)
\end{equation}
such that
\[
A_w(\left|\vx\right|)=V_w^{(1)}(\vx)=V_w^{(2)}\left(\vx\right)\equiv 0,\quad |\vx|\leq\rho_0, \qquad A_w(\rho)=\dfrac{e^{i\frac{\omega}{c_0}\rho}}{\rho^{\frac{d-1}{2}}},\quad \rho>\rho_1,
\]
and we have
\[
|\vx|^{{\frac{d+3}{2}}}\left(|V_w^{\left(1\right)}(\vx)|+|V_w^{\left(2\right)}(\vx)|+|\nabla V_w^{\left(1\right)}(\vx)|\right)\leq C,\quad \vx\in\mathbb{R}^d,
\]
for some constant $C>0$. 

Let {$Z\left[w_{0},w_{1}\right]\equiv Z_w$ denote} the solution of
the wave equation $\partial_{t}^{2}Z_w\left(\vx,t\right)- \linebreak c_{0}^{2}\Delta Z_w\left(\vx,t\right)=0$
for $\vx\in\mathbb{R}^{d}$, $t>0$, subject to the initial conditions
$Z_w\left(\vx,0\right)=w_{0}\left(\vx\right)$, $\partial_{t}Z_w\left(\vx,0\right)=w_{1}\left(\vx\right)$.
By linearity, we have
\begin{equation}\label{eq:Zw_decomp}
Z_w=Z_{AY}+Z_V,
\end{equation}
where the first term corresponds to the solution produced by the $A_{w}Y_{w}$
terms whereas the second one is due to the $V_w^{\left(1\right)},$
$V_w^{\left(2\right)}$ terms in \eqref{eq:w0_def}, \eqref{eq:w1_def}.

Note that, in the proof of the present lemma, we have already shown the decay 
\[
\left|Z_{AY}\left(\vx,t\right)\right|\leq\frac{C}{\left(1+t^{2}\right)^{1/2}},\hspace{1em}t\geq0,\hspace{1em}\vx\in\Omega,
\]
since the regularity of $A$ and $Y$ in \eqref{def-w1} was sufficient for this decay. 
An analogous time-decay estimate holds for
the $Z_{V}$ term in \eqref{eq:Zw_decomp}, as follows from Lemma \ref{lem:Y0Y1_zero}.
Consequently, we obtain
\begin{equation*}
\left|Z_w\left(\vx,t\right)\right|\leq\frac{C}{\left(1+t^{2}\right)^{1/2}},\hspace{1em}t\geq0,\hspace{1em}\vx\in\Omega.
\end{equation*}

In other words, class \eqref{eq:B-set} consists of the initial conditions with somewhat minimal assumptions for which we can deduce the $\mathcal{O}\left(1/t\right)$ decay of the solution (but not necessarily of its derivatives).

Now we consider $Z\left[v_{0},v_{1}\right]\equiv: Z$. Clearly, we have $\left(v_{0},v_{1}\right)\in\mathcal{B}$ with $V_w^{(1)}=V_w^{(2)}\equiv0$ and thus
\begin{equation}\label{eq:Z_dec}
\left|Z\left(\vx,t\right)\right|\leq\frac{C}{\left(1+t^{2}\right)^{1/2}},\hspace{1em}t\geq0,\hspace{1em}\vx\in\Omega.
\end{equation}
We shall deduce similar results for $\Delta Z$ and $\Delta^2 Z$.
To this effect, we first compute, for $\rho>\rho_{1}$,
\[
A^{\prime}\left(\rho\right)=\left(\frac{i\omega}{c_{0}}-\frac{d-1}{2\rho}\right)\frac{e^{i\frac{\omega}{c_{0}}\rho}}{\rho^{\frac{d-1}{2}}},
\]
\[
A^{\prime\prime}\left(\rho\right)=-\left(\frac{\omega^{2}}{c_{0}^{2}}+i\frac{\omega}{c_{0}}\frac{d-1}{\rho}\right)\frac{e^{i\frac{\omega}{c_{0}}\rho}}{\rho^{\frac{d-1}{2}}}+\mathcal{O}\left(\frac{1}{\rho^{\frac{d+3}{2}}}\right),
\]
\[
A^{\prime\prime\prime}\left(\rho\right)=-\frac{\omega^{2}}{c_{0}^{2}}\left(i\frac{\omega}{c_{0}}-\frac{3\left(d-1\right)}{2\rho}\right)\frac{e^{i\frac{\omega}{c_{0}}\rho}}{\rho^{\frac{d-1}{2}}}+\mathcal{O}\left(\frac{1}{\rho^{\frac{d+3}{2}}}\right).
\]
Therefore, we have for $\left|\vx\right|>\rho_{1}$ (with a computation in polar/spherical coordinates):

\begin{align}
\Delta v_{0}\left(\vx\right) & =A^{\prime\prime}\left(\left|\vx\right|\right)Y\left(\frac{\vx}{\left|\vx\right|}\right)+\frac{d-1}{\left|\vx\right|}A^{\prime}\left(\left|\vx\right|\right)Y\left(\frac{\vx}{\left|\vx\right|}\right)+\frac{1}{\left|\vx\right|^2}A\left(\left|\vx\right|\right)\Delta_{\mathbb{S}^{d-1}}Y\left(\frac{\vx}{\left|\vx\right|}\right)\label{eq:Lapl_w0}\\
 & =-\frac{\omega^{2}}{c_{0}^{2}}\frac{e^{i\frac{\omega}{c_{0}}\left|\vx\right|}}{\left|\vx\right|^{\frac{d-1}{2}}}Y\left(\frac{\vx}{\left|\vx\right|}\right)+\mathcal{O}\left(\frac{1}{\left|\vx\right|^{\frac{d+3}{2}}}\right),\nonumber 
\end{align}
\begin{align}
-\frac{1}{c_{0}}\Delta v_{1}\left(\vx\right) & =A^{\prime\prime\prime}\left(\left|\vx\right|\right)Y\left(\frac{\vx}{\left|\vx\right|}\right)+\frac{d-1}{\left|\vx\right|}A^{\prime\prime}\left(\left|\vx\right|\right)Y\left(\frac{\vx}{\left|\vx\right|}\right)+\frac{1}{\left|\vx\right|^2}A^{\prime}\left(\left|\vx\right|\right)\Delta_{\mathbb{S}^{d-1}}Y\left(\frac{\vx}{\left|\vx\right|}\right)\label{eq:Lapl_w1}\\
 & =-\frac{\omega^{2}}{c_{0}^{2}}\left(i\frac{\omega}{c_{0}}-\frac{d-1}{2\left|\vx\right|}\right)\frac{e^{i\frac{\omega}{c_{0}}\left|\vx\right|}}{\left|\vx\right|^{\frac{d-1}{2}}}Y\left(\frac{\vx}{\left|\vx\right|}\right)+\mathcal{O}\left(\frac{1}{\left|\vx\right|^{\frac{d+3}{2}}}\right),\nonumber 
\end{align}
where $\Delta_{\mathbb{S}^{d-1}}$ is the Laplace-Beltrami operator
on the $\left(d-1\right)$-dimensional unit sphere.\\
Note that, due to cancellation, we do not have any $\mathcal{O}\left(1/\left|\vx\right|^{\frac{d+1}{2}}\right)$
terms in (\ref{eq:Lapl_w0}). 
Also, observe
that we can rewrite these quantities as
\[
\Delta v_{0}\left(\vx\right)=-\left(\frac{\omega}{c_{0}}\right)^{2}A\left(\left|\vx\right|\right)Y\left(\frac{\vx}{\left|\vx\right|}\right)+\mathcal{O}\left(\frac{1}{\left|\vx\right|^{\frac{d+3}{2}}}\right),
\]
\[
\Delta v_{1}\left(\vx\right)=\left(\frac{\omega}{c_{0}}\right)^{2}c_{0}A^{\prime}\left(\left|\vx\right|\right)Y\left(\frac{\vx}{\left|\vx\right|}\right)+\mathcal{O}\left(\frac{1}{\left|\vx\right|^{\frac{d+3}{2}}}\right),
\]
respectively. This, together with the regularity assumptions of this lemma, yields that $\left(\Delta v_{0},\Delta v_{1}\right)\in\mathcal{B}$ (with
$Y_{w}=-\left(\omega/c_0\right)^2 Y$ and some $V_w^{(1)}$, $V_w^{(2)}$ which are no longer identically zero). Since $\Delta Z \left[v_{0}, v_{1}\right] = Z\left[\Delta v_{0},\Delta v_{1}\right]$ {(implied by {$Z\in C^6\left(\mR^d\times\mR_+\right)$ due to $v_0\in C^7\left(\mR^d\right)$ and $v_1\in C^6\left(\mR^d\right)$, as discussed in an analogous setting in the } 
Conclusion of the proof of Lemma~\ref{lem:gen_dec})}, we obtain
\begin{equation}\label{eq:Z_Lapl_dec}
\left|\Delta Z\left(\vx,t\right)\right|\leq\frac{C}{\left(1+t^{2}\right)^{1/2}},\hspace{1em}t\geq0,\hspace{1em}\vx\in\Omega.
\end{equation}
By iteration we also deduce that $\left(\Delta^{2}v_{0},\Delta^{2}v_{1}\right)\in\mathcal{B}$ and, {{using again the regularity} $Z\in C^6\left(\mR^d\times\mR_+\right)$, {that}}
$\Delta^{2}Z=Z\left[\Delta^{2}v_{0},\Delta^{2}v_{1}\right]$.
{Consequently,}
\begin{equation}
\left|\Delta^{2}Z\left(\vx,t\right)\right|=\frac{1}{c_{0}^{2}}\left|\partial_{t}^{2}\Delta Z\left(\vx,t\right)\right|=\frac{1}{c_{0}^{4}}\left|\partial_{t}^{4}Z\left(\vx,t\right)\right|\leq\frac{C}{\left(1+t^{2}\right)^{1/2}},\hspace{1em}t\geq0,\hspace{1em}\vx\in\Omega.\label{eq:Z_Lapl2_dt4_dec}
\end{equation}
Employing the interpolation argument as at the final stage of the proof of Lemma~\ref{lem:gen_dec}, we use (\ref{eq:Z_dec}),  \eqref{eq:Z_Lapl_dec}, and (\ref{eq:Z_Lapl2_dt4_dec})
to obtain the bounds for all the intermediate derivatives in space
and time, thus deducing \eqref{eq:sol_der_bnds}.
\qed

\section{Conclusions and outlook}
\label{sec:conc}
{Motivated by the development of time-domain methods for the 
numerical solution of Helmholtz problems with variable
           coefficients, we have established {a} rigorous proof of the
           LAP under physically reasonable assumptions on the
           coefficients of the wave equation and the source
           term. Under an appropriate modification, the LAP has been extended to 1D.
Moreover, since the speed of stabilisation towards the harmonic regime
           is a deciding factor for using time-domain approaches in
           practice, we have also provided rigorous estimates for this
           convergence in time. 

Our main focus was on the 1D and 2D cases for which the LAP was
           generally understudied previously. In these cases,
           exponential (for 1D) and algebraic (for 2D) convergence
           rates are generally sharp. In the 3D case, previous works
           on wave equations of similar form and some of our numerical
           experiments (for radial data, see Appendix~\ref{sec:app2}) seem to suggest that our
           algebraic convergence result could be improved to the
           exponential one.

An interesting extension of our results would be to remove the non-trapping assumption on the coefficients. Even though the LAP is still expected to be valid, in this case, the time convergence rate would be much slower. Namely, \cite[Thm. 3]{Shapiro} suggests that the convergence rate $1/\left(1+t\right)$ would be replaced by $1/\left[\log\left(2+t\right)\right]^\gamma$ for some $\gamma>0$.   
}

{
\section{Acknowledgements}
This work was supported by the Austrian Science Fund (FWF) projects
           \href{https://doi.org/10.55776/F65}{10.55776/F65} (AA, SG, IP),
           \href{https://doi.org/10.55776/P33477}{10.55776/P33477} (IP), and
           \href{https://doi.org/10.55776/I3538}{10.55776/I3538} (AA, DP).

We would like to express our gratitude to the anonymous reviewer whose suggestions have helped improving the manuscript.
}


\appendix
\section{}  \label{sec:app}
We collect here some technical estimates needed in the proofs of Sections \ref{sec:LAP_proof}--\ref{sec:proofs}.
\begin{lem}\label{lem:K_estims}
For $\vy\in\Omega$, a bounded domain $\Omega\subset\mathbb{R}^d$, $d\geq2$, and $K$ defined by~\eqref{eq:K_def}, the following asymptotic expansions are valid for $\left|\vx\right|\gg 1$:
\begin{align}\label{eq:K_estim}
K\left(\vx-\vy\right)= &\frac{1}{4\pi}\left(\frac{\omega}{2\pi i
                         c_0}\right)^\frac{d-3}{2}
                         \frac{e^{i\frac{\omega}{c_{0}}\left(\left|\vx\right|-\frac{\vx\cdot\vy}{\left|\vx\right|}\right)}}{\left|\vx\right|^{\frac{d-1}{2}}}\left[1+\frac{1}{\left|\vx\right|}\left(
\left(d-3\right)\left(d-1\right)\frac{ic_0}{8\omega}
                         \right.\right.\\
&\left.\left.+\frac{d-1}{2}\frac{\vx\cdot\vy}{\left|\vx\right|}+\frac{i\omega}{2c_{0}}\frac{\left|\vx\right|^2\left|\vy\right|^{2}-\left(\vx \cdot\vy\right)^{2}}{\left|\vx\right|^2}\right)\right]+\mathcal{O}\left(\frac{1}{\left|\vx\right|^{\frac{d+3}{2}}}\right),\nonumber
\end{align}
\begin{align}\label{eq:K_der_estim}
\partial_{\left|\vx\right|}K\left(\vx-\vy\right)= & \frac{1}{4\pi}\left(\frac{\omega}{2\pi i c_0}\right)^\frac{d-3}{2}\frac{e^{i\frac{\omega}{c_{0}}\left(\left|\vx\right|-\frac{\vx\cdot\vy}{\left|\vx\right|}\right)}}{\left|\vx\right|^{\frac{d-1}{2}}}\left[i\frac{\omega}{c_{0}}-\frac{1}{\left|\vx\right|}\left(\frac{\omega^{2}}{2c_{0}^{2}}\frac{\left|\vx\right|^2\left|\vy\right|^{2}-\left(\vx \cdot\vy\right)^{2}}{\left|\vx\right|^2}\right.\right. \\
                                                  &\left.\left.
-\frac{d-1}{2}\frac{i\omega}{c_0}
\frac{\vx\cdot\vy}{\left|\vx\right|}+\frac{d^2-1}{8}\right)\right]+\mathcal{O}\left(\frac{1}{\left|\vx\right|^{\frac{d+3}{2}}}\right)\nonumber.
\end{align}
\end{lem}
\begin{proof}
Setting
\begin{equation}\label{eq:K_tilde}
\widetilde{K}\left(\vx-\vy\right):=\frac{1}{\left|\vx-\vy\right|^{\frac{d-2}{2}}}H_{\frac{d-2}{2}}^{\left(1\right)}\left(\frac{\omega}{c_{0}}\left|\vx-\vy\right|\right),
\end{equation}
we have
\begin{align}\label{eq:K_tilde_der}
\partial_{\left|\vx\right|}\widetilde{K}\left(\vx-\vy\right)=&\frac{\vx}{\left|\vx\right|}\cdot\nabla\widetilde{K}\left(\vx-\vy\right)=\frac{\left|\vx\right|^{2}-\vx\cdot\vy}{\left|\vx\right|}\frac{1}{\left|\vx-\vy\right|^{\frac{d}{2}}}\left[\frac{\omega}{c_{0}}\left(H_{\frac{d-2}{2}}^{\left(1\right)}\right)^{\prime}\left(\frac{\omega}{c_{0}}\left|\vx-\vy\right|\right)\right.\\&\left.+\left(1-\frac{d}{2}\right)\frac{1}{\left|\vx-\vy\right|}H_{\frac{d-2}{2}}^{\left(1\right)}\left(\frac{\omega}{c_{0}}\left|\vx-\vy\right|\right)\right]\nonumber.
\end{align}
Using the asymptotic behavior of $H_p^{\left(1\right)}$ for large arguments \cite[Sect. 10.17(i--ii)]{Olver}
	\begin{equation*}
	H_{p}^{\left(1\right)}\left(x\right)=\left(\frac{2}{\pi}\right)^{\frac{1}{2}}e^{i\left(x-\frac{2p+1}{4}\pi\right)}\left(\frac{1}{x^{1/2}}+\frac{i\left(4p^2-1\right)}{8x^{3/2}}\right)+\mathcal{O}\left(\frac{1}{x^{5/2}}\right),\label{eq:Hank_asympt}
	\end{equation*}
	\begin{equation*}
	\frac{d}{dx}H_{p}^{\left(1\right)}\left(x\right)=\left(\frac{2}{\pi}\right)^{\frac{1}{2}}e^{i\left(x-\frac{2p+1}{4}\pi\right)}\left(\frac{i}{x^{1/2}}-\frac{4p^2+3}{8x^{3/2}}\right)+\mathcal{O}\left(\frac{1}{x^{5/2}}\right),\hspace{1em}x\gg1,\label{eq:Hank_der_asympt}
	\end{equation*}
we can write~\eqref{eq:K_tilde} and~\eqref{eq:K_tilde_der}, respectively, as
\begin{align}\label{eq:K_tilde_estim1}
\widetilde{K}\left(\vx-\vy\right)=&\left(\frac{2c_{0}}{\pi\omega}\right)^{1/2}\frac{e^{-i\left(d-1\right)\pi/4}}{\left|\vx-\vy\right|^{\frac{d-1}{2}}}e^{i\frac{\omega}{c_{0}}\left|\vx-\vy\right|}\left[1+\frac{ic_{0}}{8\omega}\left(d-3\right)\left(d-1\right)\frac{1}{\left|\vx-\vy\right|}\right]\\
&+\mathcal{O}\left(\frac{1}{\left|\vx-\vy\right|^{\frac{d+3}{2}}}\right),\nonumber
\end{align}
\begin{align}\label{eq:K_tilde_der_estim1}
\partial_{\left|\vx\right|}\widetilde{K}\left(\vx-\vy\right)=&\left(\frac{2c_{0}}{\pi\omega}\right)^{1/2}\frac{\left|\vx\right|^{2}-\vx\cdot\vy}{\left|\vx\right|}\frac{e^{-i\left(d-1\right)\pi/4}}{\left|\vx-\vy\right|^{\frac{d+1}{2}}}e^{i\frac{\omega}{c_{0}}\left|\vx-\vy\right|}\\
&\times\left[i\frac{\omega}{c_{0}}+\left(1-\frac{d}{2}-\frac{\left(d-2\right)^{2}+3}{8}\right)\frac{1}{\left|\vx-\vy\right|}\right]+\mathcal{O}\left(\frac{1}{\left|\vx-\vy\right|^{\frac{d+5}{2}}}\right).\nonumber
\end{align}
By using the identity \[
\left|\vx-\vy\right|=\left|\vx\right|\left(1-2\frac{\vx\cdot\vy}{\left|\vx\right|^2}+\frac{\left|\vy\right|^2}{\left|\vx\right|^2}\right)^{1/2},
\]
and the Taylor expansion of $\left(1+z\right)^{-\gamma/2}$ with $z:=-2\frac{\vx\cdot\vy}{\left|\vx\right|^2}+\frac{\left|\vy\right|^2}{\left|\vx\right|^2}$ about $z=0$,
we obtain, for $\left|\vx\right|\gg 1$,
\begin{equation*}
\frac{1}{\left|\vx-\vy\right|^{\gamma}}=\frac{1}{\left|\vx\right|^{\gamma}}\left(1+\gamma\frac{\vx\cdot\vy}{\left|\vx\right|^{2}}\right)+\mathcal{O}\left(\frac{1}{\left|\vx\right|^{\gamma+2}}\right).
\end{equation*}
We will use this formula with $\gamma=\frac{d-1}{2}$, $\frac{d+1}{2}$, $\frac{d+3}{2}$.
Moreover, using the Taylor expansions of $\left(1+z_1\right)^{1/2}$ and $\exp\left(z_2\right)$ with $z_1:=-2\frac{\vx\cdot\vy}{\left|\vx\right|^2}+\frac{\left|\vy\right|^2}{\left|\vx\right|^2}$ and $z_2:=i\frac{\omega}{c_{0}}\left|\vx\right|\left[\left(1-\frac{2\vx\cdot\vy}{\left|\vx\right|^{2}}+\frac{\left|\vy\right|^{2}}{\left|\vx\right|^{2}}\right)^{1/2}-1+\frac{\vx\cdot\vy}{\left|\vx\right|^{2}}\right]$ about $z_1=z_2=0$, we obtain
\begin{align*}
e^{i\frac{\omega}{c_{0}}\left|\vx-\vy\right|}&=e^{i\frac{\omega}{c_{0}}\left(\left|\vx\right|-\frac{\vx\cdot\vy}{\left|\vx\right|}\right)}e^{i\frac{\omega}{c_{0}}\left|\vx\right|\left[\left(1-\frac{2\vx\cdot\vy}{\left|\vx\right|^{2}}+\frac{\left|\vy\right|^{2}}{\left|\vx\right|^{2}}\right)^{1/2}-1+\frac{\vx\cdot\vy}{\left|\vx\right|^{2}}\right]}\\&=e^{i\frac{\omega}{c_{0}}\left(\left|\vx\right|-\frac{\vx\cdot\vy}{\left|\vx\right|}\right)}\left(1+i\frac{\omega}{c_{0}}\left|\vx\right|\left[\left(1-\frac{2\vx\cdot\vy}{\left|\vx\right|^{2}}+\frac{\left|\vy\right|^{2}}{\left|\vx\right|^{2}}\right)^{1/2}-1+\frac{\vx\cdot\vy}{\left|\vx\right|^{2}}\right]\right)+\mathcal{O}\left(\frac{1}{\left|\vx\right|^{2}}\right)\\&=e^{i\frac{\omega}{c_{0}}\left(\left|\vx\right|-\frac{\vx\cdot\vy}{\left|\vx\right|}\right)}\left(1+i\frac{\omega}{2c_{0}}\frac{\left|\vx\right|^{2}\left|\vy\right|^{2}-\left(\vx\cdot\vy\right)^{2}}{\left|\vx\right|^{3}}\right)+\mathcal{O}\left(\frac{1}{\left|\vx\right|^{2}}\right).
\end{align*}
Therefore, we get from~\eqref{eq:K_tilde_estim1} and~\eqref{eq:K_tilde_der_estim1}
\begin{align}\label{eq:K_tilde_estim2}
\widetilde{K}\left(\vx-\vy\right)=&\left(\frac{2c_{0}}{\pi\omega}\right)^{1/2}\frac{e^{-i\left(d-1\right)\pi/4}}{\left|\vx\right|^{\frac{d-1}{2}}}e^{i\frac{\omega}{c_{0}}\left(\left|\vx\right|-\frac{\vx\cdot\vy}{\left|\vx\right|}\right)}\left[1+\frac{1}{\left|\vx\right|}\left(i\frac{c_{0}}{\omega}\frac{\left(d-3\right)\left(d-1\right)}{8}\right.\right.\\&\left.\left.+\frac{d-1}{2}\frac{\vx\cdot\vy}{\left|\vx\right|}+i\frac{\omega}{2c_{0}}\frac{\left|\vx\right|^{2}\left|\vy\right|^{2}-\left(\vx\cdot\vy\right)^{2}}{\left|\vx\right|^{2}}\right)\right]+\mathcal{O}\left(\frac{1}{\left|\vx\right|^{\frac{d+3}{2}}}\right),\nonumber
\end{align}
\begin{align}\label{eq:K_tilde_der_estim2}
\partial_{\left|\vx\right|}\widetilde{K}\left(\vx-\vy\right)=&\left(\frac{2c_{0}}{\pi\omega}\right)^{1/2}\frac{e^{-i\left(d-1\right)\pi/4}}{\left|\vx\right|^{\frac{d-1}{2}}}e^{i\frac{\omega}{c_{0}}\left(\left|\vx\right|-\frac{\vx\cdot\vy}{\left|\vx\right|}\right)}\left[i\frac{\omega}{c_{0}}-\frac{1}{\left|\vx\right|}\left(\frac{\omega^{2}}{2c_{0}^{2}}\left(\left|\vy\right|^{2}-\frac{\left(\vx\cdot\vy\right)^2}{\left|\vx\right|^2}\right)\right.\right.\\&\left.\left.+i\frac{\omega}{2c_{0}}\left(1-d\right)\frac{\vx\cdot\vy}{\left|\vx\right|}+\frac{d^{2}-1}{8}\right)\right]+\mathcal{O}\left(\frac{1}{\left|\vx\right|^{\frac{d+3}{2}}}\right).\nonumber
\end{align}
Since $K\left(\vx\right)=\frac{i}{4}\left(\frac{\omega}{2\pi
    c_0}\right)^\frac{d-2}{2}\widetilde{K}\left(\vx\right)$, 
estimates~\eqref{eq:K_tilde_estim2} and~\eqref{eq:K_tilde_der_estim2}
imply~\eqref{eq:K_estim} and~\eqref{eq:K_der_estim}.
\end{proof}

\begin{lem}\label{lem:app_alg_int_estim}
Let $a$, $b>0$, and define
\begin{equation}\label{eq:app_J_def}
J:=\int_0^1 \frac{dx}{\left(x^2+a^2\right)^{b/2}}\geq 0.
\end{equation}
Then, we have
\begin{equation}\label{eq:app_J_estim}
J\leq \begin{cases}
C_{1,b},& b<1,\\
\log\left(\dfrac{1}{a}+\sqrt{1+\dfrac{1}{a^2}}\right),& b=1,\\
C_{2,b}\dfrac{1}{a^{b-1}},& b>1,
\end{cases} 
\end{equation}
where $C_{1,b}:=\dfrac{1}{1-b}$, $C_{2,b}:=\displaystyle{\int_0^\infty}\dfrac{dx}{\left(1+x^2\right)^{b/2}}$.
\end{lem}
\begin{proof}
After the change of variable $x\mapsto z:=x/a$, we have
\[
J=\frac{1}{a^{b-1}}\int_0^{1/a}\frac{dz}{\left(z^2+1\right)^{b/2}}.
\]
Using \[
\int_0^{1/a}\frac{dz}{\left(z^2+1\right)^{b/2}}\leq\int_0^{1/a}\frac{dz}{z^b}=\frac{a^{b-1}}{1-b}
\]
when $b<1$, 
\[
\int_0^{1/a}\frac{dz}{\left(z^2+1\right)^{b/2}}\leq\int_0^{\infty}\frac{dz}{\left(z^2+1\right)^{b/2}}=:C_{2,b}
\]
when $b>1$,
{and
  \[
    J=\int_0^{1/a}\frac{dz}{\left(z^2+1\right)^{1/2}}
    =\log\left(\dfrac{1}{a}+\sqrt{1+\dfrac{1}{a^2}}\right)
    \]
when $b=1$,}
the estimate \eqref{eq:app_J_estim} follows immediately.
\end{proof}

\begin{lem}\label{lem:app_osc_int_estim}
Let $a>0$. For $t\gg 1$, we have the following estimate
\begin{equation}\label{eq:app_osc_int_estim}
\int_0^a \frac{e^{-ixt}}{x^{1/2}}dx=\left(\frac{\pi}{2t}\right)^{1/2}\left(1-i\right)+\mathcal{O}\left(\frac{1}{t}\right). 
\end{equation}
\end{lem}
\begin{proof}
Making a change of variable $x\mapsto z\left(x\right):=\sqrt{xt}$, we have
\begin{equation}\label{eq:I_def}
I\left(t\right):=\int_0^a \frac{e^{-ixt}}{x^{1/2}}dx=\frac{2}{\sqrt{t}}\int_{0}^{\sqrt{at}}e^{-iz^{2}}dz.
\end{equation}
Since the integrand in \eqref{eq:I_def} is analytic in $z$, we can invoke the
Cauchy theorem to deform the integration contour in the complex
plane. In particular, we choose the {new} contour {$\Gamma_1\cup\Gamma_2$} that consists of
two
parts: the straight line segment $\Gamma_1$
and the circular arc $\Gamma_2$. {This contour is traversed counterclockwise with $\Gamma_1$ and $\Gamma_2$} defined, respectively, as
\[
  \begin{split}
&\Gamma_1:=\left\{
  z\in\mathbb{C}:\,z=re^{-i\pi/4},\,r\in\left(0,R\right)\right\},\\
&\Gamma_2:=\left\{
  z\in\mathbb{C}:\,z=Re^{i\phi},\,\phi\in\left(-\frac{\pi}{4},0\right)\right\},
\end{split}
\]
where for the sake of brevity, we have set $R:=\sqrt{at}$.
In other words, we can write
\begin{align}\label{eq:I_estim}
  I\left(t\right)=&
                    \frac{2}{\sqrt{t}}\left(\int_{\Gamma_1}e^{-iz^{2}}dz+\int_{\Gamma_2}e^{-iz^{2}}dz\right)\\
=&\frac{2\left(1-i\right)}{\sqrt{2t}}\int_{0}^{R}e^{-r^{2}}dr+\frac{2iR}{\sqrt{t}}\int_{-\pi/4}^0 \exp\left(-iR^2 e^{2i\phi}+i\phi\right)d\phi.\nonumber
\end{align}
Note that, for $R=\sqrt{at}\gg 1$, we have
\begin{align}\label{eq:I_term1}
\int_{0}^{R}e^{-r^{2}}dr & =\int_{0}^{\infty}e^{-r^{2}}dr-\int_{R}^{\infty}e^{-r^{2}}dr=\frac{\sqrt{\pi}}{2}-e^{-R^{2}}\int_{0}^{\infty}e^{-r^{2}-2Rr}dr\\
 & =\frac{\sqrt{\pi}}{2}+\mathcal{O}\left(e^{-at}\right),\nonumber
\end{align}
\begin{align}\label{eq:I_term2}
\int_{-\frac{\pi}{4}}^{0}\exp\left(-iR^{2}e^{2i\phi}+i\phi\right)d\phi =& \frac1{2R^2}\int_{-\frac{\pi}{4}}^{0}\Big[\frac{\partial}{\partial\phi}\exp\left(-iR^{2}e^{2i\phi}\right)\Big]e^{-i\phi}d\phi\\
 =&\frac{1}{2R^{2}}\left(e^{-iR^{2}}-\frac{1+i}{\sqrt{2}}e^{-R^{2}}\right)\nonumber\\
& +\frac{i}{2R^{2}}\int_{-\frac{\pi}{4}}^{0}\exp\left(-iR^{2}e^{2i\phi}-i\phi\right)d\phi\nonumber\\
 =&\mathcal{O}\left(\frac{1}{R^{2}}\right)=\mathcal{O}\left(\frac{1}{t}\right). \nonumber
\end{align}
Inserting \eqref{eq:I_term1} and~\eqref{eq:I_term2} into \eqref{eq:I_estim} furnishes the claimed estimate \eqref{eq:app_osc_int_estim}.
\end{proof}


\begin{lem}\label{1D-interpol}
Let $F\in C^2(\mR_+)$ satisfy for some $C_1>0$
\begin{equation}\label{eq:F_Ftt_bnd}
  |F(t)|+|F''(t)|\leq \frac{C_1}{\left(1+t^2\right)^{1/2}},   \hspace{1em}t\geq 0.
\end{equation}
Then
\begin{equation}\label{eq:Ft_bnd_intrpl}
  |F'(t)|\leq \frac{C_2}{\left(1+t^2\right)^{1/2}},   \hspace{1em}t\geq 0
\end{equation}
holds true with some $C_2>0$.
\end{lem}
\begin{proof}
For all $N\in\mN_0$, we have
$$
  \|F\|_{L^\infty(N,N+1)} + \|F''\|_{L^\infty(N,N+1)}\le \frac{C_1}{(1+N^2)^{1/2}}.
$$
By interpolation (see \cite[Prop. 2.2]{BM}), we deduce the same estimate also for $F'$, with some generic constant $\tilde C_1>0$ independent of $N$. Hence
$$
  |F'(t)| \le \frac{\tilde C_1}{(1+\lfloor t\rfloor^2)^{1/2}},\hspace{1em}t\geq 0,
$$
where $\lfloor t\rfloor$ is the floor function. Using
$$
1+\left\lfloor t\right\rfloor ^{2}=\frac{1+\left\lfloor t\right\rfloor ^{2}}{1+t^{2}}\left(1+t^{2}\right)\geq\frac{1+\left(t-1\right)^{2}\chi_{\left(1,\infty\right)}\left(t\right)}{1+t^{2}}\left(1+t^{2}\right)\geq C_{0}\left(1+t^{2}\right),
$$
with some constant $C_0>0$, we deduce \eqref{eq:Ft_bnd_intrpl} with $C_2=\tilde C_1/C_0^{1/2}$. 
\end{proof}


\section{}  \label{sec:app2}
We test our results numerically on an example where the material
    parameters $\alpha$, $\beta$ and the source term $F$ are radially
    symmetric. Namely, {for $r:=\vert\vx\vert$, $\vx\in\mathbb R^d$,} we choose 
\begin{equation}
\alpha\left(r\right):=2\chi_{\left[0,2\right)}\left(r\right)+\frac{1}{2}\chi_{\left[2,4\right)}\left(r\right)\left(3+\cos\left(\frac{\pi}{2}\left(r-2\right)\right)\right)+\chi_{\left[4,\infty\right)}\left(r\right),\label{eq:alph_num}
\end{equation}
\begin{equation}
\beta\left(r\right):=1+\chi_{\left(3,7\right)}\left(r\right)\left(1+\cos\left(\frac{\pi}{2}\left(r-5\right)\right)\right),\label{eq:bet_num}
\end{equation}
\begin{equation}
F\left(r\right):=10\chi_{\left(0,8/3\right)}\left(r\right)\left(1+\cos\left(\pi\left(\frac{3}{4}r-1\right)\right)\right),\label{eq:F_num}
\end{equation}
where $\chi$ denotes the characteristic function. Observe that, for this choice of $\alpha$ and $\beta$, the background medium parameters are $\alpha_0=\beta_0=1$. We illustrate the functions in~\eqref{eq:alph_num}, \eqref{eq:bet_num}, and~\eqref{eq:F_num} in Figure~\ref{fig:params}.

We fix the dimension $d\in\left\{1,2,3\right\}$ and the frequency
    $\omega=\pi/4$, and let
$U(\vx)$ and $u(\vx,t)$ be the solutions to
    problems~\eqref{eq:U_Helm} and~\eqref{eq:wave_LAP}, respectively.

As in Table~\ref{table:summary}, we define 
\[
              u^{\text{\sc diff}}(\vx,t):=\begin{cases}
              u(\vx,t)-e^{-i\omega t}U(\vx)\qquad&\text{if
                                                          $d=2,3$},\\
              u(\vx,t)-e^{-i\omega t}U(\vx)-U_\infty \qquad&\text{if
                $d=1$},
              \end{cases}
\]       
where $U_\infty$ is the constant in~\eqref{eq:U_infty_expl}.

We consider the bounded domain $B_{R_0}:=\{\vx\in
              \mathbb R^d:\, \vert\vx\vert<R_0\}$ with $R_{0}=5$, and set             
\begin{equation}\label{eq:num_err_def}                                                           
\mathcal{E}\left(t\right):=\left(\Vert u^{\text{\sc diff}}(\cdot,t)\Vert_{L^2\left(B_{R_0}\right)}^{2}+\Vert \partial_t u^{\text{\sc diff}}(\cdot,t)\Vert_{L^2\left(B_{R_0}\right)}^{2}\right)^{1/2}.                            
\end{equation}


Exploiting the radial symmetry, we rewrite
               problems~\eqref{eq:U_Helm} and~\eqref{eq:wave_LAP} in
               the $(r,t)$-variables and we solve them numerically on
               the domain $\left(0,R\right)\times\left(0,T\right)$
               with $R=120$ and $T=240$.
               For the time-dependent wave problem, we use finite differences in space on a uniform grid of size $6\cdot10^{-2}$,
and the Leapfrog method in time on a uniform grid of size
               $1.33\cdot 10^{-2}$. We solve the Helmholtz problem on the same
               spatial grid.

The following second-order radiation conditions have been used at
$r=R$: 
\begin{equation}
\partial_{t}u\left(R,t\right)=-\sqrt{\frac{\alpha_{0}}{\beta_{0}}}\left[\partial_{r}u\left(R,t\right)+\frac{1}{R}\frac{1-\delta_{d,1}}{1+\delta_{d,2}}u\left(R,t\right)\right],\hspace{1em}t\geq0,\label{eq:num_u_rad_BC}
\end{equation}
\begin{equation}
U^{\prime}\left(R\right)=\left[i\omega\sqrt{\frac{\beta_{0}}{\alpha_{0}}}-\frac{1}{R}\frac{1-\delta_{d,1}}{1+\delta_{d,2}}\right]U\left(R\right),\label{eq:num_U_rad_BC}
\end{equation}
where $\delta$ denotes the Kronecker delta symbol (e.g., see \cite[eq. (1.27)]{engquist1977} and  \cite[eq. (7.10)]{grote1995} for $d=2$ and $d=3$, respectively). Note that these
radiation conditions are exact in case $d=1$.

Central finite differences stencils were used to approximate first-order derivatives
and ghost points were added at the boundaries $r=0$ and $r=R$. In
doing so, the following relations were instrumental 
\[
\partial_{t}^{2}u\left(0,t\right)-\frac{d\alpha\left(0\right)}{\beta\left(0\right)}\partial_{r}^{2}u\left(0,t\right)=e^{-i\omega t} F\left(0\right),
\]
\[
-\omega^{2}U\left(0\right)-\frac{d\alpha\left(0\right)}{\beta\left(0\right)}U^{\prime\prime}\left(0\right)=F\left(0\right).
\]
These equations are obtained by passing to the limit $r\to 0$
in the equations in~\eqref{eq:U_Helm} and \eqref{eq:wave_LAP}, and
using the boundary conditions at $r=0$: $\partial_r u\left(0,t\right)=0$, $U^{\prime}\left(0\right)=0$.

The quantity $\mathcal{E}\left(t\right)$ defined by~\eqref{eq:num_err_def} was computed from the numerical solution of ~\eqref{eq:U_Helm} and~\eqref{eq:wave_LAP}, and is shown in Figure~\ref{fig:err_d1d2d3}--\ref{fig:err_d2_loglog}. In particular, in Figure~\ref{fig:err_d1d2d3}, we observe a much faster decay in time for $d=1$ and $d=3$ than for $d=2$. The semilogarithmic plot in Figure~\ref{fig:err_d1d3_semilog} shows this decay to be exponential (up to the saturation due to numerical errors for small quantities at large times). A linear region for large times is observed in logarithmic plot in Figure~\ref{fig:err_d2_loglog} for $d=2$. As a comparison, we have plotted in black a line of slope $-1$. This clearly illustrates an algebraic convergence, corroborating the sharpness of the decay estimate in Theorem \ref{thm:LAP} for this case.







\begin{figure}

\centering
\begin{subfigure}[b]{0.49\textwidth}
\includegraphics[width=1.1\textwidth]{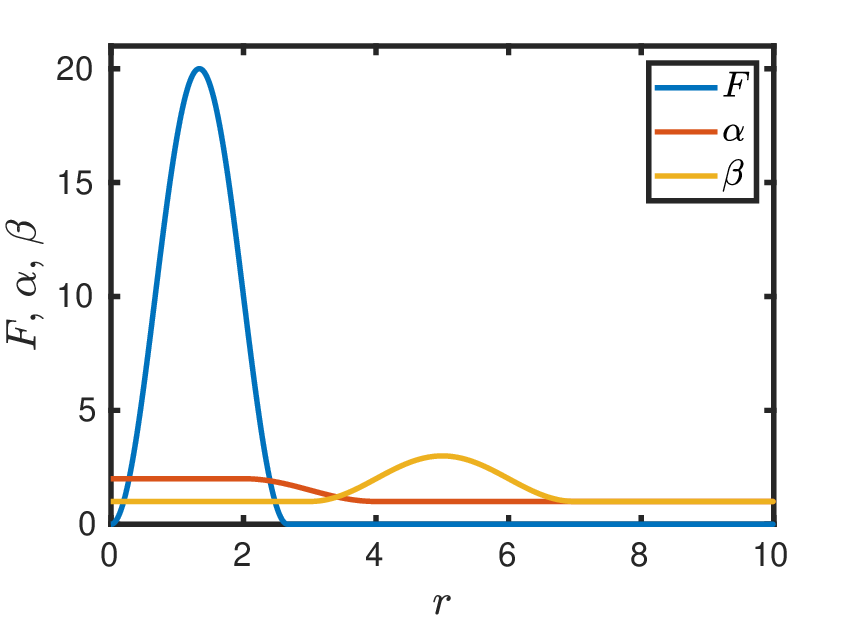}
\caption{Source $F$ and parameters $\alpha$ and $\beta$}
\label{fig:params}
\end{subfigure}
\hfill
\begin{subfigure}[b]{0.49\textwidth}
\includegraphics[width=1.1\textwidth]{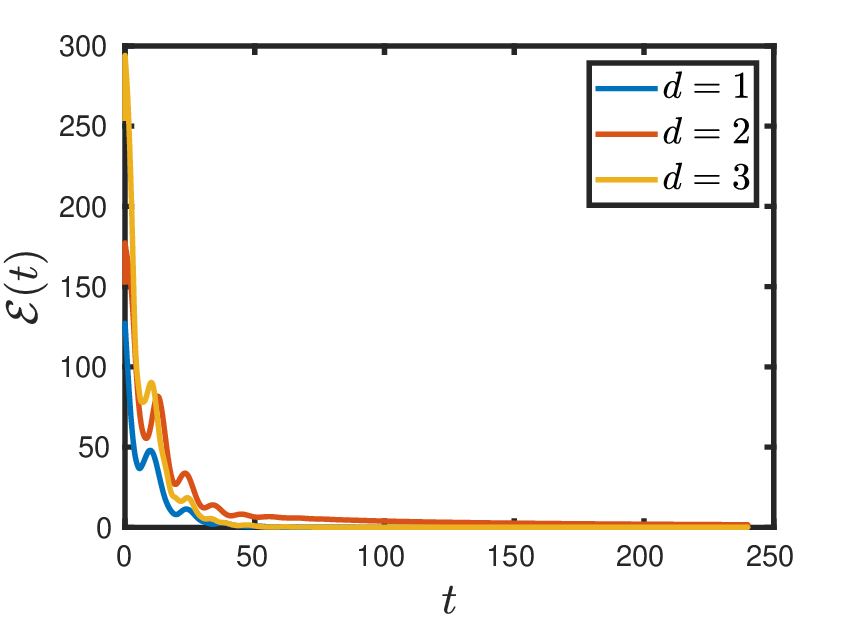}
\caption{$\mathcal{E}\left(t\right)$ for $d\in\{1,2,3\}$}
\label{fig:err_d1d2d3}
\end{subfigure}
\hfill
\begin{subfigure}[b]{0.49\textwidth}
\includegraphics[width=1.1\textwidth]{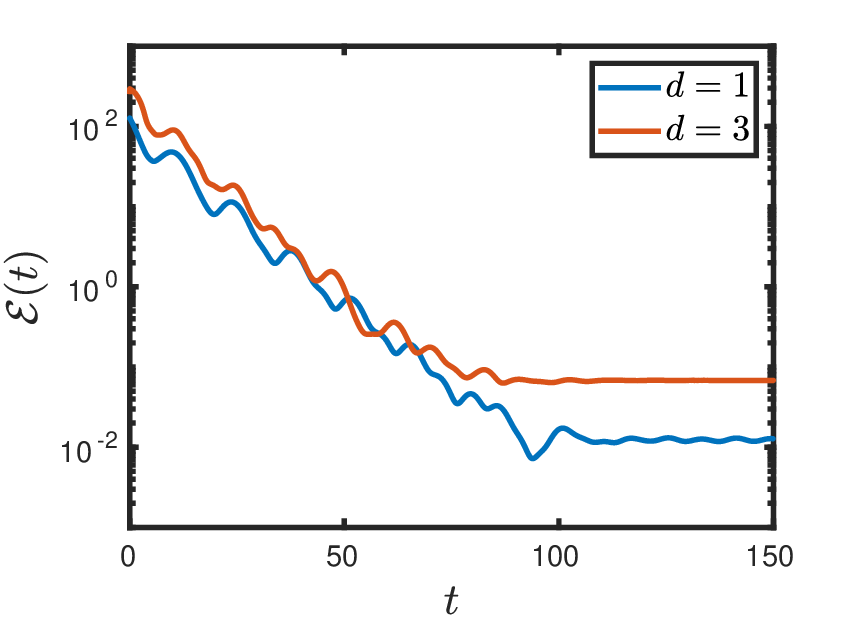}
\caption{$\mathcal{E}\left(t\right)$ for $d\in\{1,3\}$ in semilog scale}
\label{fig:err_d1d3_semilog}
\end{subfigure}
\hfill
\begin{subfigure}[b]{0.49\textwidth}
\includegraphics[width=1.02\textwidth]{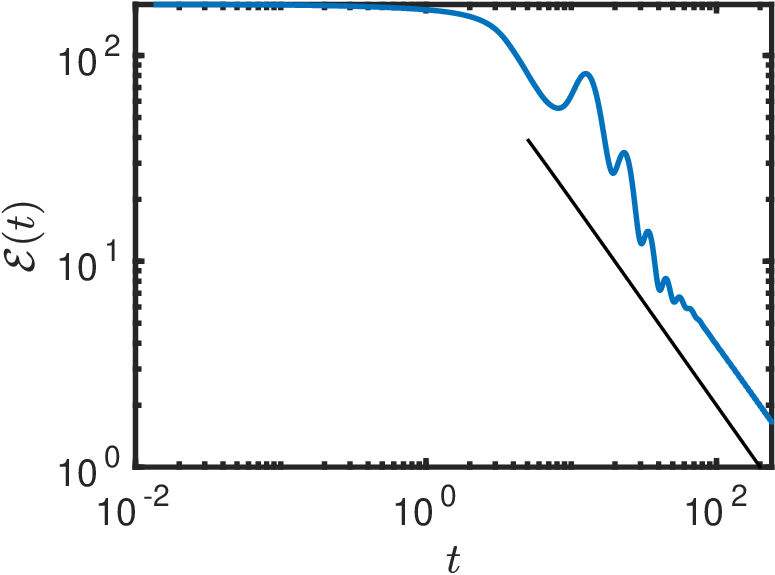}
\caption{$\mathcal{E}\left(t\right)$ for $d=2$ in log scale}
\label{fig:err_d2_loglog}
\end{subfigure}

\caption{Large-time convergence for the radially symmetric example
  with the data as in~\eqref{eq:alph_num}--\eqref{eq:F_num}.}
\label{fig:radial}

\end{figure}               

%
%
%

\bibliographystyle{abbrv}
\bibliography{limampl_vs_decay}

\end{document}